\documentclass[11pt,letterpaper]{amsart}
\usepackage{amsmath}
\usepackage{amsfonts}
\usepackage{amssymb}
\usepackage{amsthm}
\usepackage{indentfirst}
\usepackage{enumerate}
\usepackage{graphicx}
\usepackage{xcolor}
\graphicspath{ {./Figures/} } 
\usepackage{hyperref}

\calclayout

\title[Quantitative stability of Gel'fand's inverse boundary problem]{Quantitative stability of Gel'fand's inverse boundary problem}

\author[D. Burago]{Dmitri Burago}                                                          
\address{Dmitri Burago: Pennsylvania State University,                          
Department of Mathematics, University Park, PA 16802, USA}                      
\email{burago@math.psu.edu}                                                     
                                                                                
\author[S. Ivanov]{Sergei Ivanov}
\address{Sergei Ivanov:
St.Petersburg Department of Steklov Mathematical Institute,
Russian Academy of Sciences,
Fontanka 27, St.Petersburg 191023, Russia}
\email{svivanov@pdmi.ras.ru}

\author[M. Lassas]{Matti Lassas}
\address{Matti Lassas: Department of Mathematics and Statistics, University of Helsinki, FI-00014 Helsinki, Finland}
\email{matti.lassas@helsinki.fi}

\author[J. Lu]{Jinpeng Lu}
\address{Jinpeng Lu: Department of Mathematics and Statistics, University of Helsinki, FI-00014 Helsinki, Finland} 
\email{jinpeng.lu@helsinki.fi}

\date{}

\numberwithin{equation}{section}

\theoremstyle{definition}
\newtheorem{deferror}{Definition}[section]
\theoremstyle{plain}

\newtheorem{stability}{Theorem}
\newtheorem{Cor1}[stability]{Theorem}
\newtheorem{uc-intro}[stability]{Theorem}

\theoremstyle{definition}
\newtheorem{Def-CATradius}{Definition}[section]
\theoremstyle{plain}

\newtheorem{global1}[Def-CATradius]{Theorem}
\newtheorem{main1}{Theorem}[section]
\newtheorem{extensionmetric}[main1]{Lemma}
\newtheorem{extension}[main1]{Lemma}

\newtheorem{distances}[main1]{Lemma}

\newtheorem{sublemmainitial}{Sublemma}
\newtheorem{sublemmainitial2}[sublemmainitial]{Sublemma}
\newtheorem{sublemma0}[sublemmainitial]{Sublemma}
\newtheorem{sublemma1}[sublemmainitial]{Sublemma}
\newtheorem{sublemma2}[sublemmainitial]{Sublemma}

\newtheorem{initial}[main1]{Corollary}
\newtheorem{area}[main1]{Proposition}
\newtheorem{wholedomain}[main1]{Proposition}
\newtheorem{u0}{Lemma}[section]
\newtheorem{smallinitial}[u0]{Lemma}
\newtheorem{projection}[u0]{Proposition}

\newtheorem{measureerror}[u0]{Proposition}
\newtheorem{volume}{Lemma}[section]
\newtheorem{coordinate}[volume]{Lemma}
\newtheorem{volumebeta}[volume]{Lemma}
\newtheorem{approximation}[volume]{Proposition}
\newtheorem{2007}[volume]{Theorem}

\newtheorem{dhsC21}[main1]{Lemma}

\theoremstyle{definition}
\newtheorem{remark0}{Remark}
\newtheorem{remark}[remark0]{Remark}

\newtheorem{definition-Omega}[Def-CATradius]{Definition}
\newtheorem{definition}[main1]{Definition}
\newtheorem{def-section5}[volume]{Definition}
\theoremstyle{plain}

\newtheorem{riccati}{Lemma}[section]
\newtheorem{dhs}[riccati]{Lemma}
\newtheorem{dd}[riccati]{Lemma}
\newtheorem{mindistance}[riccati]{Lemma}
\newtheorem{geodiff}[riccati]{Lemma}
\newtheorem{areaLipschitz}[riccati]{Lemma}
\newtheorem{closeness-curve}[riccati]{Lemma}
\newtheorem{CATradius}[riccati]{Lemma}

\newcommand{\vol}{\textrm{vol}}


\begin{document}
\maketitle

\vspace*{-0.5cm}
\begin{center}
\emph{Dedicated to the memory of Yaroslav Kurylev}
\end{center}

\begin{abstract}
In Gel'fand's inverse problem, one aims to determine the topology, differential structure and Riemannian metric of a compact manifold $M$ with boundary from the knowledge of the boundary $\partial M,$ the Neumann eigenvalues $\lambda_j$ and the boundary values of the eigenfunctions $\varphi_j|_{\partial M}$. We show that this problem has a stable solution with quantitative stability estimates in a class of manifolds with bounded geometry. More precisely, we show that finitely many eigenvalues and the boundary values of corresponding eigenfunctions, known up to small errors, determine a metric space that is close to the manifold in the Gromov-Hausdorff sense. We provide an algorithm to construct this metric space. This result is based on  an explicit estimate on the stability of the unique continuation for the wave operator.  
\end{abstract}

\section{Introduction} \label{section-intro}

Gel'fand's inverse problem, formulated by I. Gel'fand in \cite{G}, concerns finding the topology, differential structure and Riemannian metric of a compact manifold with boundary from the spectral data for the Neumann Laplacian on the boundary,
that is, the Neumann eigenvalues and the boundary values of the corresponding eigenfunctions. 
The problem is closely related to an inverse problem for the wave equation that can be solved using the boundary control method developed by  Belishev \cite{Bel1} on domains of $\mathbb R^n$. 
The uniqueness of Gel'fand's inverse problem on manifolds was proved in 1992
by Belishev and Kurylev  in \cite{BK}, see also \cite{AKKLT,Bel2,Bel3,Caday,KrKL,KOP}, in the form of an inverse spectral problem: the geometry of a compact Riemannian manifold with boundary is uniquely determined by the boundary spectral data for the Neumann Laplacian. 

On a given domain of the Euclidean space, Gel'fand's problem was reduced in \cite{NSU} to inverse coefficient problems for elliptic
equations which were solved in \cite{Astala,Nachman1,Nachman2,SyUl}, see also \cite{DosSantos,Guillarmou,Isozaki,KenigSalo,KSU,U}, and the stability of the solutions of these problems has been studied in \cite{A,AlS,SyUl2}.    
Gel'fand's inverse problem is ill-posed in the sense of Hadamard, as one can make large changes to the geometry of the interior without affecting the boundary spectral data much. One approach of stabilizing the inverse problem is to study the conditional stability by assuming \emph{a priori} knowledge of the desired quantities, for instance higher regularity of coefficients \cite{A}, and higher regularity of Riemannian metrics if they are close to Euclidean \cite{SU}. For a general Riemannian manifold, it is natural to impose \emph{a priori} bounds on geometric parameters such as the diameter, injectivity radius and sectional curvature. An abstract continuity result for the stability of the problem was proved in \cite{AKKLT}, however with no stability estimates, and the related determination of the smooth structure was shown in \cite{FIKLN}.
With additional geometric assumptions, strong stability estimates for this problem can be obtained, e.g. \cite{BD,SU2}, when the metric is close to simple (i.e., with strictly convex boundary and no conjugate points).
One could also consider the inverse interior problem, that is, an inverse problem on closed manifolds analogous to Gel'fand's problem. For the inverse interior problem where the eigenfunctions are measured in a ball of a closed manifold, the unique solvability of the problem was proved in \cite{KrKL} and a quantitative stability estimate for general metric has recently been obtained in \cite{BKL3}. 
A quantitative stability of Gel'fand's inverse problem for manifolds with boundary in the general case was yet unknown. 
The main purpose of the present paper is to provide an answer to this question.

The key result for establishing the uniqueness of Gel'fand's inverse problem was Tataru's unique continuation theorem \cite{T} for the wave operator. Its stability, i.e., quantitative unique continuation, is essential to the stability of the inverse problem. The quantitative unique continuation for the wave operator on Riemannian manifolds, from sets of the form $\Gamma \times [-T,T]$ where $\Gamma$ is the observation region, has been investigated independently in \cite{BKL2,BKL1} for closed manifolds, and in \cite{LL} when $T$ is larger than the diameter of the manifold. Using \cite{BKL2,BKL1}, the authors established a $\log$-$\log$ type of stability estimate \cite{BKL3} for the analogous inverse problem on a closed manifold where spectral data are measured in a ball. However, for manifolds with boundary, the quantitative unique continuation for arbitrary time $T$ is yet unclear, partly due to the lack of smoothness caused by geodesics touching the boundary.
This brought substantial difficulty into propagating the local unique continuation to a global one without losing any domain of dependence.
It turns out that it is beneficial for us to treat these geodesics as distance-minimizing paths in Alexandrov spaces with curvature bounded above, instead of handling them in boundary normal coordinates.
As our main technical task occupying most of Section \ref{section-uc} and \ref{auxiliary}, we focus on geometric issues brought by geodesics near the boundary, and give a fully explicit stability estimate for the unique continuation in the optimal domain of dependence.
Our result also makes it possible to obtain quantitative stability of other inverse problems that are solved using the boundary control method.

We hope our results may have applications in medicine, especially to cancer treatment, more concretely, to imaging necessary for radiation therapy
(e.g. the navigation of cyber knives) and for ultrasound surgery, see e.g. \cite{W}. In these treatment, many thin beams of X-rays or high amplitude ultrasound waves are concentrated in the cancerous tissue and the planning of the treatment requires stable imaging methods. 
A significant potential instance is the focused ultrasound surgery \cite{Tempany}, where a cancerous tissue is destroyed by an excessive
  heat dose generated by focused ultrasound waves. The location where the ultrasound waves are focused
  is determined by the intrinsic Riemannian metric corresponding to the wave speed of acoustic waves, see \cite{Dahl,L}.
   In particular, in an anisotropic medium where the inverse problem is not uniquely solvable in Euclidean coordinates, see \cite{Sylvester}, it is beneficial 
  to do imaging in the same Riemannian structure that determines the wave propagation. 
 The  imaging of the Riemannian metric associated with the wave propagation is an 
  inverse problem for the wave equation, which is equivalent, see  \cite{KKLM}, to
  Gel'fand's inverse problem studied in this paper. 
Numerical methods to solve these problems
have been studied in \cite{HoopOksanen1,HoopOksanen2}.
The quantitative stability of reconstruction from other types of data, e.g. the Dirichlet-to-Neumann map or the source-to-solution map for the wave equation, has not yet been studied; however, in the light of \cite{BKL3,KKLM}, we think a similar stability estimate might be possible.

\smallskip
Let $(M,g)$ be a compact, connected, orientable Riemannian manifold of dimension $n\geqslant 2$ with smooth boundary $\partial M$. We consider the manifold $M$ in the class $\mathcal{M}_n(D,K_1,K_2,i_0,r_0)$ of bounded geometry defined by the bounds on the diameter $\textrm{diam}(M)$, the injectivity radius $\textrm{inj}(M)$, the Riemannian curvature tensor $R_M$ of $M$, and the second fundamental form $S$ of the boundary $\partial M$ embedded in $M$:
$$\textrm{diam}(M)\leqslant D, \quad \textrm{inj}(M)\geqslant i_0,$$
$$\|R_M\|_{C^0}\leqslant K_1^2,\quad \|S\|_{C^0}\leqslant K_1,$$
\begin{equation}\label{boundedgeometry}
\sum_{i=1}^5 \|\nabla^i R_M\|_{C^0}\leqslant K_2,\quad \sum_{i=1}^4 \|\nabla^i S\|_{C^0}\leqslant K_2,
\end{equation}
where $\nabla^i$ denotes the $i$-th covariant derivative on $M$. The injectivity radius for a manifold with boundary is defined in Section \ref{subsection-bounded}. In addition, we impose the lower bound on the following quantity $r_{\textrm{CAT}}(M)$ (Definition \ref{Def-CATradius}):
\begin{equation}\label{bound-CATradius}
r_{\textrm{CAT}}(M)\geqslant r_0\, ,
\end{equation}
where $r_{\textrm{CAT}}(M)$ is defined as the largest number $r$, such that any pair of points with distance less than $r$ is connected by a unique distance-minimizing geodesic (possibly touching the boundary) of $M$. This quantity is known to be positive for a compact Riemannian manifold with smooth boundary. For Riemannian manifolds without boundary, the condition (\ref{bound-CATradius}) is already incorporated in the lower bound for the injectivity radius.

Denote by $\lambda_j$ ($j\geqslant 1$) the $j$-th eigenvalue of the (nonnegative) Laplace-Beltrami operator $-\Delta_g$ on $(M,g)$ with the Neumann boundary condition at $\partial M$, and by $\varphi_j$ an (smooth) eigenfunction with respect to $\lambda_j$. We know that $0=\lambda_1<\lambda_2\leqslant \cdots \leqslant \lambda_j\leqslant \lambda_{j+1}\leqslant \cdots$, and $\lambda_j\to +\infty$ as $j\to +\infty$. Assume the eigenfunctions are orthonormalized with respect to the $L^2$-norm of $M$. In particular $\varphi_1=\textrm{vol}_n(M)^{-1/2}$. The \emph{Neumann boundary spectral data} of $M$ refers to the collection of data 
$$\Big(\partial M,g_{_{\partial M}}, \{\lambda_j,\varphi_j|_{\partial M}\}_{j=1}^{\infty}\Big),$$ 
which consists of the boundary $\partial M$ and its intrinsic metric $g_{_{\partial M}}$, the Neumann eigenvalues and the boundary values of a choice of orthonormalized Neumann eigenfunctions.

\begin{deferror}\label{deferror}
We say a collection of data $\big(\partial M,g_{_{\partial M}},\{\lambda_j^a,\varphi_j^a|_{\partial M}\}_{j=1}^{J}\big)$ is a $\delta$-approximation of the Neumann boundary spectral data of $(M,g)$ (in $C^2$) for some $\delta\geqslant J^{-1}$, if there exists a choice of Neumann boundary spectral data $\{\lambda_j,\varphi_j|_{\partial M}\}_{j=1}^{\infty}$ such that the following three conditions are satisfied for all $j\leqslant \delta^{-1}$:
\begin{enumerate}[(1)]
\item $\lambda_j^a\in [0,\infty)$, $\varphi_j^a |_{\partial M}\in C^2(\partial M)$; 
\item $\big|\sqrt{\lambda_j}-\sqrt{\lambda_j^a} \big|<\delta$;
\item $\big\|\varphi_j - \varphi_j^a \big\|_{C^{0,1}(\partial M)}+ \big\|\nabla_{\partial M}^2 (\varphi_j- \varphi_j^a)|_{\partial M} \big\|_{C^0}< \delta$, where $\nabla_{\partial M}^2$ denotes the second covariant derivative with respect to the induced metric $g_{_{\partial M}}$ on $\partial M$.
\end{enumerate}

Let $M_1,M_2$ be two Riemannian manifolds with isometric boundaries, and let $\Phi:\partial M_1\to \partial M_2$ be the Riemannian isometry (diffeomorphism) between boundaries. We say the Neumann boundary spectral data of $M_1,M_2$ are $\delta$-close, if the pull-back via $\Phi$ of the Neumann boundary spectral data of $M_2$ (or $M_1$) is a $\delta$-approximation of the Neumann boundary spectral data of $M_1$ (or $M_2$).
\end{deferror}

Note that the definition above is coordinate-free. The second covariant derivative of a function is called the Hessian of the function, which is a symmetric (0,2)-tensor. In a local coordinate on $\partial M$, Definition \ref{deferror}(3) translates to $(\varphi_j- \varphi_j^a)|_{\partial M}$ having small $C^2$-norm. A similar definition in $L^2$-norm was seen in \cite{BKL3}.

If finite boundary spectral data $\{\lambda_j,\varphi_j|_{\partial M}\}_{j=1}^{J}$ are known without error, then this set of finite data is a $\delta$-approximation of the Neumann boundary spectral data with $\delta=J^{-1}$ by definition. If we are given a certain choice of Neumann boundary spectral data, then Definition \ref{deferror}(3) is equivalent to the existence of orthogonal matrices acting on eigenfunctions in eigenspaces, such that the condition is satisfied by the given spectral data after applying these matrices.

\medskip
The main purpose of this paper is to prove the following stability estimate for the reconstruction of a manifold from the Neumann boundary spectral data.

\begin{stability}\label{stability}
There exists $\delta_0=\delta_0(n,D,K_1,K_2,i_0,r_0) >0$ such that the following holds. If we are given a $\delta$-approximation of the Neumann boundary spectral data of a Riemannian manifold with boundary $M\in \mathcal{M}_n(D,K_1,K_2,$ $i_0,r_0)$ for $\delta<\delta_0$, then one can construct a finite metric space $X$ directly from the given boundary data such that
$$d_{GH}(M,X)< C_1\Big( \log\big(|\log\delta|\big) \Big)^{-C_2},$$
where $d_{GH}$ denotes the Gromov-Hausdorff distance between metric spaces. The constant $C_1$ depends on $n,D,K_1,K_2,i_0,r_0$, and the constant $C_2$ explicitly depends only on $n$.
\end{stability}

Theorem \ref{stability} implies the stability of Gel'fand's inverse problem.

\begin{Cor1}\label{Cor1}
There exists $\delta_0=\delta_0(n,D,K_1,K_2,i_0,r_0)>0$ such that the following holds. Suppose two Riemannian manifolds $M_1,M_2\in \mathcal{M}_n(D,K_1,K_2,$ $i_0,r_0)$ have isometric boundaries and their Neumann boundary spectral data are $\delta$-close for $\delta<\delta_0$. Then $M_1$ is diffeomorphic to $M_2$, and
$$d_{GH}(M_1,M_2)< C_1 \Big( \log\big(|\log\delta|\big) \Big)^{-C_2}.$$
\end{Cor1}

\begin{remark0}
The dependency of $C_1,\delta_0$ is not explicit. An explicit estimate with dependence additionally on $\vol_n(M),\vol_{n-1}(\partial M)$ can be obtained, but this process results in a third logarithm. More details can be found in Appendix \ref{constants}.

If any explicitness for the results is not of interest, the bounds (\ref{boundedgeometry}) we assumed on the Riemannian curvature tensor and the second fundamental form can be relaxed to bounds on Ricci curvatures of $M,\partial M$ and the mean curvature of $\partial M$, due to Corollary 2 in \cite{KKL2}.
\end{remark0}

We do not know if the $\log$-$\log$ type of estimates above is optimal.
While strong (H\"older-type) stability results \cite{BD,SU,SU2} were known near simple metrics, the stability of the problem is likely weak in the general case, see \cite{KRS,M}.

\smallskip
The key result in proving Theorem \ref{stability} is a uniform stability estimate for the unique continuation in the class of Riemannian manifolds with bounded geometry, and without loss of domain in the domain of dependence.
Let $\Gamma$ be an open subset of the boundary $\partial M$ and $T>0$. The \emph{domain of influence} of the set $\Gamma$ at a time $t\in [0,T]$ is defined as
\begin{equation}\label{def-Mt}
M(\Gamma,t)=\big\{x\in M: d(x,\Gamma)<t \big\},
\end{equation}
where $d$ is the intrinsic distance function of $M$. 
The \emph{double cone of influence} of $\Gamma\times [-T,T]$ is defined as
\begin{equation}\label{def-Kcone}
K(\Gamma,T)=\big\{(x,t)\in M\times [-T,T] : d(x,\Gamma)< T-|t| \big\}.
\end{equation}
Recall Tataru's unique continuation theorem in \cite{T}: if the Cauchy boundary data of a wave $u$ vanish on $\Gamma\times [-T,T]$, i.e.,
$$u|_{\Gamma\times [-T,T]}=0,\quad \partial_{\textbf{n}}u |_{\Gamma \times [-T,T]}=0,$$
then the wave $u$ vanishes in the double cone of influence $K(\Gamma,T)$, and in particular, the initial value $u(\cdot,0)$ vanishes in the domain of influence $M(\Gamma,T)$. 
Note that the domain $K(\Gamma,T)$ (and $M(\Gamma,T)$ for the initial value) in this result is optimal due to finite speed of propagation of waves.
The stability of the unique continuation, i.e., quantitative unique continuation, asks if $u$ is small when the Cauchy boundary data are small.

\begin{uc-intro}\label{uc-intro}
Let $M$ be a compact, orientable Riemannian manifold with smooth boundary $\partial M$, and let $\Gamma$ (possibly $\Gamma=\partial M$) be a connected open subset of $\partial M$ with smooth boundary. Suppose $u\in H^2(M\times[-T,T])$ is a solution of the wave equation $(\partial_t^2-\Delta_g) u(x,t)=0$ with the Neumann boundary condition $\partial_{\bf n}u|_{\partial M \times [-T,T]}=0$ and the initial condition $\partial_t u (\cdot,0)=0$. Assume the Dirichlet boundary value of $u$ satisfies
$$u|_{\partial M\times [-T,T]}\in H^{2}(\partial M \times [-T,T]).$$
If  
$$\|u(\cdot,0)\|_{H^1(M)}\leqslant \Lambda,\quad \|u\|_{H^{2}(\Gamma\times [-T,T])}\leqslant \varepsilon_0,$$
then for $0<h<h_0$, the following estimate holds:
$$\|u(\,\cdotp,0)\|_{L^2(M(\Gamma,T))} \leqslant C_3^{\frac{1}{3}}h^{-\frac{2}{9}}\exp(h^{-C_4 n}) \frac{\Lambda+h^{-\frac{1}{2}}\varepsilon_0}{\big(\log (1+h+h^{\frac{3}{2}}\frac{\Lambda}{\varepsilon_0})\big) ^{\frac{1}{6}}}+C_5\Lambda h^{\frac{1}{3\max{\{n,3\}}}}.$$
The constants $h_0,C_3,C_4,C_5$ explicitly depend only on intrinsic geometric parameters of $M$ and $\Gamma$ (in particular, independent of $\varepsilon_0$).
\end{uc-intro}

\begin{figure}[h]
  \begin{center}
    \includegraphics[width=0.35\linewidth]{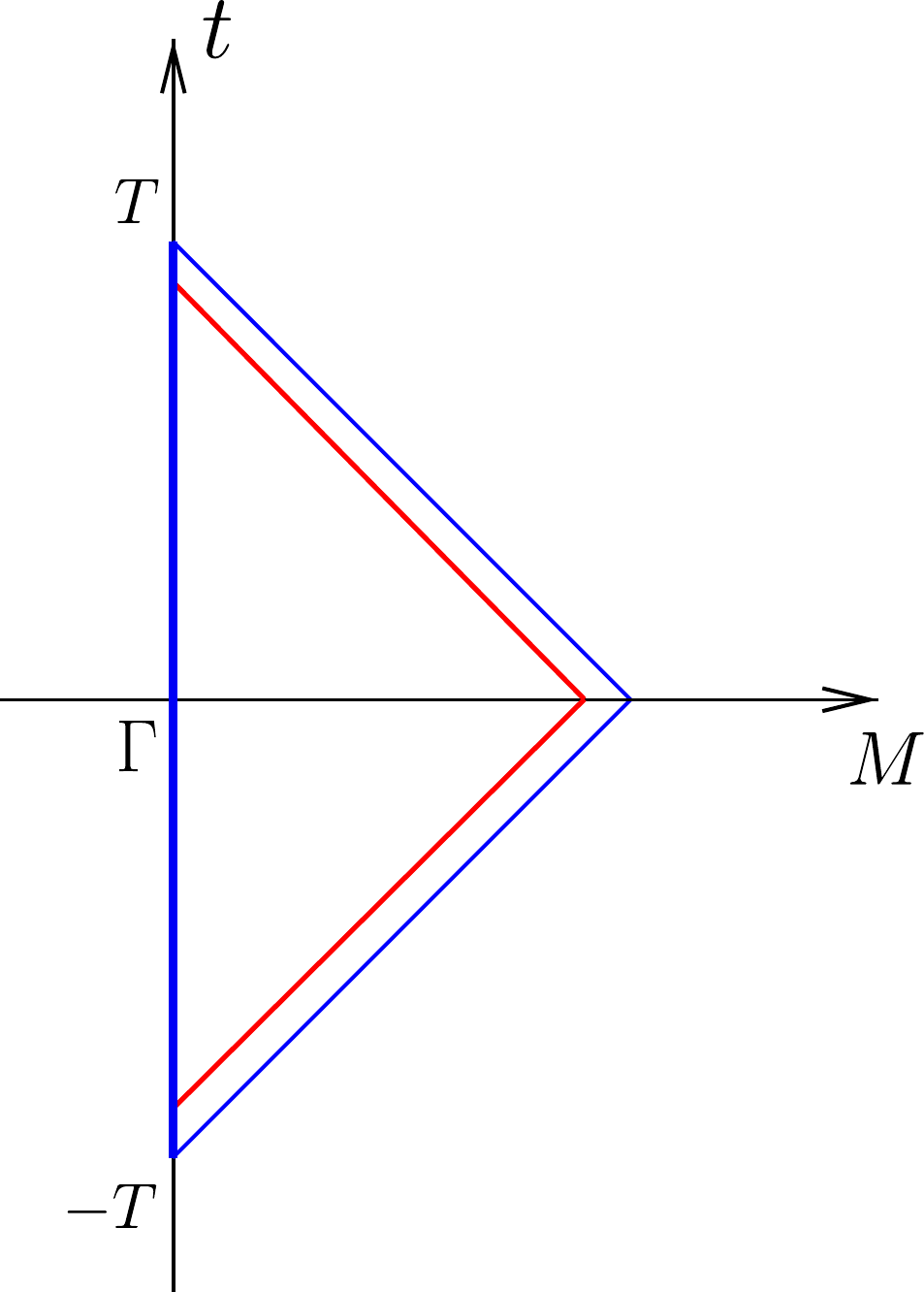}
    \caption{Domains of unique continuation.
    The blue vertical line is $\Gamma\times [-T,T]$.
    The domain enclosed by the blue lines is the optimal domain $K(\Gamma,T)$.
    The domain enclosed by the red lines is $\Omega (h)$ defined in \eqref{Omegaht}, obtained by propagating local unique continuation.
    The distance between the blue and red lines is $\sqrt{h}$.
     }
    \label{fig_domains}
  \end{center}
\end{figure}

Quantitative unique continuation for the wave operator has been investigated independently in \cite{BKL2,BKL1} for closed manifolds and in \cite{LL}, both inspired by Tataru's ideas in \cite{Tataru-preprint}.
In particular, the case for manifolds with boundary for large $T$ was studied in \cite{LL}, however without addressing how the geometry of the manifold affects the estimate.
Our result explicitly shows how the constants depend on the geometry and how close the domain of quantitative unique continuation can approach the optimal domain.
These are crucial questions frequently showing up in the stability of inverse problems.
In Theorem \ref{uc-intro}, the stability estimate is obtained up to the optimal domain for arbitrary $T$, and can be made fully explicit only in terms of intrinsic geometric parameters.
The estimate comprises two parts. One is by propagating local unique continuation up to the $\sqrt{h}$-neighborhood of the boundary of the optimal domain. This is the most technical part of the paper and gives a global estimate (Theorem \ref{main1}) on a domain arbitrarily close to the optimal domain, see Figure \ref{fig_domains}, since $h$ is a small parameter one can freely choose in advance.
The second part is to estimate the $L^2$-norm on the region which the first part does not reach.
Once we prove that this region has uniformly controlled volume (Proposition \ref{area}), the second part of the estimate immediately follows from the \emph{a priori} $H^1$-norm.
We remark that one can also balance the parameters $\varepsilon_0,h$ in Theorem \ref{uc-intro} and arrive at a $\log$-$\log$ type of estimate with a single parameter $\varepsilon_0$.

Equipped with Theorem \ref{uc-intro}, we can adopt the approach introduced in \cite{KKL0} to obtain a stability estimate for Gel'fand's inverse problem. Namely, we apply a quantitative version of the boundary control method to evaluate an approximate volume for the domain of dependence. The error of the approximate volume can be made arbitrarily small as long as sufficient boundary spectral data are known. Then we define approximations to the boundary distance functions through slicing procedures, from which the manifold can be reconstructed (\cite{KKL2}).

The method we use to obtain the quantitative unique continuation may be of independent interest. Essentially it is proved by propagating local stability estimates to obtain a global estimate. However, the presence of general manifold boundaries brings significant trouble in defining the process, especially when the path of propagation touches the boundary. One straightforward approach would be to avoid the boundary. Namely, one can approximate a geodesic touching the boundary with a curve in the interior of the manifold, and propagate local estimates through balls along this curve. This approach works well if the time domain is larger than the diameter of the manifold, in which case the domain of dependence is smooth, i.e. the whole manifold. However, difficulties arise for an arbitrary time domain, where the domain of dependence in the manifold has corners. An estimate obtained with this approach may not be uniform in a class of manifolds.

Our method directly defines a series of non-characteristic domains through which local estimates are propagated, using the intrinsic distance of the manifold and the distance to the boundary. This is made possible by directly handling geodesics near the boundary. These domains are globally defined in a coordinate-free way. The boundaries of these domains normally have the shape of a hyperboloid  and warp quickly near the boundary (and the injectivity radius). In this way, the local estimates propagate (almost) along distance-minimizing geodesics, and naturally produce a uniform global estimate depending only on intrinsic geometric parameters.

\smallskip
This paper is organized as follows. We review relevant concepts and the unique continuation in Section \ref{pre}. Section \ref{section-uc} is devoted to proving Theorem \ref{uc-intro}, an explicit stability estimate for the unique continuation from a subset of the boundary. Section \ref{section-uc} uses several technical lemmas and their proofs can be found in Section \ref{auxiliary}. In Section \ref{section-projection}, we apply Theorem \ref{main1} to introduce
the essential step of our reconstruction method where we compute, in a stable way, how the Fourier coefficients of a function (with respect to the basis of eigenfunctions) change, when the function is multiplied by an indicator function of a union of balls with center points on the boundary. 
 The new feature of this method is that it is directly based on the unique continuation theorem.
 The main results Theorem \ref{stability} and Theorem \ref{Cor1} are proved in Section \ref{section-appro}, with the dependency of constants on geometric parameters derived in Appendix \ref{constants}.

\smallskip
\noindent \textbf{Acknowledgement.} Yaroslav Kurylev worked with us until the middle of this project and quit when he could not work anymore, just a couple of months before he passed away. We cannot formally include him as a co-author, but we still consider him to be one of the authors of this paper.

We would like to thank A. Petrunin for helpful discussions. D. Burago was partially supported by NSF grant DMS-1205597. S. Ivanov was partially supported by RFBR grant 20-01-00070. M. Lassas and J. Lu were partially supported by Academy of Finland, grants 273979, 284715, 312110, and Finnish Centre of Excellence in Inverse Modelling and Imaging.

\section{Preliminaries}\label{pre}

\subsection{Bounded geometry} \label{subsection-bounded}

Let $(M,g)\in \mathcal{M}_n(D,K_1,K_2,i_0,r_0)$ be a compact, connected, orientable Riemannian manifold of dimension $n\geqslant 2$ with smooth boundary $\partial M$. The $C^0$-norm of the Riemannian curvature tensor $R_M$ appeared in (\ref{boundedgeometry}) is defined as
$$\|R_M\|_{C^0}=\sup_{x\in M}\big|R_M|_x \big|\, ,$$
where $\big|R_M|_x\big|$ denotes the operator norm of $R_M$ at $x\in M$ as a multi-linear operator to $\mathbb{R}$. The $C^0$-norms of $S$ and the covariant derivatives are defined in the same way. In this paper, we usually omit the subscript $C^0$ for brevity.

Since the Riemannian curvature tensor is completely determined by the sectional curvatures, assuming a bound on the curvature tensor is equivalent to assuming a bound on sectional curvatures. By the Gauss equation, the bounds on the curvature tensor of $M$ and the second fundamental form of $\partial M$ yield a bound on the curvature tensor $R_{\partial M}$ of $\partial M$ (when $\partial M$ is at least two-dimensional), also denoted by $K_1^2$. Without loss of generality, assume $K_1,K_2>0$.

From now on, we denote $\|A \|=\|A\|_{C^0}$ for a tensor field $A$ on $M$.
For convenience, we denote
$$\|R_M\|_{C^k}=\|R_M\|+\sum_{i=1}^k \|\nabla^i R_M\|,\quad \|S\|_{C^k}=\|S\|+\sum_{i=1}^k \|\nabla^i S\|.$$
Then the curvature bound assumptions in (\ref{boundedgeometry}) are written as
$$\|R_M\|\leqslant K_1^2,\quad \|S\|\leqslant K_1,\quad \|R_{\partial M}\|\leqslant K_1^2,$$
$$\|R_M\|_{C^5}\leqslant K_1^2+K_2,\quad \|S\|_{C^4}\leqslant K_1+K_2,\quad \|R_{\partial M}\|_{C^4}\leqslant C(K_1,K_2).$$

The boundary $\partial M$ is said to admit a boundary normal neighborhood of width $r$ if the exponential map $(z,s)\mapsto \exp_z(s\textbf{n}_z)$ defines a homeomorphism from $\partial M\times [0,r]$ to the $r$-neighborhood of $\partial M$, where $\textbf{n}_z$ denotes the inward-pointing unit normal vector at $z\in\partial M$ (see e.g. Section 2.1.16 in \cite{KKL}). The \emph{boundary injectivity radius} $i_b(M)$ of $M$ is defined as the largest number with the following property that $\partial M$ admits a boundary normal neighborhood of width $r$ for any $r<i_b(M)$. The injectivity radius $\textrm{inj}(M)$ of $M$ is usually defined as the largest number $r\leqslant \min\{\textrm{inj}(\partial M),i_b(M)\}$ satisfying the following condition: the open ball $B_{r}(x)$ of radius $r$ is a domain of Riemannian normal coordinates on $M$ centered at any $x\in M$ with $d(x,\partial M)\geqslant r$.

This definition of the injectivity radius for a manifold with boundary gives little information on the geometry near the boundary. We find it convenient to consider the following quantity.

\begin{Def-CATradius}\label{Def-CATradius}
For $x\in M$, $r_{\textrm{CAT}}(x)$ is defined to be the largest number $r$, such that the (distance-)minimizing geodesic of $M$ connecting $x$ and any $y\in B_r(x)$ is unique. Define 
$$r_{\textrm{CAT}}(M)=\inf_{x\in M} r_{\textrm{CAT}}(x).$$
We call this quantity the radius of radial uniqueness (or CAT radius).
\end{Def-CATradius}
The radius of radial uniqueness is positive for a compact Riemannian manifold with smooth boundary (Lemma \ref{CATradius}(1)). This definition is a natural extension of the injectivity radius for manifolds without boundary. More precisely, for a Riemannian manifold without boundary, $\min\{\pi/\sqrt{K},r_{\textrm{CAT}}\}$ gives a lower bound for the injectivity radius, where $K$ is the upper bound for the sectional curvatures.

The radius of radial uniqueness has an immediate connection with metric spaces of curvature bounded above in the sense of Alexandrov. A metric space has curvature bounded above (globally) by $K>0$ if every minimizing geodesic triangle in the space has perimeter less than $2\pi/\sqrt{K}$, and has each of its angles at most equal to the corresponding angle in a triangle with the same side-lengths in the surface of constant curvature $K$. This space is denoted by CAT$(K)$. A CAT$(K)$ space has the property that any pair of points with distance less than $\pi/\sqrt{K}$ is connected by a unique (within the space) minimizing geodesic, and the geodesic continuously depends on its endpoints. It is well-known that a Riemannian manifold $M$ with smooth boundary is locally CAT$(K)$, where $K$ is the upper bound for the sectional curvatures of $M$ and the second fundamental form of $\partial M$ (the Characterization Theorem in \cite{ABB2}). In fact, more is known: the open ball around any point in $M$ of the radius $\min\{\pi/2\sqrt{K},r_{\textrm{CAT}}(M)\}$ is CAT$(K)$ (Theorem 4.3 in \cite{AB}). This is where the notation $r_{\textrm{CAT}}$ comes from. The CAT space provides useful non-differential tools to work with manifold boundaries where the standard differential machinery is often problematic. 

\subsection{Wave operator and the unique continuation}

The Laplace-Beltrami operator $\Delta_g$ with respect to the metric $g$ has the following form in local coordinates $(x^1,\cdots,x^n)$:
\begin{equation}\label{Laplacian}
\Delta_g=\frac{1}{\sqrt{\det(g_{ij})}} \sum_{i,j=1}^n \frac{\partial}{\partial x^i} \Big(\sqrt{\det(g_{ij})} g^{ij}\frac{\partial}{\partial x^j}\Big).
\end{equation}
Then the wave operator $P=\partial_t^2-\Delta_g$ has the following form in local coordinates:
\begin{eqnarray}\label{Pdef}
P&=& \frac{\partial^2}{\partial t^2}-\frac{1}{\sqrt{\det(g_{ij})}} \sum_{i,j=1}^n \frac{\partial}{\partial x^i} \Big(\sqrt{\det(g_{ij})} g^{ij}\frac{\partial}{\partial x^j}\Big) \\ 
&=&\frac{\partial^2}{\partial t^2}-\sum_{i,j=1}^n g^{ij}\frac{\partial^2}{\partial x^i \partial x^j} + \textrm{lower order terms}. \nonumber
\end{eqnarray}
The Riemannian metric $g$ approximates the standard Euclidean metric in small scale. In sufficiently small coordinate charts, the Laplace-Beltrami operator is a strongly elliptic operator given by the formula (\ref{Laplacian}). However, the wave operator of the form above is only locally defined on manifolds, different from the wave operator on Euclidean spaces with global coefficients. 

In the boundary normal neighborhood of $\partial M$, it is convenient to use the boundary normal coordinate $(x^1,\cdots,x^{n-1},x^n)$, where $(x^1,\cdots,x^{n-1})$ is a choice of coordinate at the nearest point on $\partial M$ and $x^n=d(x,\partial M)$. In other words, the coordinate $(x^1,\cdots,x^{n-1},d(x,\partial M))$ is defined by pushing forward the local coordinate $(x^1,\cdots,x^{n-1})$ on $\partial M$ via the family of exponential maps $z\mapsto\exp_z(s\textbf{n}_z)$ from the boundary in the normal direction. Note that the choice of coordinate on $\partial M$ is fixed. Hence by the Gauss lemma, the metric $g$ has the form of a product metric in such coordinate: 
$$g=(d x^n)^2+\sum_{\alpha,\beta=1}^{n-1}g_{\alpha\beta}dx^{\alpha}dx^{\beta}.$$

On the boundary $\partial M$, two frequent choices of coordinates are the geodesic normal coordinates and the harmonic coordinates. In this paper, we use the geodesic normal coordinates of $\partial M$. Namely, at any point on $\partial M$, we have a geodesic normal coordinate $(x^{\alpha})_{\alpha=1}^{n-1}$ in the ball (of $\partial M$) of a sufficiently small radius, such that 
\begin{equation}\label{coorb}
\frac{1}{2}|\xi|^2\leqslant \sum_{\alpha,\beta=1}^{n-1} g^{\alpha\beta}\xi_{\alpha}\xi_{\beta} \leqslant 2|\xi|^2\;\; (\xi\in\mathbb{R}^{n-1}),
\end{equation}
$$\|g_{\alpha\beta}\|_{C^{1}}\leqslant 2,\quad  \|g_{\alpha\beta}\|_{C^{4}}\leqslant C(n,K_1,K_2,i_0).$$
It is known that the radius of the ball in which the conditions above are satisfied is uniformly bounded below by a positive number explicitly depending on $n, \|R_{\partial M}\|_{C^1},i_0$ (Lemma 8 in \cite{HV} and Theorem A in \cite{E}). We denote this uniform radius by $r_g(\partial M)$. 

\smallskip
Recall that the wave operator $P$ enjoys the unique continuation property from the boundary, namely if the Cauchy boundary data of a wave $u$ (a solution of the wave equation $Pu=0$) vanish on $\Gamma\times [-T,T]$, i.e.,
$$u|_{\Gamma\times [-T,T]}=0,\quad \frac{\partial u}{\partial \textbf{n}} \big|_{\Gamma \times [-T,T]}=0,$$
then the wave vanishes in the double cone of influence $K(\Gamma,T)$ defined in \eqref{def-Kcone}, see \cite{T} or e.g. Theorem 3.16 in \cite{KKL}. Here $\textbf{n}$ denotes the unit normal vector field on $\partial M$ pointing inwards. We are interested in its stability: when the Cauchy boundary data are small on $\Gamma\times [-T,T]$, we consider if the wave is small in the double cone. The following global stability result on Tataru's unique continuation principle (\cite{T}) was proved in \cite{BKL2}, from which the stability of the unique continuation from a ball on a closed Riemannian manifold can be obtained (Theorem 3.3 in \cite{BKL2}).
\begin{global1}\label{global}(Theorem 1.2 in \cite{BKL2})
Let $\Omega_{bd}$ be a bounded connected open subset of $\mathbb{R}^n\times \mathbb{R}$ and $P$ be the wave operator (\ref{Pdef}). Assume $u\in H^1(\Omega_{bd})$ and $Pu\in L^2(\Omega_{bd})$. In $\Omega_{bd}$, we assume the existence of a finite number of connected open subsets $\Omega_{j}^0$ and $\Omega_{j}$, $j=1,2,\dots,J$ a connected set $\Upsilon$ and functions $\psi_j$ satisfying the following assumptions.
\begin{enumerate}[(1)]
\item $\psi_j\in C^{2,1}(\Omega_{bd})$; $p(\cdot,\nabla \psi_{j})\neq 0$ and $\nabla \psi_j\neq 0$ in $\Omega_{j}^0$, where $p$ denotes the principle symbol of the wave operator $P$.
\item ${\rm supp}(u)\cap \Upsilon=\emptyset$; there exists $\psi_{max,j}\in\mathbb{R}$ such that $\emptyset\neq\{y\in \Omega_{j}^0: \psi_j(y)> \psi_{max,j}\}\subset \overline{\Upsilon}_j$, where $\Upsilon_j=\Omega_{j}^0\cap(\cup_{l=1}^{j-1}\Omega_l\cup \Upsilon)$.
\item $\Omega_j=\{y\in \Omega_{j}^0-\overline{\Upsilon}_j: \psi_j(y) > \psi_{min,j}\}$ for some $\psi_{min,j}\in\mathbb{R}$, and $dist(\partial\Omega_{j}^0,\Omega_j)>0$.
\item $\overline{\Omega}$ is connected, where $\Omega=\cup_{j=1}^J \Omega_j$.
\end{enumerate}
Then the following estimate holds for $\Omega$ and $\Omega^0=\cup_{j=1}^J\Omega_{j}^0$:
$$\|u\|_{L^2(\overline{\Omega})}\leqslant C \frac{\|u\|_{H^1(\Omega^0)}}{\Big(\log\big(1+\frac{\|u\|_{H^1(\Omega^0)}}{\|Pu\|_{L^2(\Omega^0)}}\big)\Big)^{\theta}}\, ,$$
where $\theta\in(0,1)$ is arbitrary, and the constant $C$ explicitly depends on $\theta$, $\psi_j$, $ dist(\partial\Omega_{j}^0,\Omega_j)$, $\|g^{ij}\|_{C^1}$, $\vol_{n+1}(\Omega_{bd})$.
\end{global1}

The intuition behind this result is propagating the unique continuation step by step to cover a large domain, as long as the error introduced in each step is small. The set $\Upsilon$ is the initial domain where the function $u$ vanishes, and $\Omega_j$ is the domain propagated by the unique continuation at the $j$-th step. The estimate is obtained by propagating local stability estimates, and the assumptions make sure that certain support conditions (Assumption A1 in \cite{BKL1}) required by the local stability estimates are satisfied at every step. For some simple cases, one choice of the domains and functions is enough, for example if the function $u$ initially vanishes over a ball in $\mathbb{R}^{n}$. However, these assumptions are rather restrictive for general cases, and multiple iterations of the domains and functions need to be carefully constructed to handle the difficulties brought by the geometry of the boundary and the injectivity radius. Note that the constant in the estimate depends on higher derivatives of $\psi_j$ in $\Omega_{j}^0$. It is crucial to construct the required domains where $\psi_j$ has uniformly bounded higher derivatives. Although Theorem \ref{global} is formulated in Euclidean spaces, it applies to manifolds since it is obtained by propagating local stability estimates, which can be done in local coordinate charts.

\subsection{Notations}\label{subsection-notations}
We introduce several notations that we frequently use in this paper. Denote by $\vol_k$ the $k$-dimensional Hausdorff measure on $M$. When the Hausdorff dimension of a set in question is clear, we omit the subscript $k$. In particular, we denote by $\vol(M)$ the Riemannian volume of $M$, and by $\vol(\partial M)$ the Riemannian volume of $\partial M$ with respect to the induced metric on $\partial M$.

Given an open subset $\Gamma\subset\partial M$, we define the following domain with a positive parameter $h<1$ by
\begin{equation}\label{Omegaht}
\Omega_{\Gamma,T}(h)=\big\{(x,t)\in M\times [-T,T]: T-|t|-d(x,\Gamma) >\sqrt{h},\; d(x,\partial M-\Gamma)>h \big\},
\end{equation}
and we write $\Omega(h)$ for short. Note that $\Omega(h)$ is a subset of the double cone of influence $K(\Gamma,T)$, and $\Omega(h)$ approximates $K(\Gamma,T)$ as $h\to 0$. If $\Gamma=\partial M$, the set above is defined with the last condition dropped. In this paper, our consideration always includes the possibility that $\Gamma=\partial M$. For the sole purpose of incorporating this special case notation-wise in later proofs, we set any distance from the empty set to be infinity.

Given a function $u:\partial M\times [-T,T]\to \mathbb{R}$ and an open subset $\Gamma\subset \partial M$, we define the following norm
\begin{equation}\label{H21}
\|u\|_{H^{2,2}(\Gamma \times [-T,T])}^2=\int_{-T}^T \big(\|u(\cdot,t)\|_{H^2(\Gamma)}^2 +\|\partial_t u(\cdot,t)\|_{L^2(\Gamma)}^2+\|\partial_t^2 u(\cdot,t)\|_{L^2(\Gamma)}^2\big)\, dt,
\end{equation}
if $u(\cdot,t)\in H^2(\Gamma)$ and $\partial_t u(\cdot,t),\,\partial_t^2 u(\cdot,t) \in L^2(\Gamma)$ for all $|t|\leqslant T$. We say $u\in H^{2,2}(\Gamma\times[-T,T])$ if the norm above is finite, and we call it the $H^{2,2}$-norm.

\section{Stability of the unique continuation}\label{section-uc}

In this section, we obtain an explicit estimate on the stability of the unique continuation for the wave operator, provided small Cauchy data on a connected open subset of the manifold boundary.  First we state this result as follows.

\begin{main1}\label{main1}
Let $M\in \mathcal{M}_n(D,K_1,K_2,i_0,r_0)$ be a compact, orientable Riemannian manifold with smooth boundary $\partial M$, and let $\Gamma$ (possibly $\Gamma=\partial M$) be a connected open subset of $\partial M$ with smooth boundary. Denote by $i_b(\overline{\Gamma})$ the boundary injectivity radius of $\overline{\Gamma}$.
Then there exist a constant $C_3>0$, that explicitly depends on $n,T,D,K_1,\|\nabla R_M\|_{C^0},\|\nabla S\|_{C^0},$ $i_0,r_0,\hbox{vol}_n(M),\hbox{vol}_{n-1}(\Gamma)$,  an absolute constant
 $C_4>0$, 
and a 
 sufficiently small constant $h_0>0$, that explicitly depends on $n,T,K_1,K_2,i_0,r_0,i_b(\overline{\Gamma}),$ $\vol_{n-1}(\partial M)$,
 such that  the following holds.

 Suppose $u\in H^2(M\times[-T,T])$ is a solution of the non-homogeneous wave equation $Pu=f$ with $f\in L^2(M\times [-T,T])$. Assume the Cauchy data satisfy
\begin{equation}\label{smoothness of C-data}
u|_{\partial M\times [-T,T]}\in H^{2,2}(\partial M \times [-T,T]),\quad \frac{\partial u}{\partial \mathbf{n}} \in H^{2,2}(\partial M \times [-T,T]).
\end{equation}
If 
\begin{equation}\label{quantitative smoothness}
\|u\|_{H^1(M\times[-T,T])}\leqslant \Lambda_0,\quad \|u\|_{H^{2,2}(\Gamma\times [-T,T])}+\big\|\frac{\partial u}{\partial \mathbf{n}}\big\|_{H^{2,2}(\Gamma\times [-T,T])}\leqslant \varepsilon_0,
\end{equation}
then for $0<h<h_0$, we have
$$\|u\|_{L^2(\Omega(h))}\leqslant C_3 \exp(h^{-C_4 n})\frac{\Lambda_0+h^{-\frac{1}{2}}\varepsilon_0}{\bigg(\log \big(1+\frac{\Lambda_0+h^{-\frac{1}{2}}\varepsilon_0}{\|Pu\|_{L^2(M\times[-T,T])}+h^{-\frac{3}{2}}\varepsilon_0}\big)\bigg) ^{\frac{1}{2}}}\, .$$
The domain $\Omega(h)$ and the $H^{2,2}$-norm are defined in Section \ref{subsection-notations}.

As a consequence, the following estimate holds for any $\theta\in (0,1)$ by interpolation:
$$\|u\|_{H^{1-\theta}(\Omega(h))}\leqslant C_3^{\theta}\exp(h^{-C_4 n})\frac{\Lambda_0+h^{-\frac{1}{2}}\varepsilon_0}{\bigg(\log \big(1+\frac{\Lambda_0+h^{-\frac{1}{2}}\varepsilon_0}{\|Pu\|_{L^2(M\times[-T,T])}+h^{-\frac{3}{2}}\varepsilon_0}\big)\bigg) ^{\frac{\theta}{2}}}\, .$$
\end{main1}

\begin{remark}
In Theorem \ref{main1}, the different smoothness indexes of the Sobolev spaces in the qualitative smoothness assumption $u\in H^2(M\times [-T,T])$ and  in the quantitative bounds for the Sobolev norms \eqref{quantitative smoothness} are related to the smooth extension of the weak solution of the wave equation to a boundary layer. We note that the non-uniform smoothness assumptions
are typical, and sometimes also optimal, for the weak solutions of the wave equation with the Neumann boundary condition, see \cite{LT}.
We also note that in Theorem \ref{main1}, the assumption $u\in H^2(M\times [-T,T])$  can be relaxed to 
the assumption that  $u$ is a weak solution of the wave equation $Pu=f$ with the Neumann boundary condition, where $f\in L^2(M \times [-T,T])$, and $u$ and its Neumann boundary value $\partial_{\bf n}u|_{\partial M \times [-T,T]}$
satisfy 
$$u \in C([-T,T];H^1(M))\cap C^1([-T,T];L^2(M)),$$ 
$$\partial_{\bf n}u|_{\partial M \times [-T,T]}\in L^2(\partial M \times [-T,T]).$$
Then, by \cite[Thm.\ A]{LT},  the Dirichlet boundary value is a well-defined function
$u|_{\partial M \times [-T,T]}\in L^2(\partial M \times [-T,T])$. In this case,
\eqref{smoothness of C-data} can be viewed as an additional smoothness requirement
for the Dirichlet and the Neumann boundary values of $u$.
This relaxation of the smoothness assumptions only affects the last part of the proof of  Lemma \ref{extension}, and this lemma can be proved via the weak version of Green's formula.
\end{remark}

Our method can also be used to derive a stability estimate for the unique continuation from any open domain in the interior of $M$, as long as the boundary of the domain is smoothly embedded in $M$. In this way, a stability estimate can be obtained on domains arbitrarily close to the double cone of influence from the interior domain in question, which provides a generalization of Theorem 3.3 in \cite{BKL2}. We remark that as the domain approaches the double cone of influence, the estimate above grows exponentially. This $\exp$-dependence and the $\log$-type of the estimate itself eventually lead to the two logarithms in Theorem \ref{stability}. We also mention that Proposition \ref{area} may be of independent interest, which provides an explicit uniform bound for the Hausdorff measure of the boundary of the domain of influence.

Most of this section is occupied by the proof of Theorem \ref{main1}. First we properly extend the manifold, the wave operator $P$ and the wave $u$, so that $Pu$ stays small on the manifold extension over $\Gamma$, given sufficiently small Cauchy data on $\Gamma$. The extension of $u$ is cut off near the boundary in the manifold extension, from which we start propagating the unique continuation. Then we carefully construct a series of domains satisfying the assumptions in Theorem \ref{global}, such that the union of these domains approximates the double cone of influence. Thus Theorem \ref{global} gives a stability estimate on domains arbitrarily close to the double cone of influence. 

The main difficulty lies in actually finding that series of domains satisfying the properties stated above, as the assumptions in Theorem \ref{global} (essentially assumptions for local estimates) are rather restrictive for a general manifold with boundary. This requires us to directly deal with the intrinsic distance and (distance-minimizing) geodesics of the manifold. In this section, we use several technical lemmas and their proofs can be found in Section \ref{auxiliary}.

\smallskip
Theorem \ref{main1} yields the following stable continuation result on the whole domain of influence $M(\Gamma,T)$.

\begin{wholedomain}\label{wholedomain}
Let $M\in \mathcal{M}_n(D,K_1,K_2,i_0,r_0)$ be a compact Riemannian manifold with smooth boundary $\partial M$, and let $\Gamma$ (possibly $\Gamma=\partial M$) be a connected open subset of $\partial M$ with smooth boundary. Suppose $u\in H^2(M\times[-T,T])$ is a solution of the wave equation $Pu(x,t)=0$ with the Neumann boundary condition $\partial_{\bf n}u|_{\partial M \times [-T,T]}=0$ and the initial condition $\partial_t u(\cdot,0)=0$. Assume the Dirichlet boundary value of $u$ satisfies
$$u|_{\partial M\times [-T,T]}\in H^{2,2}(\partial M \times [-T,T]).$$
If  
$$\|u(\cdot,0)\|_{H^1(M)}\leqslant \Lambda,\quad \|u\|_{H^{2,2}(\Gamma\times [-T,T])}\leqslant \varepsilon_0,$$
then for $0<h<h_0$, the following estimate holds:
$$\|u(\,\cdotp,0)\|_{L^2(M(\Gamma,T))} \leqslant C_3^{\frac{1}{3}}h^{-\frac{2}{9}}\exp(h^{-C_4 n}) \frac{\Lambda+h^{-\frac{1}{2}}\varepsilon_0}{\big(\log (1+h+h^{\frac{3}{2}}\frac{\Lambda}{\varepsilon_0})\big) ^{\frac{1}{6}}}+C_5\Lambda h^{\frac{1}{3\max{\{n,3\}}}}.$$
Here $C_3$ explicitly depends on $n,T,D,\|R_M\|_{C^1},\|S\|_{C^1},i_0,r_0,\vol(M),\vol_{n-1}(\Gamma)$; $C_4$ is an absolute constant; $C_5$ explicitly depends on $n,\|R_M\|_{C^1},\|S\|_{C^1},i_0, \vol(M), \vol(\partial M)$; $h_0>0$ is a sufficiently small constant explicitly depending on $n,T,K_1,K_2,i_0,r_0,i_b(\overline{\Gamma})$, $\vol(\partial M)$.
\end{wholedomain}

We postpone the proof of Proposition \ref{wholedomain}
after the proof of Theorem \ref{main1}.

\subsection{Extension of manifolds} \label{subsection-extension} \hfill

\smallskip
Let $(M,g)\in \mathcal{M}_n(D,K_1,K_2,i_0,r_0)$ be a compact, orientable Riemannian manifold with bounded geometry defined in Section \ref{section-intro}.

\begin{extensionmetric}\label{extensionmetric}
For sufficiently small $\delta_{ex}$ explicitly depending on $n,K_1,K_2,i_0,$ $\vol(\partial M)$, we can extend $(M,g)$ to a Riemannian manifold $(\widetilde{M},\widetilde{g})$ with smooth boundary such that the following properties are satisfied.
\begin{enumerate}[(1)]
\item $\widetilde{M}-M$ lies in a normal neighborhood of $\partial M$ in $\widetilde{M}$, and $\widetilde{d}(x,\partial M)=\delta_{ex}$ for any $x\in \partial \widetilde{M}$, where $\widetilde{d}$ denotes the distance function of $\widetilde{M}$.
\item $\widetilde{g}$ is of $C^{3,1}$ in some atlas on $\widetilde{M}$, in which 
$$\|\widetilde{g}_{ij}|_{\widetilde{M}-M}\|_{C^1}\leqslant C(K_1),\quad \|\widetilde{g}_{ij}|_{\widetilde{M}-M}\|_{C^4}\leqslant C(n,K_1,K_2,i_0).$$
\item $\|R_{\widetilde{M}}\|\leqslant 2K_1^2$, $\|S_{\partial \widetilde{M}}\|\leqslant 2K_1$ and $\|\nabla R_{\widetilde{M}}\|\leqslant 2K_2$, where $S_{\partial \widetilde{M}}$ denotes the second fundamental form of $\partial \widetilde{M}$ in $\widetilde{M}$.
\end{enumerate}

As a consequence, we have\\
(4) \,$r_{\textrm{CAT}}(\widetilde{M})\geqslant \min \big\{C(K_1),i_0/4,r_0/2 \big\}$.
\end{extensionmetric}

\begin{proof}
We glue a collar $\partial M \times [-\delta_{ex},0]$ for $0<\delta_{ex}<\min\{1,i_0/2\}$ onto $M$ by identifying $\partial M\times \{0\}$ of the collar with $\partial M$. Denote the topological space after the gluing procedure by $\widetilde{M}$. Any $(y,\rho)\in \partial M \times [-\delta_{ex},0]$ admits coordinate charts by extending boundary normal coordinate charts at $(y,-\rho)\in M$. The transition maps are clearly smooth and therefore $\widetilde{M}$ is a smooth manifold.

Let $\{y_i\}$ be a maximal $r_g(\partial M)/2$-separated set (and hence an $r_g(\partial M)/2$-net) in $\partial M$. Let $U_i$ be the ball of radius $r_g(\partial M)$ in $\partial M$ around $y_i$, and therefore $\{U_i\}$ is an open cover of $\partial M$. We take a partition of unity $\{\phi_i\}$ subordinate to $\{U_i\}$ satisfying 
$$\|\phi_i\|_{C^s}\leqslant C\,r_g(\partial M)^{-s},\textrm{ for }s\in [1,4].$$
Then $\{\widetilde{U}_i:=U_i\times [-\delta_{ex},0]\}$ is an open cover of the collar $\partial M \times [-\delta_{ex},0]$, and $\{\widetilde{\phi}_i\}$ is a partition of unity subordinate to this cover satisfying the same bound on $C^s$-norm, where $\widetilde{\phi}_i$ is defined by $\widetilde{\phi}_i(y,\rho)=\phi_i(y)$ for $(y,\rho)\in \partial M \times [-\delta_{ex},0]$.

We choose the geodesic normal coordinate $(y^{\alpha})_{\alpha=1}^{n-1}$ on each $U_i$ such that (\ref{coorb}) holds. Within each coordinate chart $\widetilde{U}_i$, we define the metric components at $(y,\rho)\in \widetilde{U}_i$ as follows: $\widetilde{g}^{(i)}_{\rho\rho}=1$, and $\widetilde{g}^{(i)}_{\alpha \rho}=0$ for $\alpha,\beta=1,\cdots,n-1$, and 
\begin{equation*}
\widetilde{g}^{(i)}_{\alpha\beta}(y,\rho)=g^{(i)}_{\alpha\beta}(y,0)+\rho \frac{\partial g^{(i)}_{\alpha\beta}}{\partial \rho} (y,0)+\frac{\rho^2}{2} \frac{\partial^2 g^{(i)}_{\alpha\beta}}{\partial \rho^2} (y,0)+ \frac{\rho^3}{6} \frac{\partial^3 g^{(i)}_{\alpha\beta}}{\partial \rho^3} (y,0), \;\textrm{ for }\rho\leqslant 0.
\end{equation*}
Then one can define a Riemannian metric $\widetilde{g}$ on $\partial M\times [-\delta_{ex},0]$ through partition of unity:
\begin{equation}\label{metricpartition}
\widetilde{g}|_{(y,\rho)}=\sum_i \widetilde{\phi}_i(y,\rho) g^{(i)}|_{(y,\rho)}=\sum_i \phi_i(y) g^{(i)}|_{(y,\rho)},\;\textrm{ for }\rho\leqslant 0.
\end{equation}
At $(y,\rho\in \mathbb{R}_{+})\in M$ with respect to the boundary normal coordinate of $\partial M$ in $M$, define $\widetilde{g}=g$. Due to the Riccati equation (e.g. Theorem 2 in \cite{PP}, p44), 
the derivatives of $g^{(i)}_{\alpha\beta}$ with respect to $\rho$ at $\rho=0$ up to the third order can be expressed in terms of the components of $S$, $R_M$ and $\nabla R_{M}$. Then the curvature bound assumptions (\ref{boundedgeometry}) implies that $\widetilde{g}_{\alpha\beta}^{(i)}$ is of $C^4$ within each coordinate chart $\widetilde{U}_i$. 

Now let us consider the coordinate charts $U_i\times [-\delta_{ex},i_0)$. In this coordinate, the components $\widetilde{g}_{\alpha\beta}^{(i)}$ is of $C^{3,1}$ in the normal direction, and $C^4$ in other directions. Therefore $\widetilde{g}$ is of $C^{3,1}$ in the local coordinate charts $\{U_i\times [-\delta_{ex},i_0)\}$.

Furthermore, it follows from a straightforward calculation that for $\rho\leqslant 0$,
\begin{equation*}
\bigg|\frac{\partial^{k+l} \widetilde{g}^{(i)}_{\alpha\beta}}{\partial  x_T^k \partial \rho^l}(y,\rho)-\frac{\partial^{k+l} g^{(i)}_{\alpha\beta}}{\partial x_T^k \partial \rho^l}(y,0)\bigg|\leqslant C(\|R_M\|_{C^5},\|S\|_{C^4}) |\rho|, \;\textrm{ for }k+l\leqslant 4,\, l\leqslant 3.
\end{equation*}
Note that $\partial^4 \widetilde{g}^{(i)}_{\alpha\beta}/\partial \rho^4=0$ by definition. Recall that the $C^4$-norm of $\phi_i$ is uniformly bounded by $C\,r_g(\partial M)^{-4}$, and $r_g(\partial M)$ explicitly depends on $n,\|R_{\partial M}\|_{C^1},i_0$. Furthermore, the total number of coordinate charts $U_i$ is bounded by $C(n,K_1)\vol(\partial M)$ $r_g(\partial M)^{-n+1}$. Hence by (\ref{metricpartition}), the estimates above hold for $\widetilde{g}_{\alpha\beta}$ and $g_{\alpha\beta}$ with another constant $C(n,\|R_M\|_{C^5},\|S\|_{C^4},$ $i_0,\vol(\partial M))$.

Therefore we can restrict the extension width $\delta_{ex}$ to be sufficiently small explicitly depending only on $n,K_1,K_2,i_0,\vol(\partial M)$, such that the matrix $(\widetilde{g}_{\alpha\beta})$ is nondegenerate and hence a metric, and 
\begin{equation}\label{curvatureextended}
\|\widetilde{g}_{\alpha\beta}|_{\widetilde{M}-M}\|_{C^1}\leqslant 4 K_1+4,\quad \|\widetilde{g}_{\alpha\beta}|_{\widetilde{M}-M}\|_{C^4}\leqslant C(n,K_1,K_2,i_0),
\end{equation}
$$\|R_{\widetilde{M}}\|\leqslant 2K_1^2, \quad \|S_{\partial \widetilde{M}}\|\leqslant 2K_1, \quad \|\nabla R_{\widetilde{M}}\|\leqslant 2K_2.$$
Here the first inequality is due to (\ref{coorb}) and the definition that $\partial_{\rho} g_{\alpha\beta} |_{\partial M}=2S_{\alpha\beta}$,
where $S_{\alpha\beta}$ denotes the components of the second fundamental form $S$ of $\partial M$. The bound on $S_{\partial \widetilde{M}}$ follows from the bound on $\partial_{\rho} \widetilde{g}_{\alpha\beta}|_{\widetilde{M}-M}$.

With this type of extension, $\widetilde{g}$ is also a product metric in the collar, which implies that the integral curve of $\partial/\partial \rho$ minimizes length and is hence a minimizing geodesic. This shows that for any $x=(y,\rho)\in \partial M \times[-\delta_{ex},0]$, we have $\widetilde{d}(x,\partial M)=-\rho$, which yields property (1). The property (4) is due to properties (1-3) and Lemma \ref{CATradius}(2).
\end{proof}


\noindent \textbf{Coordinate system.} From now on, we extend the manifold $(M,g)$ to $(\widetilde{M},\widetilde{g})$ such that Lemma \ref{extensionmetric} holds. We say $(\widetilde{M},\widetilde{g})$ is an extension of $(M,g)$ with the extension width $\delta_{ex}$. We choose a coordinate system on $\widetilde{M}$ as follows.

In the boundary normal (tubular) neighborhood of $\partial M$, we choose the boundary normal coordinate of $\partial M$. Let $\{y_i\}$ be a maximal $r_g(\partial M)/2$-separated set in $\partial M$, and $U_i$ be the ball of radius $r_g(\partial M)$ in $\partial M$ around $y_i$. The proof of Lemma \ref{extensionmetric} shows that $\widetilde{g}$ is of $C^{3,1}$ in the coordinate charts $U_i\times [-\delta_{ex},i_0)$. In each coordinate chart, we choose the boundary normal coordinate $(x^1,\cdots,x^{n-1},\rho(x))$ of $\partial M$, where $(x^1,\cdots,x^{n-1})$ is the geodesic normal coordinate of $\partial M$ such that (\ref{coorb}) holds. The coordinate function $\rho(x)$ in the normal direction is defined as
\begin{equation}\label{def-rhox}
\rho(x) = \left\{ \begin{array}{ll}
         d(x,\partial M), & \mbox{if $x\in M$}; \\
         -\widetilde{d}(x,\partial M), & \mbox{if $x\in \widetilde{M}-M$}.\end{array} \right. 
\end{equation}
Note that $\widetilde{d}(x,\partial M)=d(x,\partial M)$ for $x\in M$. Lemma \ref{extensionmetric}(2) shows that the metric components on $\widetilde{M}-M$ have uniformly bounded $C^4$-norm. On the other side, due to Lemma \ref{riccati}, we can find a uniform width $r_b=r_b(K_1,i_0)$, such that the $C^4$-norm of metric components is uniformly bounded by $C(n,K_1,K_2,i_0)$ in the boundary normal coordinate of width $r_b$ in $M$. Consequently, we have a uniform bound for the $C^{3,1}$-norm of metric components in the coordinate charts $U_i\times [-\delta_{ex},i_0)$.

For any point $x\in M$ with $d(x,\partial M)> r_b/2$, we choose the geodesic normal coordinate of $M$ around $x$ of the radius $ \min\{r_b/2,r_g(x)\}$, such that the $C^4$-norm of metric components is uniformly bounded. By Lemma 8 in \cite{HV} and Theorem A in \cite{E}, this radius is uniformly bounded below by $n,\|R_M\|_{C^1},i_0,r_b$. Denote by $r_g$ the minimum of this radius and $r_g(\partial M)$, and therefore $r_g$ explicitly depends only on $n,\|R_M\|_{C^1},\|S\|_{C^1},i_0$. 

Combining these two types of coordinates, we have a coordinate system on $\widetilde{M}$ in which the metric components satisfy the following properties:
$$\frac{1}{4}|\xi|^2\leqslant \sum_{i,j=1}^{n} \widetilde{g}^{ij}\xi_{i}\xi_{j} \leqslant 4|\xi|^2\; (\xi\in\mathbb{R}^n),$$
\begin{equation}\label{metricboundex}
\|\widetilde{g}_{ij}\|_{C^{1}}\leqslant C(n,\|R_{M}\|_{C^1},\|S\|_{C^1}),\quad \|\widetilde{g}_{ij}\|_{C^{3,1}}\leqslant C(n,K_1,K_2,i_0).
\end{equation}

Observe that for any $x\in \widetilde{M}$, the ball $\widetilde{B}_{r_g/2}(x)$ of $\widetilde{M}$ or the cylinder $B_{\partial M}(y,r_g/2)\times (\rho-r_g/2,\rho+r_g/2)$ is contained in at least one of the coordinate charts defined above, where $x=(y,\rho)$ if $x$ is in the boundary normal coordinate of $\partial M$. To see this, it suffices to show that for any $y\in \partial M$, the ball $B_{\partial M}(y,r_g/2)$ of $\partial M$ is contained in at least one of $U_i$. The latter statement is a direct consequence of the fact that $\{y_i\}$ is an $r_g(\partial M)/2$-net in $\partial M$.


\subsection{Extension of functions} \hfill

\smallskip
Let $(\widetilde{M},\widetilde{g})$ be an extension of $(M,g)$ satisfying Lemma \ref{extensionmetric} with the extension width $\delta_{ex}$. Points in the boundary normal neighborhood of $\partial M$ have coordinates $(x^1,\cdots,x^{n-1},\rho(x))$, where $\rho(x)$ is defined in (\ref{def-rhox}). We write the coordinate as $(x_T,\rho(x))$ for short, where $x_T=(x^1,\cdots,x^{n-1})$ denotes the tangential coordinate.

We define an extension of functions on $M$ to $\widetilde{M}$ as follows. Given a function $u$ on $M$ and its Cauchy data $u,\, \frac{\partial u}{\partial \mathbf{n}}$ on $\partial M$, we extend $u$ to a function $\widetilde{u}_{ex}$ on $\widetilde{M}$ by
\begin{equation*}
\widetilde{u}_{ex}(x_T,\rho,t) = \left\{ \begin{array}{ll}
         u(x_T,\rho,t), & \mbox{if $\rho \geqslant 0$};\\
         u(x_T,0,t)+\rho \frac{\partial u}{\partial \mathbf{n}}(x_T,0,t), & \mbox{if $\rho<0$}.\end{array} \right. 
\end{equation*}
For $0<h<\delta_{ex}$, we define another function $\widetilde{u}:\widetilde{M}\times [-T,T]\to \mathbb{R}$ by $\widetilde{u}=u$ on $M\times [-T,T]$, and
\begin{equation}\label{extensionu}
\widetilde{u}(x_T,\rho,t)=\phi(\frac{\rho}{h}) \widetilde{u}_{ex} (x_T,\rho,t), \; \textrm{ for }\rho< 0,
\end{equation}
where $\phi$ is a monotone increasing smooth function vanishing on $(-\infty,-1]$ and equals to $1$ on $[0,\infty)$ with $\|\phi\|_{C^2}\leqslant 8$. Then $\widetilde{u}=0$ when $\rho\leqslant -h$.

\begin{extension}\label{extension}
Let $(\widetilde{M},\widetilde{g})$ be an extension of $(M,g)$ satisfying Lemma \ref{extensionmetric} with the extension width $\delta_{ex}$. Let $\Gamma$ be a connected open subset of $\partial M$. Assume
$$u|_{\partial M\times [-T,T]}\in H^{2,2}(\partial M \times [-T,T]),\quad \frac{\partial u}{\partial \mathbf{n}} \in H^{2,2}(\partial M \times [-T,T]).$$
Then we have 
$$\|\widetilde{u}\|^2_{H^1(\Omega_{\Gamma}\times[-T,T])}\leqslant Ch^{-1}\|u\|^2_{H^1(\Gamma\times [-T,T])}+Ch \big\|\frac{\partial u}{\partial \mathbf{n}} \big\|^2_{H^1(\Gamma\times [-T,T])},$$
and
\begin{eqnarray*}
\|(\partial_{t}^2 -\Delta_{\widetilde{g}})\widetilde{u}\|^2_{L^2(\Omega_{\Gamma}\times [-T,T])}
\leqslant Ch^{-3}\|u\|^2_{H^{2,2}(\Gamma\times [-T,T])} + Ch^{-1}\big\|\frac{\partial u}{\partial \mathbf{n}} \big\|^2_{H^{2,2}(\Gamma\times [-T,T])},
\end{eqnarray*}
where $\Omega_{\Gamma}=\Gamma\times[-\delta_{ex},0]$ denotes the part of the manifold extension over $\Gamma$, and the constants explicitly depend on $n,K_1$.

Furthermore, suppose $u\in H^2(M\times[-T,T])$ is a solution of the non-homogeneous wave equation $Pu=f$ with $f\in L^2(M\times [-T,T])$. Then $\widetilde{u}\in H^1(\widetilde{M}\times [-T,T])$ and $(\partial_{t}^2 -\Delta_{\widetilde{g}})\widetilde{u} \in L^2(\widetilde{M}\times [-T,T])$.
\end{extension}

\begin{proof}
First we estimate the $H^1$-norm of $\widetilde{u}$ over $\Omega_{\Gamma}$. Here we only estimate the dominating term in $h$; the other terms can be done in the same way. Denote by $\partial_{\alpha},\partial_{n},\partial_t$ the derivatives with respect to $x^{\alpha},x^n$ coordinates and time $t$, respectively. We denote $\partial_{\alpha} u, \partial_{n} u,\partial_t u$ evaluated at $(x_T,0,t)$ by $u_{\alpha},u_{n},u_t$ and $ \phi^{\prime}(s)=\frac{d}{ds}\phi(s)$, evaluated at $s=\rho/h$. In addition, whenever we write the function $u$ without specifying where it is evaluated,  the evaluation is also done at $(x_T,0,t)$. By the definition of $\widetilde{u}$,
\begin{equation}\label{un}
(\partial_n \widetilde{u})(x_T,\rho,t)=h^{-1} (u+\rho u_n) \phi^{\prime}  + u_n\phi.
\end{equation}
Since $\widetilde{u}$ vanishes unless $\rho\in [-h,0]$, we have
\begin{eqnarray*}
\|\partial_n \widetilde{u}\|^2_{L^2(\Omega_{\Gamma}\times [-T,T])} &=& \int_{-T}^{T}  \int_{\Gamma}\int_{-\delta_{ex}}^0   \big| h^{-1}  (u+\rho u_n) \phi^{\prime} + u_n\phi \big|^2 dx_T d\rho dt\\
&\leqslant& C\int_{-T}^{T}  \int_{\Gamma}\int_{-h}^0 (h^{-2}u^2+h^{-2}\rho^2 u_n^2+u_n^2) dx_Td\rho dt \\
&\leqslant& Ch^{-1}\|u\|_{L^2(\Gamma\times [-T,T])}^2+Ch\big\|\frac{\partial u}{\partial \mathbf{n}} \big\|_{L^2(\Gamma\times [-T,T])}^2.
\end{eqnarray*}

Next we estimate the Laplacian of $\widetilde{u}$ over $\Omega_{\Gamma}$ for $\rho\in[-h,0]$. In the boundary normal coordinates of our choice, by definition (\ref{Laplacian}) we have
\begin{eqnarray*}
\Delta_{\widetilde{g}}\widetilde{u}&=&\sum_{i,j=1}^n \frac{1}{\sqrt{|\widetilde{g}|}} \partial_i \big(\sqrt{|\widetilde{g}|} \widetilde{g}^{ij}\partial_j \widetilde{u} \big) \\
&=& \frac{1}{\sqrt{|\widetilde{g}|}} \partial_{n} \big(\sqrt{|\widetilde{g}|} \widetilde{g}^{nn}\partial_{n} \widetilde{u}\big)+
\sum_{\alpha,\beta=1}^{n-1} \frac{1}{\sqrt{|\widetilde{g}|}} \partial_{\alpha} \big(\sqrt{|\widetilde{g}|} \widetilde{g}^{\alpha\beta}\partial_{\beta} \widetilde{u} \big) \\
&=& A_1+A_2,
\end{eqnarray*}
where $|\widetilde{g}|$ denotes the determinant of the matrix $(\widetilde{g}_{ij})$. We estimate $A_2$ as follows.
\begin{eqnarray*}
A_2(x_T,\rho,t) &=& \sum_{\alpha,\beta=1}^{n-1} \frac{1}{\sqrt{|\widetilde{g}|}} \partial_{\alpha} \big(\sqrt{|\widetilde{g}|} \widetilde{g}^{\alpha\beta}\partial_{\beta} \widetilde{u} \big) \\
&=& \sum_{\alpha,\beta} \frac{\partial_{\alpha} |\widetilde{g}|}{2|\widetilde{g}|}  \widetilde{g}^{\alpha \beta} \partial_{\beta}\widetilde{u}+ (\partial_{\alpha} \widetilde{g}^{\alpha \beta})( \partial_{\beta}\widetilde{u})+ \widetilde{g}^{\alpha \beta} \partial_{\alpha}\partial_{\beta}\widetilde{u}.
\end{eqnarray*}
Hence we have
\begin{eqnarray*}
|A_2(x_T,\rho,t)| &\leqslant& C \sum_{\alpha,\beta} (|u_{\beta}|+h |u_{n\beta}|)+C \sum_{\alpha,\beta}|\partial_{\alpha}\partial_{\beta}(u+\rho u_n)|(x_T,0,t)\\
&\leqslant& C\sum_{\alpha,\beta} (|u_{\alpha\beta}|+h |u_{n\alpha\beta}|)+C\sum_{\beta}(|u_{\beta}|+h|u_{n\beta}|),
\end{eqnarray*}
where the constants explicitly depend on $n,K_1$ due to the $C^1$ metric bound (\ref{curvatureextended}).

Finally we estimate $A_1$ and the time derivatives.
Since $\widetilde{g}^{nn}=1$, we know that 
$$A_1(x_T,\rho,t)=\frac{\partial_n |\widetilde{g}|}{2|\widetilde{g}|} \partial_{n}\widetilde{u} +  \partial_{n}^2\widetilde{u}.$$
We differentiate (\ref{un}) again:
\begin{eqnarray*}
(\partial_{n}^2\widetilde{u})(x_T,\rho,t) = h^{-2}(u+\rho u_n)\phi^{\prime\prime}+ 2h^{-1}u_n\phi^{\prime}.
\end{eqnarray*}
Hence we have
\begin{eqnarray*}
\big| \big((\partial_{t}^2 -\partial_{n}^2)\widetilde{u} \big)(x_T,\rho,t)\big|&=&\big| (u_{tt}+\rho u_{ntt})\phi-(\partial_{n}^2\widetilde{u})(x_T,\rho,t) \big|\\
&\leqslant& Ch^{-2}|u|+Ch^{-1}|u_n| +C|u_{tt}|+Ch|u_{ntt}|,
\end{eqnarray*}
which leads to a similar estimate for $(\partial_{t}^2 \widetilde{u}-A_1)(x_T,\rho,t)$ by (\ref{un}). Thus,
\begin{eqnarray*}\label{waveextension}
\big| \big((\partial_{t}^2 -\Delta_{\widetilde{g}})\widetilde{u} \big)(x_T,\rho,t)\big|&\leqslant& Ch^{-2}|u|+Ch^{-1}|u_n| +C(|u_{tt}|+h|u_{ntt}|) \nonumber \\
&+&C\sum_{\alpha,\beta} \big(|u_{\alpha}|+|u_{\alpha\beta}|+h|u_{n\alpha}|+h|u_{n\alpha\beta}| \big),
\end{eqnarray*}
where all terms on the right-hand side are boundary data evaluated at $(x_T,0,t)$. Then the second estimate of the lemma immediately follows from integrating the last inequality.

\smallskip
Now we additionally assume that $u\in H^2(M\times[-T,T])$ is a (strong) solution of the non-homogeneous wave equation $Pu=f$ with $f\in L^2(M\times [-T,T])$. 
By the regularity result for the wave equation (e.g. Theorem 2.30 in \cite{KKL}), the solution $u$ is in the energy class
$$u \in C([-T,T];H^1(M))\cap C^1([-T,T];L^2(M)).$$
From the definition (\ref{extensionu}), the weak derivatives of $\widetilde{u}(\cdot,t)$ exist on $\widetilde{M}$ for any fixed $t\in [-T,T]$. Since the Cauchy data are in $H^{2,2}$, we have $\widetilde{u}(\cdot,t)\in H^1(\widetilde{M})$ for all $t$ directly by definition (\ref{extensionu}), and therefore $\widetilde{u}\in H^1(\widetilde{M}\times [-T,T])$.

Since the Cauchy data are in $H^{2,2}$, the definition (\ref{extensionu}) also indicates that $\widetilde{u}\in H^{2,2}\big((\widetilde{M}-M)\times [-T,T]\big)$. Hence over $\widetilde{M}-M$,
$$\widetilde{f}_{ex}:=(\partial_{t}^2 -\Delta_{\widetilde{g}})\widetilde{u} \in L^2\big((\widetilde{M}-M) \times [-T,T]\big).$$
Define a function $\widetilde{f}:\widetilde{M}\times [-T,T]\to \mathbb{R}$ by $\widetilde{f}=f$ over $M$ and $\widetilde{f}=\widetilde{f}_{ex}$ over $\widetilde{M}-M$. Clearly $\widetilde{f}\in L^2(\widetilde{M}\times [-T,T])$. Thus the only part left is to show that $(\partial_{t}^2 -\Delta_{\widetilde{g}})\widetilde{u}=\widetilde{f}$ on $\widetilde{M}\times [-T,T]$ in the weak form. Observe that the wave equation on either $M$ or $\widetilde{M}-M$ is well-defined pointwise. Then for any test function $\varphi\in H_0^{1}(\widetilde{M}\times [-T,T])$, by applying the wave equation separately on $M,\,\widetilde{M}-M$ and Green's formula, we have
\begin{eqnarray*}
&&\int_{-T}^T \int_{\widetilde{M}} \Big(-\partial_{t} \widetilde{u}\, \partial_t \varphi  + \langle \nabla \widetilde{u},  \nabla \varphi \rangle_{\widetilde{g}} \Big) = \int_{-T}^T \int_{M\cup(\widetilde{M}-M)}  \Big(-\partial_{t} \widetilde{u}\, \partial_t \varphi  + \langle \nabla \widetilde{u},  \nabla \varphi \rangle_{\widetilde{g}} \Big) \\
&&= \int_{-T}^T \int_{M} f\varphi - \int_{-T}^T \int_{\partial M} \frac{\partial u}{\partial \mathbf{n}} \varphi+
 \int_{-T}^T \int_{\widetilde{M}-M}\widetilde{f}_{ex} \varphi+\int_{-T}^T \int_{\partial M} \frac{\partial \widetilde{u}}{\partial \mathbf{n}} \varphi\, .
\end{eqnarray*}
Due to the definition (\ref{extensionu}), the normal derivative of $\widetilde{u}$ from either side of $\partial M$ coincides and hence the boundary terms cancel out. This shows that the wave equation is satisfied on $\widetilde{M}\times[-T,T]$ in the weak form, with the source term in $L^2(\widetilde{M}\times[-T,T])$.
\end{proof}

\subsection{Distance functions} \hfill 

\smallskip
Later in the proof of Theorem \ref{main1}, we will need to switch back and forth to different distance functions. The following lemma shows relations between distance functions.

\begin{distances}\label{distances}
Let $(\widetilde{M},\widetilde{g})$ be an extension of $(M,g)$ satisfying Lemma \ref{extensionmetric} with the extension width $\delta_{ex}$. Denote the distance functions of $M$ and $\widetilde{M}$ by $d$ and $\widetilde{d}$, respectively. Then there exists a uniform constant $r_b$ explicitly depending only on $K_1,i_0$, such that the following inequality holds for any $x,y\in M$ as long as $\delta_{ex}\leqslant r_b$:
$$\widetilde{d}(x,y)\leqslant d(x,y)\leqslant (1+3K_1\delta_{ex})\widetilde{d}(x,y)\, .$$
If $x,y\in \widetilde{M}-M$, then the second inequality holds after replacing $d(x,y)$ with $d(x^{\perp},y^{\perp})$, where $x^{\perp}$ denotes the normal projection of $x$ onto $\partial M$. If $x\in \widetilde{M}-M,\,y\in M$, then the second inequality holds for $d(x^{\perp},y)$.

Furthermore, if a minimizing geodesic of $\widetilde{M}$ between $x,y \in \widetilde{M}$ lies in the boundary normal (tubular) neighborhood of $\partial M$ of width $\delta_{ex}$, then we have
$$d_{\partial M}(x^{\perp},y^{\perp})\leqslant (1+3K_1\delta_{ex})\widetilde{d}(x,y)\, ,$$
where $d_{\partial M}$ denotes the intrinsic distance function of $\partial M$.
\end{distances}

\begin{proof}
The first inequality is trivial and we prove the second inequality. Consider any (distance) minimizing geodesic $\widetilde{\gamma}$ of $\widetilde{M}$ from $x$ to $y$, and its length $L(\widetilde{\gamma})$ satisfies $L(\widetilde{\gamma})=\widetilde{d}(x,y)$ by definition. It is known that $\widetilde{\gamma}$ is a $C^1$ curve with arclength parametrization (e.g. Section 2 in \cite{ABB}). Observe that the second inequality follows trivially if $\widetilde{\gamma}$ lies entirely in $M$. Since the statement of the lemma is independent of the choice of coordinates, we work in the boundary normal coordinate $(x^1,\cdots,x^{n-1},\rho(x))$ of $\partial M$.

Suppose $\widetilde{\gamma}$ lies entirely in $\widetilde{M}-\textrm{int}(M)$ with both endpoints $x,y$ on $\partial M$. Consider the normal projection, denoted by $\gamma$, of $\widetilde{\gamma}$ onto the boundary $\partial M$ with respect to the boundary normal coordinate. More precisely, if $\widetilde{\gamma}(s)=(x_1(s),\cdots,x_{n-1}(s),x_n(s))$ in a boundary normal coordinate near a point on $\widetilde{\gamma}$, then its normal projection has the form $\gamma(s)=(x_1(s),\cdots,x_{n-1}(s),0)$. The fact that $\widetilde{\gamma}$ is of $C^1$ implies that $x_i(s)$ is a $C^1$ function for any $i$. Hence $\gamma$ is a $C^1$ (possibly not regular or simple) curve in $\partial M$ from $x$ to $y$ with the induced parametrization from $\widetilde{\gamma}$. Note that $\gamma$ may not be differentiable with respect to its own arclength parameter. 

As a consequence, the length $L(\gamma)$ of $\gamma$ can be written as:
$$L(\gamma)=\int_0^{L(\widetilde{\gamma})} \sqrt{g(\gamma^{\prime}(s),\gamma^{\prime}(s))} \,ds = \int_0^{L(\widetilde{\gamma})} \sqrt{g \big(\widetilde{\gamma}^{\prime}_T(s)|_{\gamma(s)},\widetilde{\gamma}^{\prime}_T(s)|_{\gamma(s)} \big)} \,ds,$$
where $\widetilde{\gamma}^{\prime}_T(s)$ denotes the vector field with constant coefficients in the frame $(\frac{\partial}{\partial x^1},\cdots,$ $\frac{\partial}{\partial x^{n-1}})$, with the coefficients being the tangential components of the tangent vector $\widetilde{\gamma}^{\prime}(s)$ of $\widetilde{\gamma}$.
Note that $\widetilde{\gamma}^{\prime}_T(s)$ is a Jacobi field for the normal coordinate function $\rho(x)$. For every fixed $s$,
by the definition of the second fundamental form (more precisely the shape operator),
$$\frac{\partial}{\partial \rho}\widetilde{g}_{\rho}(\widetilde{\gamma}^{\prime}_T,\widetilde{\gamma}^{\prime}_T) = 2\widetilde{g}_{\rho}(S_{\rho}(\widetilde{\gamma}^{\prime}_T),\widetilde{\gamma}^{\prime}_T),$$
where $\widetilde{g}_{\rho}$ and $S_{\rho}$ denote the metric and the shape operator of the equidistant hypersurface from $\partial M$ (in $\widetilde{M}-M$) with distance $|\rho|$ (i.e. the level set $\widetilde{d}(\cdot,\partial M)=|\rho|$). Observe that Lemma \ref{riccati} holds in the boundary normal neighborhood of $\partial M$ regardless of which side the neighborhood extends to, thanks to Lemma \ref{extensionmetric}(3). Then the first part of Lemma \ref{riccati} indicates that for sufficiently small $|\rho|$ depending only on $K_1,i_0$,
$$\big|\frac{\partial}{\partial \rho}\widetilde{g}_{\rho}(\widetilde{\gamma}^{\prime}_T,\widetilde{\gamma}^{\prime}_T) \big| \leqslant 4K_1\widetilde{g}_{\rho}(\widetilde{\gamma}^{\prime}_T,\widetilde{\gamma}^{\prime}_T).$$
Thus by Gronwall's inequality, we have
$$g(\widetilde{\gamma}^{\prime}_T|_{\gamma},\widetilde{\gamma}^{\prime}_T|_{\gamma}) \leqslant \widetilde{g}_{\rho}(\widetilde{\gamma}^{\prime}_T,\widetilde{\gamma}^{\prime}_T) e^{4K_1|\rho|}.$$
Since the extended metric $\widetilde{g}$ is a product metric in the boundary normal coordinate, then $\widetilde{g}(\widetilde{\gamma}^{\prime}_T|_{\widetilde{\gamma}},\widetilde{\gamma}^{\prime}_T|_{\widetilde{\gamma}})\leqslant \widetilde{g}(\widetilde{\gamma}^{\prime},\widetilde{\gamma}^{\prime})$. Hence for sufficiently small $\delta_{ex}$ depending only on $K_1$ and $|\rho|\leqslant \delta_{ex}$, we obtain
\begin{eqnarray*}
L(\gamma) &\leqslant& e^{2K_1|\rho|} \int_0^{L(\widetilde{\gamma})} \sqrt{\widetilde{g}_{\rho}(\widetilde{\gamma}^{\prime}_T(s),\widetilde{\gamma}^{\prime}_T(s))}\, ds\\
&\leqslant& e^{2K_1 \delta_{ex}} \int_0^{L(\widetilde{\gamma})} \sqrt{\widetilde{g}(\widetilde{\gamma}^{\prime}(s),\widetilde{\gamma}^{\prime}(s))} \, ds \leqslant (1+3K_1\delta_{ex})\widetilde{d}(x,y),
\end{eqnarray*}
which yields the second inequality by definition.

In general, if $\widetilde{\gamma}$ crosses $\partial M$ with both endpoints in $M$, we can divide $\widetilde{\gamma}$ into segments in $M$ and segments in $\widetilde{M}-M$. The lemma is trivially satisfied for the endpoints of any segment in $M$. Any (continuous) segment in $\widetilde{M}-M$ has endpoints on $\partial M$ and lies entirely in $\widetilde{M}-\textrm{int}(M)$. Thus we apply the argument above for every segment in $\widetilde{M}-M$ and the estimate follows. Finally, if the endpoints of $\widetilde{\gamma}$ are not both in $M$, then its projection $\gamma$ is a curve between the projections of the endpoints of $\widetilde{\gamma}$ onto $M$. This concludes the proof for the first part of the lemma.

Now we prove the second part of the lemma. Let $\widetilde{\gamma}$ be the minimizing geodesic of $\widetilde{M}$ from $x$ to $y$ lying in the boundary normal tubular neighborhood of $\partial M$. If $\widetilde{\gamma}$ lies entirely in $M$ or $\widetilde{M}-\textrm{int}(M)$, one can use the previous argument to project $\widetilde{\gamma}$ to a curve on $\partial M$ and show the same estimate as the first part. The only difference is that when $x,y$ are not in $\partial M$, the projection $\gamma$ is a curve on $\partial M$ from $x^{\perp}$ to $y^{\perp}$. In general, the estimate follows from dividing $\widetilde{\gamma}$ into segments in $M$ and in $\widetilde{M}-M$, and projecting both types of segments onto $\partial M$.
\end{proof}

\begin{definition}\label{Mhdh}
For $h<i_0/2$, we consider the submanifold
$$M_{h}=\big\{x\in M: d(x, \partial M)\geqslant h \big\}.$$
Denote by $d_h:M_h\times M_h \to \mathbb{R}$ the intrinsic distance function of the submanifold $M_h$, and we extend it to any point $x\in \widetilde{M}-M_h$ by 
\begin{equation}\label{dh}
d_h (x,z)=d_h (x^{\perp_h},z)+h^{-1}\widetilde{d}(x,x^{\perp_h}), \textrm{ for }z\in M_h,\, x\in \widetilde{M}-M_h,
\end{equation}
where $x^{\perp_h}\in \partial M_h$ is the unique normal projection of $x\in \widetilde{M}-M_h$ onto $\partial M_h$ within the boundary normal neighborhood of $\partial M$ such that $\widetilde{d}(x,x^{\perp_h})=\widetilde{d}(x,\partial M_h)$. In this definition we require at least one of the points belongs to $M_h$. Note that a similar notation $x^{\perp}$ denotes the normal projection of $x$ onto $\partial M$.
\end{definition}

Thus the path between $z\in M_h$ and a point $x\in \widetilde{M}-M_h$ realizing $d_h(x,z)$ is a broken curve consisting of a geodesic of $M_h$ and a vertical line of the boundary neighborhood (see Figure \ref{figure0}). 

In general, the intrinsic distance function of a manifold with boundary is at most of $C^{1,1}$: the function $d_h(\cdot,z)$ is at most of $C^{1,1}$ even on $M_h-\{z\}$. We need to smoothen it in order to match the $C^{2,1}$ regularity required by Theorem \ref{global}.

\begin{definition}\label{definition-dhs}
For a fixed $z\in M_h$ and any $x\in M$, we denote by $d_h^s (x,z)$ the smoothening of $d_h(x,z)$ via convolution in a ball of radius $r<\delta_{ex}/2$ around the center $x$ with respect to the distance $\widetilde{d}$ of $\widetilde{M}$. More precisely,
\begin{equation}\label{dhsdef}
d_h^s(x,z)=c_{n}r^{-n}\int_{\widetilde{M}}k_1\big(\frac{\widetilde{d}(y,x)}{r}\big)d_h(y,z)dy,
\end{equation}
where $k_1:\mathbb{R}\to \mathbb{R}$ is a nonnegative smooth mollifier supported on $[1/2,1]$, and $dy$ denotes the Riemannian volume form on $\widetilde{M}$. The constant $c_n$ is the normalization constant such that 
\begin{equation}\label{normalization}
c_{n}r^{-n}\int_{\mathbb{R}^n} k_1\big(\frac{|v|}{r}\big)dv=1,
\end{equation}
where $dv$ denotes the Euclidean volume form on $\mathbb{R}^n$.
\end{definition}

\begin{dhsC21}\label{dhsC21}
Let $\delta_{ex}$ be sufficiently small determined in Lemma \ref{extensionmetric}. For sufficiently small $r$ depending on $n,K_1,K_2,i_0,r_0,r_g$, the function $d_h^s(\cdot,z)$ is of $C^{2,1}$ on $M$ for any fixed $z\in M_h$. Furthermore, in the coordinates of our choice, the $C^{2,1}$-norm of $d_h^s(\cdot,z)$ is uniformly bounded explicitly depending on $r,n,\|R_M\|_{C^1}$.
\end{dhsC21}
\begin{proof}
By Lemma \ref{extensionmetric}(4), for sufficiently small $\delta_{ex}$, we know $r_{\textrm{CAT}}(\widetilde{M})$ is bounded below by $C(K_1,i_0,r_0)$. We restrict the smoothening radius to be less than this lower bound: $r<C(K_1,i_0,r_0)$. Then for any $y\in \widetilde{B}_r(x)$, there is a unique minimizing geodesic between $x$ and $y$. Furthermore, no conjugate points occur along geodesics of length less than $\pi/2K_1$ (Corollary 3 in \cite{ABB2}). Since $\widetilde{B}_r(x)\cap \partial \widetilde{M}=\emptyset$ for any $x\in M$ as $r<\delta_{ex}/2$, then $\widetilde{d}(\cdot,x)$ is simply a geodesic distance function in the ball of the smoothening radius around any $x\in M$. As a consequence, $\widetilde{d}(\cdot,x)$ is differentiable on $\widetilde{B}_r(x)$ and $|\nabla \widetilde{d}(\cdot,x)|=1$.

By our choice of coordinate charts in Section \ref{subsection-extension}, for any $x^{\prime}\in \widetilde{M}$, the ball $\widetilde{B}_{r_g/2}(x^{\prime})$ or the cylinder $B_{\partial M}(y,r_g/2)\times (\rho-r_g/2,\rho+r_g/2)$ is contained in at least one of the coordinate charts defined in Lemma \ref{extensionmetric}, where $x^{\prime}=(y,\rho)$ if $x^{\prime}$ is in the boundary normal coordinate of $\partial M$. Then by Lemma \ref{distances}, the ball $\widetilde{B}_{r_g/4}(x^{\prime})$ of $\widetilde{M}$ is contained in one of the coordinate charts if we choose a smaller $r_b$ depending on $K_1$. Hence for $r<r_g/4$, $\widetilde{B}_r(x)$ is contained in one of these coordinate charts for any $x\in M$, and therefore $\widetilde{d}(\cdot,x)$ is of $C^{2,1}$ on $\widetilde{B}_r(x)-\{x\}$ by Lemma \ref{extensionmetric}(2) and Theorem 2.1 in \cite{DK}. Observe that $\widetilde{d}(\cdot,x)$ is bounded below by $r/2$ in the support of $k_1$, which yields a bound on higher derivatives of $\widetilde{d}(\cdot,x)$. This shows that the function $d_h^s(\cdot,z)$ is of $C^{2,1}$.

To estimate the $C^{2,1}$-norm of $d_h^s(\cdot,z)$, it suffices to estimate the $C^{2,1}$-norm of $\widetilde{d}(\cdot,y)$ on the annulus $\widetilde{B}_{r}(y)-\widetilde{B}_{r/2}(y)$. Due to the Hessian comparison theorem (e.g. Theorem 27 in \cite{PP}, p175), for sufficiently small $r$ depending on $K_1$, we have $\|\widetilde{\nabla}^2 \widetilde{d}(\cdot,y) \|\leqslant 4r^{-1}$ on the annulus, where $\widetilde{\nabla}^2$ denotes the second covariant derivative on $\widetilde{M}$. In a local coordinate $(x^1,\cdots,x^n)$ on $\widetilde{M}$, the covariant derivative has the form (e.g. Chapter 2 in \cite{PP}, p32)
\begin{equation}\label{Hessian-local}
\big(\widetilde{\nabla}^2 \widetilde{d}(\cdot,y)\big)(\frac{\partial}{\partial x^k},\frac{\partial}{\partial x^l})=\frac{\partial^2}{\partial x^k \partial x^l} \widetilde{d}(\cdot,y)-\sum_{i=1}^n \widetilde{\Gamma}_{kl}^i \frac{\partial}{\partial x^i} \widetilde{d}(\cdot,y), \quad k,l=1,\cdots,n.
\end{equation}
Hence in the coordinate charts of our choice, for sufficiently small $r$, (\ref{metricboundex}) yields
\begin{equation}\label{dC2}
\|\widetilde{d}(\cdot,y)\|_{C^2} \leqslant C r^{-1},\textrm{ on } \widetilde{B}_{r}(y)-\widetilde{B}_{r/2}(y).
\end{equation}

An estimate on the $C^{2,1}$-norm can be obtained by differentiating the Riccati equation in polar coordinates $\widetilde{g}=dr^2+\widetilde{g}_r$ around $y$, where $\partial_r$ is the radial direction in the geodesic normal coordinate.
Then on the annulus, examining the proof of Lemma 8 in \cite{HV} gives a bound
$$\|\widetilde{\nabla}^3\widetilde{d}(\cdot,y)\|\leqslant C(n, \|R_{\widetilde{M}}\|_{C^1})r^{-2}.$$
Hence by differentiating the formula (\ref{Hessian-local}), for sufficiently small $r$ depending on $n,K_1,K_2,i_0$, we obtain
\begin{equation}\label{dC21}
\|\widetilde{d}(\cdot,y)\|_{C^{2,1}}\leqslant C(n, \|R_{\widetilde{M}}\|_{C^1})r^{-2},\textrm{ on } \widetilde{B}_{r}(y)-\widetilde{B}_{r/2}(y).
\end{equation}
Then a straightforward differentiation yields an estimate on the $C^{2,1}$-norm of $d_h^s(\cdot,z)$.
\end{proof}

\subsection{Proof of Theorem \ref{main1}} \label{subsection3.4} \hfill 

\smallskip
Now we prove the main technical result Theorem \ref{main1}, by constructing the functions and domains assumed in Theorem \ref{global}. The proof consists of several parts.

To begin with, let $h$ be a positive number satisfying $h<\min\{1/5,i_0/10,$ $r_b/10\}$, where $r_b=r_b(K_1,i_0)$ is the width of the boundary normal neighborhood determined in Lemma \ref{riccati}. For sufficiently small $h$ only depending on $n,K_1,K_2,i_0,\vol(\partial M)$, we extend $(M,g)$ to $(\widetilde{M},\widetilde{g})$ with the extension width $\delta_{ex}=5h$ such that Lemma \ref{extensionmetric} holds. Then we extend $u$ to $\widetilde{u}$ by (\ref{extensionu}) with the cut-off width $h$. Let $r_g$ be the uniform radii of $C^1$ geodesic normal coordinates of $M$ and $\partial M$ such that metric bounds (\ref{metricboundex}) hold. We have shown that $r_g$ explicitly depends on $n,\|R_M\|_{C^1},\|S\|_{C^1},i_0$. Now we collect all these relevant parameters and impose the following requirements on the choice of $h$ due to technical reasons:
\begin{equation}
0<h<\min\big\{\frac{1}{10}, \frac{T}{8},\frac{i_0}{10},\frac{r_0}{10},\frac{r_g}{10},\frac{r_b}{10},\frac{i_b(\overline{\Gamma})}{10},\frac{\pi}{12K_1}\big\}.
\end{equation}
The part of the manifold extension over $\Gamma$ is denoted by $\Omega_{\Gamma}=\Gamma\times [-5h,0]$. The number $\min\{1,T^{-1}\}$ will be frequently used in this proof and we denote it by
\begin{equation}\label{aT}
a_T=\min\{1,T^{-1}\}.
\end{equation}

We restrict the choice of $h$ once again, such that for sufficiently small $h$,
\begin{equation}\label{CATchoice}
r_{\textrm{CAT}}(M_h)\geqslant \min\big\{\frac{2}{3}r_0,\frac{\pi}{2K_1}\big\},\quad r_{\textrm{CAT}}(\widetilde{M})\geqslant \min\big\{\frac{2}{3}r_0,\frac{\pi}{2K_1}\big\}.
\end{equation}
This is possible due to Lemma \ref{CATradius}. We remark that the dependency of $h$ is not explicit in Lemma \ref{CATradius}(3), and one can instead use the explicit lower bound in Lemma \ref{CATradius}(2).

With the choice of $\delta_{ex}=5h$ and $h$ as above, the function $d_h(\cdot,z)$ defined in (\ref{dh}) is Lipschitz with a Lipschitz constant $2h^{-1}$ (Lemma \ref{dhs}(3)). In Definition \ref{definition-dhs}, we set the smoothening radius to be $r=a_T h^3$. Then it follows that $|d_h^s(x,z)-d_h(x,z)| < 2a_T h^2$ for any $x\in M$ (Lemma \ref{dhs}(4)).

Assume $h$ is sufficiently small so that Lemma \ref{dhsC21} holds. For any $z\in M_h$ and $x\in M$ satisfying $h/4\leqslant d_h (x,z)\leqslant \min\{i_0/2,r_0/2,\pi/6K_1\}$, we have $|\nabla_x d_h^s(x,z)|> 1-2h$ (Lemma \ref{dd}). Outside the injectivity radius this gradient can be 0 if cut points are involved. This lower bound being close to 1 is crucial for our method to ensure no loss of domain, and we define $d_h$ (\ref{dh}) with the $h^{-1}$ scaling in the boundary neighborhood specifically to guarantee it. While this lower bound is almost trivial when $z$ is far from $\partial M_h$, careful treatment is required when the manifold boundary is involved. 

\smallskip
For $|b|\leqslant 5h$, we define the following set:
\begin{equation}\label{Gammabh}
\Gamma_{b}(h)=\big\{x\in \widetilde{M}: \rho(x)=b, \, x^{\perp}\in \Gamma, \, d_{\partial M}(x^{\perp},\partial \Gamma)\geqslant h \big\},
\end{equation}
where $\partial \Gamma$ denotes the boundary of $\Gamma$ in $\partial M$. The function $\rho(x)$ is the coordinate function in the normal direction defined in (\ref{def-rhox}). Note that if $\Gamma=\partial M$, the last two conditions above automatically satisfy, and then the set above is simply the level set of the normal coordinate function.

Recall that $\widetilde{u}$, the extension of $u$ to $\widetilde{M}$ defined by (\ref{extensionu}), vanishes on $\Gamma_b(0)$ for all $b\leqslant -h$. The set $\Gamma_{-2h}(0)$ is the set from which we intend to propagate the unique continuation. More precisely, we start the propagation from an $h$-net in $\Gamma_{-2h}(8h)$. The reason of this specific choice is the following.
\begin{sublemmainitial}\label{sublemmainitial}
For sufficiently small $h$ only depending on $K_1$, we have
$$\widetilde{d}\big(z,\partial (M\cup \Omega_{\Gamma})-\partial\widetilde{M} \big)\geqslant 7h, \textrm{ for any }z\in \Gamma_{-2h}(8h),$$
where $\Omega_{\Gamma}=\Gamma\times [-5h,0]$ is the part of the manifold extension over $\Gamma$.
\end{sublemmainitial}
\begin{proof}
Let $y$ be a point in $\partial (M\cup \Omega_{\Gamma})-\partial\widetilde{M}$ realizing the distance to $z$. Suppose $\widetilde{d}(z,y)<7h$. Then the minimizing geodesic of $\widetilde{M}$ from $z$ to $y$ lies in the boundary normal (tubular) neighborhood of $\partial M$ of width $5h$. Hence Lemma \ref{distances} implies that
$$d_{\partial M}(z^{\perp},y^{\perp})\leqslant (1+15K_1 h)\widetilde{d}(z,y) < 7h(1+15K_1h)\, .$$
However, we know $d_{\partial M}(z^{\perp},y^{\perp})\geqslant 8h$ by the definition (\ref{Gammabh}). Hence we get a contradiction for sufficiently small $h$ only depending on $K_1$.
\end{proof}

\begin{figure}[h]
\includegraphics[scale=0.45]{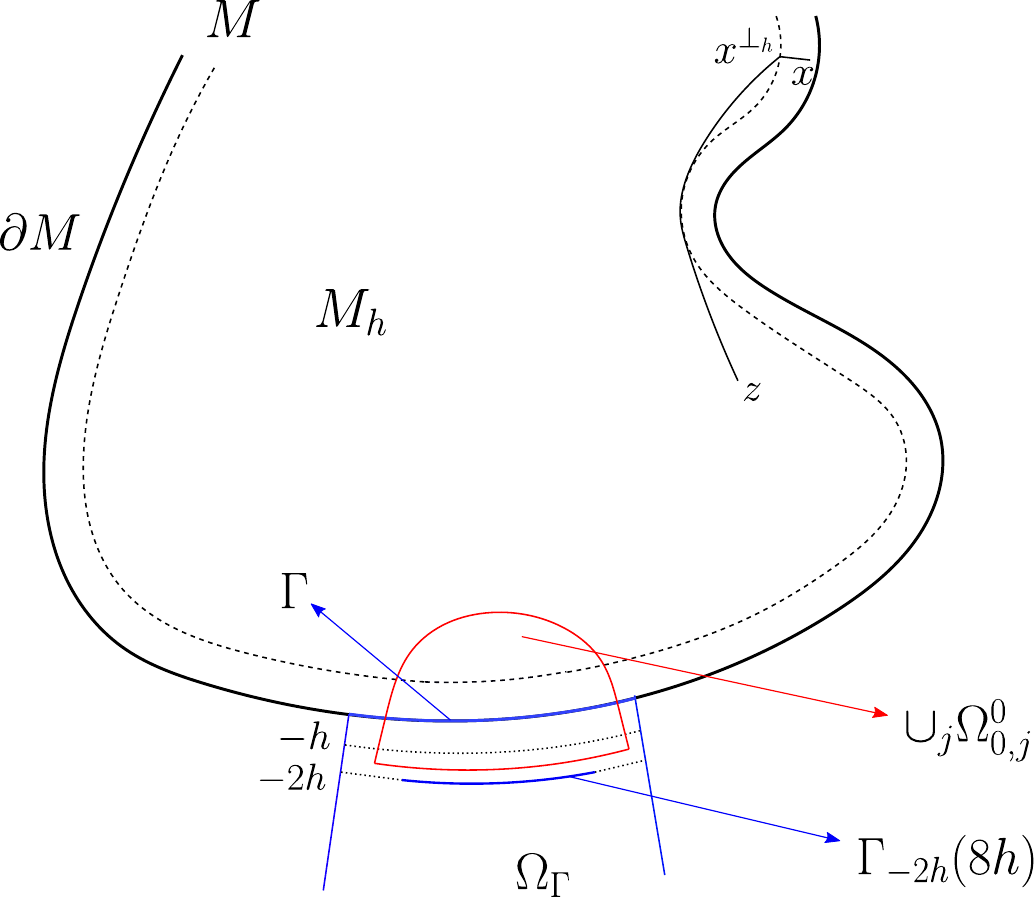}
\caption{Domains for the initial step. Enclosed by the red solid line is the domain we work in, and it is close to $\Gamma$.}
\label{figure0}
\end{figure}

\smallskip
\textbf{Initial Step.} As the initial step, we propagate the unique continuation from outside the manifold $M$ to a region close to $\Gamma$ in $M$.

Consider the function $\xi:[0,+\infty)\to\mathbb{R}$ defined by
\begin{equation}\label{xidef}
\xi(x)=\frac{(h-x)^3}{h^3}, \textrm{ for } x\in [0,h],
\end{equation}
and $\xi(x)=0$ for $x>h$. The function $\xi(x)$ on negative numbers can be defined in any way so that $\xi(x)\geq 1$ for $x<0$, and $\xi(x)$ is smooth on $(-\infty,h)$. 
The function $\xi(x)$ is of $C^{2,1}$ on $\mathbb{R}$ and monotone decreasing on $[0,+\infty)$. Let $\{z_{0,j}\}_{j=1}^{J(0)}$ be an $h$-net in $\Gamma_{-2h}(8h)$: that is, for any $z\in \Gamma_{-2h}(8h)$, there exists some $z_{0,j}$ such that $\widetilde{d}(z,z_{0,j})<h$. We define
\begin{equation}\label{psi0}
\psi_{0,j}(x,t)=\bigg(\Big(1-\xi \big(6h-\widetilde{d}(x,z_{0,j})\big)\Big)T-\widetilde{d}(x,z_{0,j})\bigg)^2-t^2,
\end{equation}
and consider the following domains (see Figure \ref{figure0}):
\begin{equation}\label{Omega00}
\Omega^0_{0,j} =\big\{(x,t)\in \widetilde{M}\times[-T,T]: \psi_{0,j}(x,t) > h^2,\; \rho(x)>-\frac{3}{2}h \big\}.
\end{equation}
Note that in general, the domain characterized by $\psi_{0,j}(x,t) > h^2$ has two connected components. Here we define $\Omega^0_{0,j}$ to be the connected component characterized by $\big(1-\xi(6h-\widetilde{d}(x,z_{0,j}))\big)T-\widetilde{d}(x,z_{0,j})>0$.\footnote{\,Throughout the proof, whenever we define a domain using level sets of a similar function, we exactly mean this one type of connected component.}
Observe that in this connected component, it satisfies that $\widetilde{d}(x,z_{0,j})<6h$ due to the definition of the function $\xi$.

Then we define
\begin{equation}\label{Upsilon}
\Upsilon=\{x\in \Omega_{\Gamma}: -2h\leqslant \rho(x)\leqslant -h\}\times [-T,T],
\end{equation} 
and
\begin{equation}\label{Omega0j}
\Omega_{0,j} =\big\{(x,t)\in  \Omega^0_{0,j}-\Upsilon: \psi_{0,j}(x,t) > 4h^2 \big\}.
\end{equation} 

Now we prove that the conditions assumed in Theorem \ref{global} satisfy for $\psi_{0,j}$, $\Omega_{0,j}^0$, $\Omega_{0,j}$, $\Upsilon$, $\psi_{max,0}=(T-h)^2$, and therefore Theorem \ref{global} applies. A stability estimate will be derived at the end of the proof.

\medskip
\noindent (1) We show that $\psi_{0,j}$ is of $C^{2,1}$ and non-characteristic in $\Omega^0_{0,j}$. Indeed, for any $(x,t)\in \Omega^0_{0,j}$, we have $\widetilde{d}(x,z_{0,j})< 6h$ by the definition of $\psi_{0,j}$. Hence any minimizing geodesic of $\widetilde{M}$ from $z_{0,j}$ to $x$ must not intersect $\partial \widetilde{M}$; otherwise the length of such geodesic would exceed $6h$ due to the condition that $\rho(x)>-3h/2$. Furthermore, by our choice $h<\min\{r_0/10,\pi/12K_1\}$ and (\ref{CATchoice}), the minimizing geodesic from $z_{0,j}$ to any $x\in \widetilde{B}_{6h}(z_{0,j})$ is unique and no conjugate points can occur. Therefore $\widetilde{d}(\cdot,z_{0,j})$ is a $C^{2,1}$ geodesic distance function in $\Omega^0_{0,j}$, which shows that $\psi_{0,j}$ is of $C^{2,1}$ in $\Omega^0_{0,j}$. Moreover, since $\widetilde{d}(x,z_{0,j})>h/2$ for any $(x,t)\in \Omega^0_{0,j}$ by definition, the $C^{2,1}$-norms of $\widetilde{d}(\cdot,z_{0,j})$ and $\psi_{0,j}$ are uniformly bounded in $\Omega^0_{0,j}$ due to (\ref{dC21}).

Next we prove that $\psi_{0,j}$ is non-characteristic in $\Omega^0_{0,j}$. For any $(x,t)\in \Omega^0_{0,j}$,
$$\nabla_x \psi_{0,j}=2\Big(\big(1-\xi(6h-\widetilde{d}(x,z_{0,j}))\big)T-\widetilde{d}(x,z_{0,j})\Big) \big(\xi^{\prime}T\nabla_x\widetilde{d}(x,z_{0,j})-\nabla_x\widetilde{d}(x,z_{0,j}) \big).$$
Note that $\xi^{\prime}$ is evaluated at $6h-\widetilde{d}(x,z_{0,j})$ in the formula above. Since $\xi^{\prime}\leqslant 0$, then
$$\big|\xi^{\prime}T\nabla_x\widetilde{d}(x,z_{0,j})-\nabla_x\widetilde{d}(x,z_{0,j}) \big|\geqslant  |\nabla_x\widetilde{d}(x,z_{0,j})|=1.$$
Hence,
\begin{eqnarray*}
p \big((x,t),\nabla \psi_{0,j} \big) &=& \sum_{k,l=1}^n \widetilde{g}^{kl} (\partial_{x_k}\psi_{0,j})( \partial_{x_l}\psi_{0,j})-|\partial_t\psi_j|^2 =|\nabla_x \psi_{0,j}|^2-|\partial_t \psi_{0,j}|^2 \nonumber\\
&\geqslant& 4\Big(\big(1-\xi(6h-\widetilde{d}(x,z_{0,j}))\big)T-\widetilde{d}(x,z_{0,j})\Big)^2-4t^2 \nonumber \\
&=& 4\psi_{0,j}^2(x,t) > 4h^2.
\end{eqnarray*}

\noindent (2) The extended function $\widetilde{u}$ defined by (\ref{extensionu}) vanishes on $\Upsilon$. We claim that $\emptyset\neq\{(x,t)\in \Omega^0_{0,j}: \psi_{0,j}(x,t) > (T-h)^2\}\subset \Upsilon$. Indeed, for any $(x,t)$ in the set, it satisfies that $\widetilde{d}(x,z_{0,j})<h$, which indicates $\rho(x)<-h$. On the other hand, Sublemma \ref{sublemmainitial} implies that $x\in \Omega_{\Gamma}$, and therefore $(x,t)\in \Upsilon$. For the non-emptyness, consider the point $x_j\in \Gamma_{-5h/4}(0)$ such that $\widetilde{d}(x_j,z_{0,j})=3h/4$ (i.e. $x_j$ is the projection of $z_{0,j}$ onto $\Gamma_{-5h/4}(0)$). By definition, we have $\psi_{0,j}(x_j,0)=(T-3h/4)^2>(T-h)^2$. This also shows that $(x_j,0)\in \Omega_{0,j}^0$ by definition when $T>2h$, which yields the non-emptyness.

\noindent (3) We show that $dist_{\widetilde{M}\times \mathbb{R}}(\partial \Omega^0_{0,j},\Omega_{0,j})>0$. It suffices to prove $\overline{\Omega}_{0,j}\subset \Omega^0_{0,j}$. For any $(x,t)\in \Omega^0_{0,j}$, we have $\widetilde{d}(x,z_{0,j})< 6h$ by the definition of $\psi_{0,j}$, which implies that $\Omega^0_{0,j}\subset M\cup \Omega_{\Gamma}$ due to Sublemma \ref{sublemmainitial}. 
This indicates that the boundaries of $\Omega_{0,j}^0,\Omega_{0,j}$ are determined only by $\psi_{0,j}$ and $\rho(x)$. Since we know $\rho(x)>-h$ for any $(x,t)\in \Omega_{0,j}$ by definition, then clearly $\overline{\Omega}_{0,j}\subset \Omega^0_{0,j}$.

\noindent (4) We claim that $\cup_{j=1}^{J(0)} \Omega_{0,j}$ is connected and therefore its closure is connected. Take two reference points $z_{0,j_1},z_{0,j_2}$ satisfying $\widetilde{d}(z_{0,j_1},z_{0,j_2})<3h$. Consider $(z_{0,j_1}^{\perp},0)\in \partial M\times [-T,T]$. Directly checked by the definition of $\Omega_{0,j}$, this point $(z_{0,j_1}^{\perp},0)$ is in both $\Omega_{0,j_1}$ and $\Omega_{0,j_2}$. 
In particular, this shows $\Omega_{0,j_1}\cap\Omega_{0,j_2}\neq\emptyset$ if $\widetilde{d}(z_{0,j_1},z_{0,j_2})<3h$. Since each $\Omega_{0,j}$ is path connected, so is $\Omega_{0,j_1}\cup\Omega_{0,j_2}$. The claim follows from the fact that for any two points in the $h$-net $\{z_{0,j}\}$, we can find a chain of $\{z_{0,j}\}$ such that every pair of adjacent points in this chain has distance less than $3h$. 

\medskip
In order to propagate further in subsequent steps, we need to estimate how much $\cup_{j} \Omega_{0,j}$ covers in the original manifold $M$.
\begin{sublemmainitial2}\label{sublemmainitial2}
$\big(\bigcup_{b\in [0,2h]}\Gamma_b(8h)\big)\times [-T+6h,T-6h] \subset \bigcup_{j=1}^{J(0)} \Omega_{0,j}.$
\end{sublemmainitial2}
\begin{proof}
For any $(x,t)$ in the left-hand set, there exists $j_0$ such that $\widetilde{d}(x,z_{0,j_0})< 5h$ due to the definition of $h$-net, which indicates that the $\xi$ term in $\psi_{0,j_0}$ (\ref{psi0}) vanishes. Thus
$$\psi_{0,j_0}(x,t)=(T-\widetilde{d}(x,z_{0,j_0}))^2-t^2 > (T-5h)^2-(T-6h)^2>5h^2,$$
where we used $T>8h$. This shows that $(x,t)$ is in both $\Omega^0_{0,j_0}$ and $\Omega_{0,j_0}$.
\end{proof}

\medskip
\textbf{Subsequent Steps.}
After the initial step, the reference set is moved to $\Gamma_{h}(8h)$ and unique continuation is propagated up to $\Gamma_{2h}(8h)$. Let $\{z_{1,j}\}$ be an $h$-net in $\Gamma_h(10h)\subset M_h$ with respect to $d_h$. Note that here the range of the $j$ index is different from that of the $j$ index in the initial step, and a precise notation would be $\{z_{1,j}\}_{j=1}^{J(1)}$. We omit this dependence on the step number to keep the notations short. Set $T_1=T-6h$ and $\rho_0=\min\{i_0/2,r_0/2,r_g/4,\pi/6K_1\}$. We divide into Case 1 and Case 2 depending on if $T$ is larger than $\rho_0$.

\begin{figure}[h]
\includegraphics[scale=0.45]{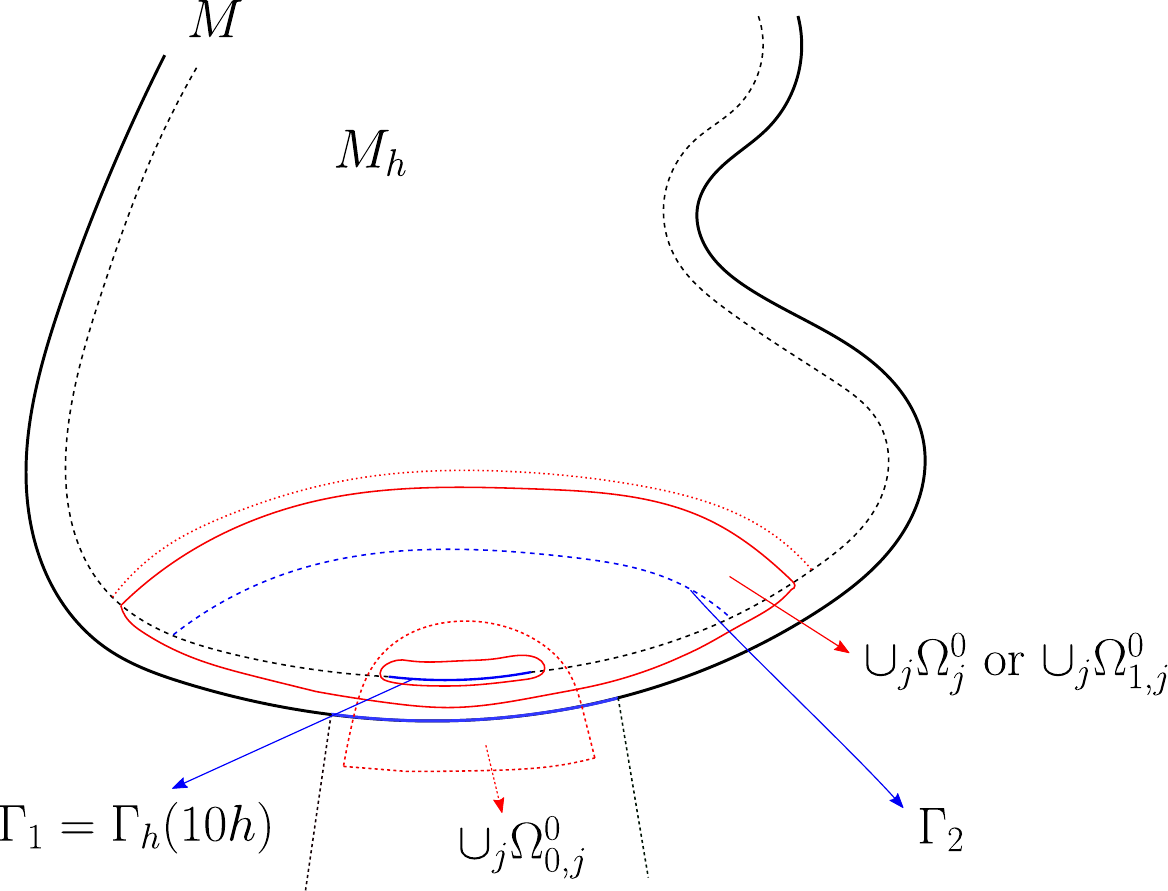}
\caption{Domains for Case 1 or the first step in Case 2. Enclosed by the red solid lines is the domain we work in, and its boundary consists of two disjoint parts. This domain never reaches outside distance $\rho_0$, which is marked by the upper red dotted line.
 The blue dashed line $\Gamma_2$ is the reference set for the second step in Case 2.}
\label{figure1}
\end{figure}

\smallskip
\textbf{Case 1:} $T\leqslant \rho_0=\min\big\{i_0/2,r_0/2,r_g/4,\pi/6K_1\big\}$.

\smallskip
For any $(x,t)\in M\times [-T_1,T_1]$, we define the following $C^{2,1}$ functions
\begin{equation}\label{psi}
\psi_j(x,t)=\bigg(\Big(1-\xi\big(d(x,\partial M)\big)\Big)T_1-d_h^s(x,z_{1,j})\bigg)^2-t^2,
\end{equation}
and consider the domains\footnote{\,the connected component characterized by $\big(1-\xi(d(x,\partial M))\big)T_1-d_h^s(x,z_{1,j})>0$.}
\begin{equation}\label{Omega0}
\Omega^0_{j} =\big\{(x,t)\in M\times [-T_1,T_1]: \psi_j(x,t) > 8T^2 h \big\}-\{x: d_h^s(x,z_{1,j})\leqslant \frac{h}{2}\}\times [-T_1,T_1].
\end{equation}
Observe that $\xi(d(x,\partial M))<1$ in $\Omega_j^0$ and hence $\Omega_j^0$ never intersect with $\partial M$ at any time. For any $(x,t)\in \Omega_{j}^0$, we have $h/2<d_h^s(x,z_{1,j})<T_1\leqslant \rho_0-6h$ by definition. Then Lemma \ref{dhs}(4) indicates that $h/4<d_h(x,z_{1,j})<\min\{i_0/2,r_0/2,\pi/6K_1\}$, and hence Lemma \ref{dd} applies.

Then we define
\begin{equation}\label{Omegajcase1}
\Omega_j=\big\{(x,t)\in \Omega^0_j-\cup_j \overline{\Omega}_{0,j}: \psi_j(x,t) > 9T^2 h \big\}.
\end{equation}

Now we prove that the conditions assumed in Theorem \ref{global} satisfy for $\psi_{j}$, $\Omega_{j}^0$, $\Omega_{j}$, $\psi_{max}=(T_1-3h/4)^2$, together with relevant functions and domains in the initial step. The relevant domains are illustrated in Figure \ref{figure1}.

\medskip
First we show that $\psi_j$ is non-characteristic at any $(x,t)\in \Omega^0_j$. For $x\in M-M_h$,
$$\nabla_x\psi_j=2\Big( \big(1-\xi(d(x,\partial M))\big)T_1-d_h^s(x,z_{1,j})\Big)\big(-\xi^{\prime}T_1\nabla_x d(x,\partial M) - \nabla_x d_h^s(x,z_{1,j})\big).$$
Note that $\xi^{\prime}$ is evaluated at $d(x,\partial M)$ in the formula above.

For $x\in M-M_h$ with $d(x,\partial M_h)\geqslant a_T h^3$, the vectors $\nabla_x d_h(x,z_{1,j})$ and $\nabla_x d_h^s(x,z_{1,j})$ only differ by a small component $C(n,K_1,K_2) h^2$ due to (\ref{bdhcloseness}). In particular, $\langle \nabla_x d_h(x,z_{1,j}),$ $\nabla_x d_h^s(x,z_{1,j})\rangle >0$ for sufficiently small $h$ depending on $n,K_1,K_2$. Hence by the definition of $d_h$ (\ref{dh}),
\begin{equation*}\label{opposite}
\langle \nabla_x d(x,\partial M),\nabla_x d_h^s(x,z_{1,j})\rangle=-h\langle \nabla_x d_h(x,z_{1,j}),\nabla_x d_h^s(x,z_{1,j})\rangle<0.
\end{equation*}
Then by Lemma \ref{dd} and $\xi^{\prime}\leqslant 0$, we have
$$\big|-\xi^{\prime}T_1\nabla_x d(x,\partial M) - \nabla_x d_h^s(x,z_{1,j})\big|\geqslant |\nabla_x d_h^s(x,z_{1,j})|>1-2h.$$

For $x\in M-M_h$ with $d(x,\partial M_h)< a_T h^3$, we have $|\xi^{\prime}(d(x,\partial M))|< 3a_T^2 h^{3}\leqslant 3T^{-1} h^3$ at such points by definitions (\ref{xidef}) and (\ref{aT}). Therefore for any $x\in M-M_h$ and sufficiently small $h$, we have
\begin{eqnarray}\label{dpsi}
|\nabla_x\psi_j|&>& 2|(1-\xi)T_1-d_h^s|(1-2h-3h^3) \nonumber \\
&>& 2|(1-\xi)T_1-d_h^s|(1-3h).
\end{eqnarray}
On the other hand, if $x\in M_h$, then $\xi$ term vanishes and the estimate above holds. Hence for any $(x,t)\in\Omega^0_j$,
\begin{eqnarray}\label{ppsi}
p \big((x,t),\nabla \psi_j \big) &=&|\nabla_x \psi_j|^2-|\partial_t \psi_j|^2 \nonumber\\
&>& 4\big((1-\xi)T_1-d_h^s \big)^2(1-3h)^2-4t^2 \nonumber \\
&>& 4\psi_j(x,t) - 24T^2 h > 8T^2 h.
\end{eqnarray}
This shows that $\psi_j$ is non-characteristic at any $(x,t)\in\Omega^0_j$. 

It is straightforward to show the connectedness of $(\cup_j \overline{\Omega}_j) \cup (\cup_j\overline{\Omega}_{0,j})$ in the same way as we did for $\cup_j\Omega_{0,j}$ in the initial step. The other conditions assumed in Theorem \ref{global} follow from Sublemma \ref{sublemma0} below and Sublemma \ref{sublemmainitial2}.

\begin{sublemma0}\label{sublemma0}
For sufficiently small $h<1/8$ depending on $K_1$, we have
$$\emptyset\neq \big\{(x,t)\in \Omega_j^0: \psi_j(x,t)>(T_1-\frac{3}{4}h)^2 \big\}\subset \big(\bigcup_{b\in [0,2h]}\Gamma_b(8h)\big)\times [-T_1,T_1],$$
and $dist_{\widetilde{M}\times \mathbb{R}}(\partial \Omega^0_j,\Omega_j)>0$.
\end{sublemma0}
\begin{proof}
The non-emptyness follows from definition. For any $(x,t)$ in the left-hand set, we know $d_h^s(x,z_{1,j})<3h/4$ by definition. Hence it suffices to show that
\begin{equation}\label{sublemma0d}
\big\{x: d_h^s(x,z_{1,j})\leqslant \frac{3}{4}h \big\}\subset \bigcup_{b\in [0,2h]}\Gamma_b(8h).
\end{equation}

For any $x$ in the left-side set in (\ref{sublemma0d}), Lemma \ref{dhs}(4) indicates that $d_h(x,z_{1,j})<h$ and hence $\rho(x)<2h$. This checks the condition on $\rho(x)$ in (\ref{Gammabh}). We proceed to check the rest of the conditions in (\ref{Gammabh}).

If $x\in M_h$, then by Lemma \ref{distances},
\begin{eqnarray*}
d_{\partial M}(x^{\perp},z_{1,j}^{\perp})&\leqslant& (1+15K_1 h)\widetilde{d}(x,z_{1,j}) \\
&\leqslant& (1+15K_1 h)d_h(x,z_{1,j})<h(1+15K_1 h)\, .
\end{eqnarray*}

If $x\in M-M_h$, then $d_h(x^{\perp_h},z_{1,j})<d_h(x,z_{1,j})<h$ by definition (\ref{dh}). Hence,
\begin{eqnarray*}
d_{\partial M}(x^{\perp},z_{1,j}^{\perp}) &=&d_{\partial M}((x^{\perp_h})^{\perp},z_{1,j}^{\perp}) \\
&\leqslant& (1+15K_1 h)d_h(x^{\perp_h},z_{1,j})<h(1+15K_1 h)\, ,
\end{eqnarray*}
where we used the fact that $(x^{\perp_h})^{\perp}=x^{\perp}$.

Therefore in either case, for sufficiently small $h$ depending only on $K_1$, we have $d_{\partial M}(x^{\perp},z_{1,j}^{\perp})$ $<2h$. Then the fact that $d_{\partial M}(z_{1,j}^{\perp},\partial \Gamma)\geqslant 10h$ yields $x^{\perp}\in \Gamma$ and $d_{\partial M}(x^{\perp},\partial \Gamma)> 8h$. This completes the proof of (\ref{sublemma0d}) and consequently the first statement of the sublemma.

For the second statement, it suffices to prove $\overline{\Omega}_j \subset \Omega_j^0$. For any $(x,t)\in \overline{\Omega}_j$, clearly we have $\psi_j(x,t)\geqslant 9T^2 h>8T^2h$ and $(x,t)\notin \cup_{j}\Omega_{0,j}$ by definition (\ref{Omegajcase1}). To show $(x,t)\in \Omega_j^0$, we only need to show $(x,t)\notin \{x:d_h^s(x,z_{1,j})\leqslant h/2\}\times [-T_1,T_1]$. This is a direct consequence of the fact that a larger cylinder $\{x:d_h^s(x,z_{1,j})\leqslant 3h/4\}\times [-T_1,T_1]$ is strictly contained in the open set $\cup_{j}\Omega_{0,j}$, due to (\ref{sublemma0d}) and Sublemma \ref{sublemmainitial2}. An explicit lower bound for the distance between their boundaries is estimated in Lemma \ref{mindistance}.
\end{proof}

\smallskip
\noindent\textbf{Error estimate for Case 1.} We prove that $\overline{\Omega}=(\cup_j \overline{\Omega}_j) \cup (\cup_j\overline{\Omega}_{0,j})$ almost covers the domain of influence in the original manifold $M$. More precisely, we prove that there exists $C^{\prime}=C^{\prime}(T,K_1)$ such that $\Omega(C^{\prime}h)\subset \overline{\Omega}$. Since $\Omega(C^{\prime}h)\subset M\times [-T,T]$, it suffices to show that $M\times [-T,T]-\overline{\Omega}\subset M\times [-T,T]-\Omega(C^{\prime}h)$.

For any $(x,t)\in M\times [-T,T]-\overline{\Omega}$, by the definitions (\ref{psi}), (\ref{Omega0}), (\ref{Omegajcase1}), we know that one of the following two situations must happen:\\
(1) $d(x, \partial M)< h$; \\
(2) $x\in M_h$ and $d_h^s(x,z_{1,j})> T_1-\sqrt{t^2+9T^2 h}$ for any $z_{1,j}$. 

\smallskip
We analyze these two situations separately as follows.

\noindent \textbf{(1)} By virtue of Sublemma \ref{sublemmainitial2} and the definition (\ref{Gammabh}), the situation (1) implies that $x^{\perp}\notin \Gamma$, or $x^{\perp}\in \Gamma$ and $d_{\partial M}(x^{\perp},\partial \Gamma)<8h$, or $|t|> T-6h$. The condition $x^{\perp}\notin \Gamma$ indicates that $d(x,\partial M-\Gamma)<h$. If $x^{\perp}\in \Gamma$ and $d_{\partial M}(x^{\perp},\partial \Gamma)<8h$, then by the triangle inequality,
$$d(x,\partial\Gamma)\leqslant d(x,x^{\perp})+d(x^{\perp},\partial \Gamma)\leqslant h+d_{\partial M}(x^{\perp},\partial \Gamma)<9h,$$
which yields $d(x,\partial M-\Gamma)<9h$ due to $\partial \Gamma\subset \partial M-\Gamma$. If $|t|>T-6h$, then the following inequality is trivially satisfied:
$$T-|t|-\sqrt{6h}<6h-\sqrt{6h}<0\leqslant d(x,\Gamma).$$
Note that if $\Gamma=\partial M$, the first two possibilities automatically do not occur and hence only the last inequality above is valid under the first situation.

\noindent \textbf{(2)} By Lemma \ref{dhs}(4), the situation (2) implies that $d_h(x,z_{1,j})> T_1-|t|-3T\sqrt{h}-2h^2$ for $x\in M_h$ and any $z_{1,j}$. Since $\{z_{1,j}\}$ is an $h$-net in $\Gamma_{h}(10h)$ with respect to $d_h$, we have
$$d_h(x,\Gamma_{h}(10h))> T_1-|t|-3T\sqrt{h}-h-2h^2.$$
Then we apply Lemma \ref{distances} after replacing $M,\widetilde{M}$ with $M_h,M$:
$$d(x,\Gamma_{h}(10h))(1+6K_1 h)\geqslant d_h(x,\Gamma_{h}(10h))>T_1-|t|-3T\sqrt{h}-h-2h^2,$$
where we used the fact that the second fundamental form of $\partial M_h$ is bounded by $2K_1$ due to Lemma \ref{riccati}. Hence by the triangle inequality,
$$d(x,\Gamma_0(10h))>(T_1-|t|-3T\sqrt{h}-h-2h^2)(1+6K_1 h)^{-1}-h.$$

For any $y\in \Gamma-\Gamma_0(10h)$, $y$ lies in the boundary normal neighborhood of $\partial \Gamma$ in $\Gamma$ due to $10h<i_b(\overline{\Gamma})$. Hence $d(y,\Gamma_0(10h))\leqslant d_{\partial M}(y,\Gamma_0(10h))\leqslant 10h$. Then,
\begin{eqnarray*}
d(x,y)&\geqslant& d(x,\Gamma_0(10h))-d(y,\Gamma_0(10h)) \\
&>& (T-|t|-3T\sqrt{h}-7h-2h^2)(1+6K_1 h)^{-1}-11h,
\end{eqnarray*}
where we used $T_1=T-6h$. Hence we arrive at
\begin{equation*}\label{type2}
d(x,\Gamma) > T-|t|-C(T,K_1)\sqrt{h}.
\end{equation*}

Finally we combine these two situations together, and we have proved that $(x,t)\in M\times [-T,T]-\Omega(Ch)$ for $C=\max\{C(T,K_1)^2,9\}$ by definition (\ref{Omegaht}). Therefore, there exists $C^{\prime}=C^{\prime}(T,K_1)$ such that $\Omega(C^{\prime}h)\subset \overline{\Omega}$, and a stability estimate can be obtained on $\Omega(C^{\prime}h)$ from Theorem \ref{global}. The stability estimate will be derived at the end of the proof.

\medskip
\textbf{Case 2:} $T>\rho_0=\min\big\{i_0/2,r_0/2,r_g/4,\pi/6K_1 \big\}$.

\begin{figure}[h]
\includegraphics[scale=0.45]{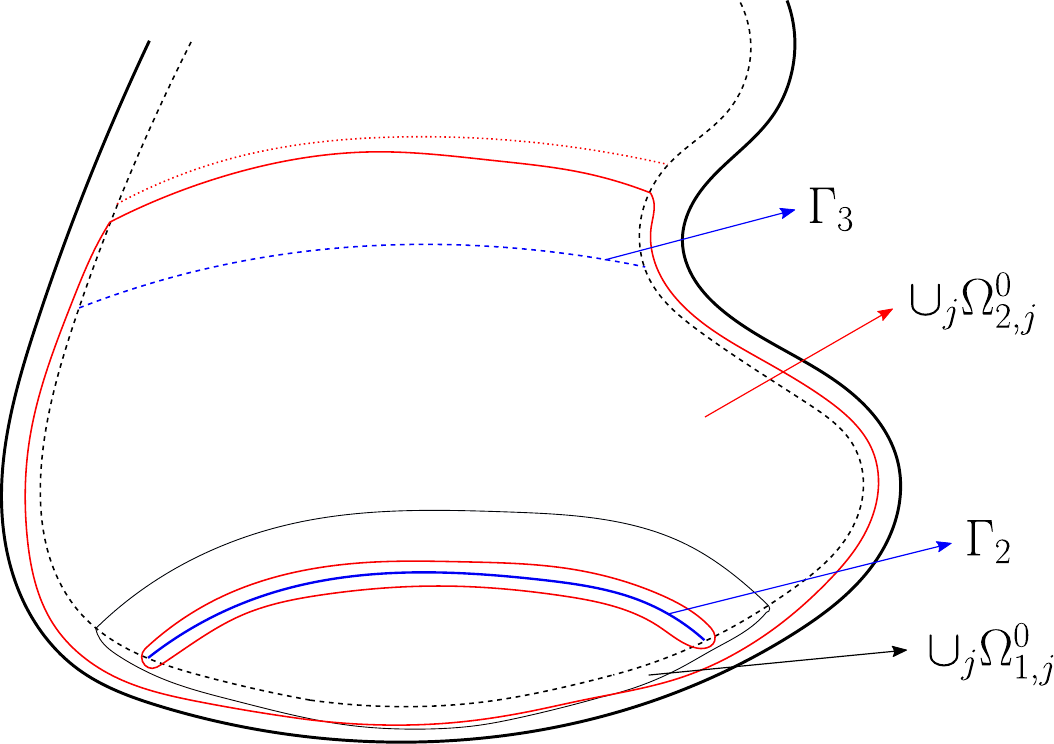}
\caption{Domains for the second step in Case 2. Enclosed by the red solid lines is the domain we work in. The blue dashed line $\Gamma_3$ is the reference set for the third step. From here, the procedure is entirely done in $M$.}
\label{figure2}
\end{figure}

\smallskip
As Lemma \ref{dd} is only valid within the injectivity radius, we define the procedure step by step and each step is done within the injectivity radius. Recall that $\{z_{1,j}\}$ be an $h$-net in $\Gamma_h(10h)\subset M_h$ with respect to $d_h$, and $T_1=T-6h$. For the first step, we define functions $\psi_{1,j}$ by adding to (\ref{psi}) another term associated with $T_1$:
\begin{equation}\label{psi1}
\psi_{1,j}(x,t)=\bigg(\Big(1-\xi \big(d(x,\partial M)\big)-\xi \big(\rho_0-d_h^s(x,z_{1,j})\big)\Big)T_1-d_h^s(x,z_{1,j})\bigg)^2-t^2,
\end{equation}
and consider the domains
\begin{equation}\label{Omega01}
\Omega^0_{1,j} =\big\{(x,t)\in M\times [-T_1,T_1]: \psi_{1,j}(x,t) > 8T^2 h \big\}-\{x: d_h^s(x,z_{1,j})\leqslant \frac{h}{2}\}\times [-T_1,T_1].
\end{equation}
One can compare these definitions here with those in Case 1. Note that the regions $\Omega_{1,j}^0$ stay within half the injectivity radius due to the definition of the function $\xi$. The gradient of $\psi_{1,j}$ has the following form:
\begin{eqnarray*}
\nabla_x\psi_{1,j}&=&2\Big(\big(1-\xi(d(x,\partial M))-\xi(\rho_0-d_h^s(x,z_{1,j}))\big)T_1-d_h^s(x,z_{1,j})\Big) \\
&&\big(-\xi^{\prime}T_1\nabla_x d(x,\partial M) +\xi^{\prime}T_1\nabla_x d_h^s(x,z_{1,j})-\nabla_x d_h^s(x,z_{1,j})\big).
\end{eqnarray*}
The vector part of $\nabla_x \psi_{1,j}$ consists of $\nabla_x d(x,\partial M)$ and $\nabla_x d_h^s(x,z_j)$, the same as in Case 1. Furthermore, the form for the vector part is the same as that in Case 1 up to multiplication by a positive function, since $\xi^{\prime}\leqslant 0$. Hence one obtains the same lower bounds for the length of the gradient and the principle symbol as (\ref{dpsi}) and (\ref{ppsi}). It follows that $\psi_{1,j}$ is non-characteristic in $\Omega^{0}_{1,j}$. And we define $\psi_{max,1}$ and $\Omega_{1,j}$ the same as in Case 1 (see Figure \ref{figure1}). More precisely, define $\psi_{max,1}=(T_1-3h/4)^2$ and
\begin{eqnarray}\label{1other}
\Omega_{1,j}=\big\{(x,t)\in \Omega^0_{1,j}-\cup_j \overline{\Omega}_{0,j}: \psi_{1,j}(x,t) > 9T^2 h \big\}.
\end{eqnarray}
Since (\ref{sublemma0d}) is still valid, Sublemma \ref{sublemma0} holds for $\psi_{1,j},\Omega_{1,j}^0,\Omega_{1,j}$. Hence Theorem \ref{global} applies to the first step. We stop the procedure right after the first step if $T_1-\rho_0-3T\sqrt{h}\leqslant 2h$.

For the second step, we need to choose a new set of reference points. Observe that the first step propagates past the level set $\Gamma_2:=\{x\in M_h: d_h(x,\Gamma_{h}(10h))=\rho_0-4h\}$ due to Lemma \ref{dhs}(4) and the procedure stopping criterion $T_1-\rho_0-3T\sqrt{h}>2h$. We choose the new reference points $\{z_{2,j}\}$ as an $h$-net in $\Gamma_2$ with respect to $d_h$. At $\Gamma_2$, the square of the maximal time allowed is $(T_1-\rho_0+4h)^2-9T^2h$, and we set the time range $T_2$ for the second step as $T_2=T_1-\rho_0-3T\sqrt{h}$. The procedure stopping criterion indicates that $T_2>2h$. Then we define the functions
$$\psi_{2,j}(x,t)=\bigg(\Big(1-\xi \big(d(x,\partial M)\big)-\xi \big(\rho_0-d_h^s(x,z_{2,j})\big)\Big)T_2-d_h^s(x,z_{2,j})\bigg)^2-t^2.$$
To apply Theorem \ref{global}, we need to ensure that small neighborhoods around the new reference points are contained in the regions already propagated by the unique continuation in the first step. To that end, we define $\psi_{max,2}=(T_2-a_T h)^2$, where $a_T=\min\{1,T^{-1}\}$, and
$$\Omega^{0}_{2,j}=\big\{(x,t)\in M\times [-T_2,T_2]: \psi_{2,j}(x,t) > 8T^2h \big\}-\{x: d_h^s(x,z_{2,j})\leqslant \frac{1}{2}a_T h \}\times [-T_2,T_2],$$
$$\Omega_{2,j}=\big\{(x,t)\in\Omega^{0}_{2,j}-\big((\cup_j \overline{\Omega}_{1,j})\cup (\cup_j \overline{\Omega}_{0,j})\big): \psi_{2,j}(x,t) > 9T^2 h \big\}.$$
These domains are illustrated in Figure \ref{figure2}. The specific choice of $\psi_{max,2}$ is justified in Sublemma \ref{sublemma2} a bit later, to ensure that $\emptyset\neq\{(x,t)\in \Omega^{0}_{2,j}: \psi_{2,j}(x,t)>\psi_{max,2}\}\subset (\cup_j \overline{\Omega}_{1,j})\cup (\cup_j \overline{\Omega}_{0,j})$.

\smallskip
Now we define the remaining steps iteratively. We define the reference sets as
$$\Gamma_{i}=\big\{x\in M_h: d_h(x, \Gamma_1)=(i-1)(\rho_0-4h) \big\}, \quad i\geqslant 2,$$
where $\Gamma_1=\Gamma_{h}(10h)\subset M_h$. The reference points $\{z_{i,j}\}$ are defined as an $h$-net in $\Gamma_i$ with respect to $d_h$. Note that the range of $j$ index for each step $i$ is different, and the notation $\{z_{i,j}\}$ here is short for $\{z_{i,j}\}_{j=1}^{J(i)}$.
We define the $C^{2,1}$ functions $\psi_{i,j}$ as follows.
$$\psi_{i,j}(x,t)=\bigg(\Big(1-\xi \big(d(x,\partial M)\big)-\xi \big(\rho_0-d_h^s(x,z_{i,j})\big)\Big)T_{i}-d_h^s(x,z_{i,j})\bigg)^2-t^2,$$
where $T_{i}=T_{i-1}-\rho_0-3T\sqrt{h}$ with $T_1=T-6h$. We stop the procedure at the $i$-th step if $T_{i+1}\leqslant 2h$ or $\Gamma_{i+1}=\emptyset$. The regions $\Omega^{0}_{i,j}$ and $\Omega_{i,j}$ for $i\geqslant 2$ are defined as\footnote{\,the connected component characterized by $\big(1-\xi(d(x,\partial M))-\xi(\rho_0-d_h^s(x,z_{i,j}))\big)T_{i}-d_h^s(x,z_{i,j})>0$.}
$$\Omega^{0}_{i,j}=\big\{(x,t)\in M\times [-T_i,T_i]:\psi_{i,j}(x,t) >8T^2h \big\}-\{x: d_h^s(x,z_{i,j})\leqslant \frac{1}{2}a_T h\}\times [-T_i,T_i].$$
$$\Omega_{i,j}=\big\{(x,t)\in \Omega^{0}_{i,j}-\cup_{l=0}^{i-1}\cup_{j} \overline{\Omega}_{l,j}: \psi_{i,j}(x,t) > 9T^2 h \big\},$$
where $a_T=\min\{1,T^{-1}\}$ in (\ref{aT}). It follows that $\psi_{i,j}$ is non-characteristic in $\Omega^{0}_{i,j}$ in the same way as for $\psi_{1,j}$. 
Due to Sublemma \ref{sublemma2} below, Theorem \ref{global} applies with $\psi_{max,i}=(T_i-a_T h)^2$.

\begin{figure}[h]
\includegraphics[width=7.5cm, height=9cm]{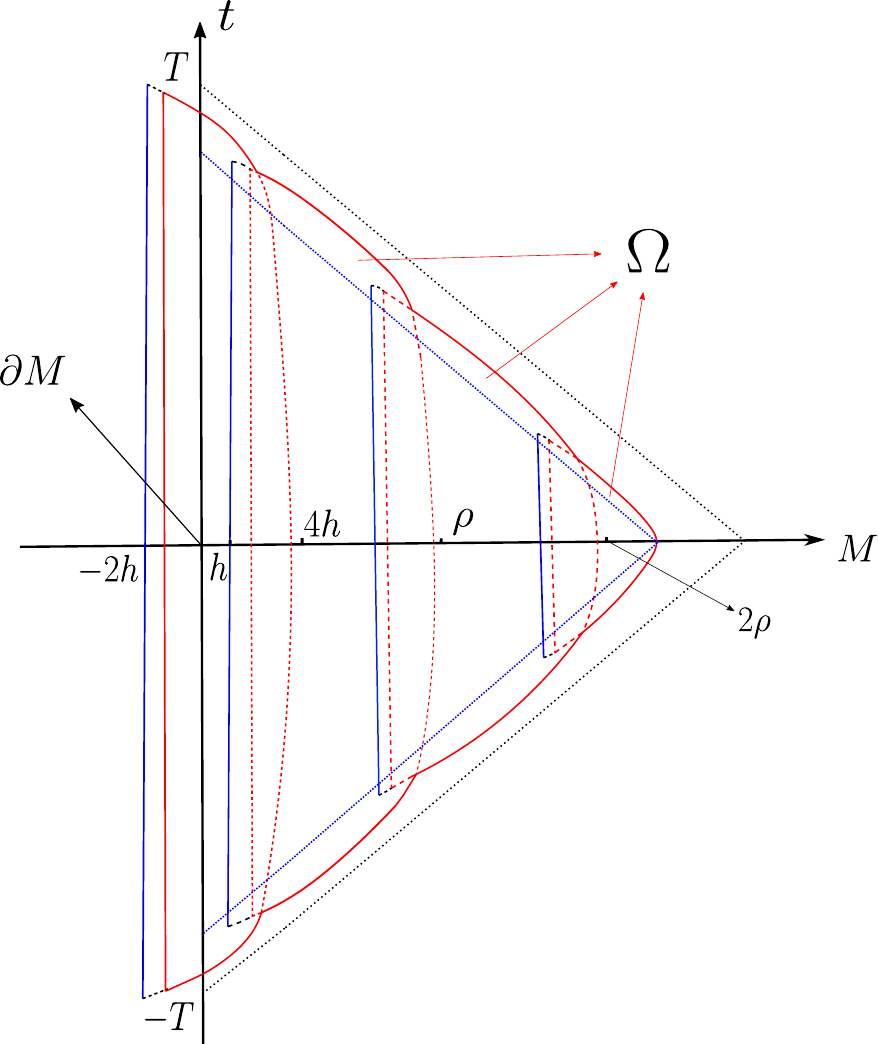}
\caption{The procedure of a three-step propagation besides the initial step. The red solid lines enclose the whole region $\Omega=\cup_{i,j}\Omega_{i,j}$ propagated by the unique continuation. The black dotted line represents the optimal region, while the blue dotted line represents the actual region we can estimate.}
\label{figure3}
\end{figure}

\begin{sublemma1}\label{sublemma1}
For $i\geqslant 2$ and any $z\in \Gamma_i$, we have $d_h(z,\Gamma_{i-1})=\rho_0-4h$.
\end{sublemma1}
\begin{proof}
Let $z_1\in \Gamma_1$ be a point in $\Gamma_1$ such that $d_h(z,z_1)=d_h(z,\Gamma_1)$. Take a minimizing geodesic of $M_h$ from $z$ to $z_1$ and the geodesic intersects with $\Gamma_{i-1}$ at $z_{i-1}\in\Gamma_{i-1}$. This geodesic has length $(i-1)(\rho_0-4h)$, and its segment from $z_{i-1}$ to $z_1$ has length at least $(i-2)(\rho_0-4h)$ by definition. Hence $d_h(z,\Gamma_{i-1})\leqslant d_h(z,z_{i-1})\leqslant \rho_0-4h$. On the other hand, for any $z^{\prime}\in\Gamma_{i-1}$, we have $d_h(z,z^{\prime})\geqslant d_h(z,\Gamma_1)-d_h(z^{\prime},\Gamma_1)=\rho_0-4h$, which shows $d_h(z,\Gamma_{i-1})\geqslant \rho_0-4h$.
\end{proof}

\begin{sublemma2}\label{sublemma2}
For $i\geqslant 2$ and sufficiently small $h<\min\{1/2,T/4\}$ depending on $n,K_1,i_0$, we have $dist_{\widetilde{M}\times \mathbb{R}} (\partial \Omega^{0}_{i,j},\Omega_{i,j})>0$, and
$$\emptyset\neq \big\{(x,t)\in \Omega^{0}_{i,j}: \psi_{i,j}(x,t)>(T_i-a_T h)^2 \big\}\subset \cup_{l=0}^{i-1}\cup_j\overline{\Omega}_{l,j}.$$

\end{sublemma2}
\begin{proof}
We prove the following stronger statement:
\begin{equation}\label{dinclusion}
\{x: d_h^s(x,z_{i,j}) \leqslant a_T h\}\times [-T_i-h,T_i+h] \subset \cup_{l=0}^{i-1}\cup_j\overline{\Omega}_{l,j}.
\end{equation}
More precisely, for any $(x,t)$ in the left-hand set, we prove that if $(x,t)\notin \cup_{l=0}^{i-2}\cup_j\overline{\Omega}_{l,j}$, then $(x,t)\in \cup_j \Omega_{i-1,j}$.

By Sublemma \ref{sublemma1} and the fact that $\{z_{i-1,j}\}$ is an $h$-net in $\Gamma_{i-1}$, we can find some $z_{i-1,j_0}$ such that $d_h(z_{i,j},z_{i-1,j_0})<\rho_0-3h$. Then for any $(x,t)$ in the left-hand set in (\ref{dinclusion}), Lemma \ref{dhs}(2) implies that for sufficiently small $h$ depending on $n,K_1,i_0$,
\begin{equation}\label{sl}
d_h^s(x,z_{i-1,j_0}) < d_h^s(x,z_{i,j})+(\rho_0-3h)(1+CnK_1^2 h^6) < \rho_0-\frac{3}{2}h,
\end{equation}
which indicates that $\xi(\rho_0-d_h^s(x,z_{i-1,j_0}))$ vanishes.

We claim that $(x,t)\in \Omega_{i-1,j_0}$. To prove this, by the definition of $\psi_{i-1,j},\Omega_{i-1,j}$ and the condition that $(x,t)\notin \cup_{l=0}^{i-2}\cup_j\overline{\Omega}_{l,j}$, we only need to show that
$$\psi_{i-1,j_0}(x,t)=\Big( \big(1-\xi(d(x,\partial M))\big)T_{i-1}-d_h^s(x,z_{i-1,j_0})\Big)^2-t^2 > 9T^2 h.$$
Since $|t|\leqslant T_i+h$, it is enough to show
$$\big(1-\xi(d(x,\partial M)\big)T_{i-1}-d_h^s(x,z_{i-1,j_0}) > T_{i}+h+3T\sqrt{h}.$$
Now since $d_h^s(x,z_{i,j})\leqslant h/T$, by the definition of $d_h$ and Lemma \ref{dhs}(4) we have
$$d(x,\partial M)\geqslant h-d_h(x,z_{i,j})h> h-(\frac{h}{T}+\frac{2h^2}{T})h>h-\frac{2h^2}{T},$$
which implies by the definition of $\xi$ (\ref{xidef}),
$$\xi(d(x,\partial M))< \xi(h-\frac{2h^2}{T}) = \frac{8 h^3}{T^3}.$$
Since $T_i=T_{i-1}-\rho_0-3T\sqrt{h}$ by definition, we have by (\ref{sl}),
\begin{eqnarray*}
\big(1-\xi(d(x,\partial M))\big)T_{i-1}-d_h^s(x,z_{i-1,j_0}) &>& T_{i-1}-\xi(d(x,\partial M)) T_{i-1}-\rho_0+\frac{3}{2}h \\
&>& T_i+3T\sqrt{h}+\frac{3}{2}h-\frac{8 h^3}{T^3}T \\
&>& T_i+3T\sqrt{h}+h.
\end{eqnarray*}
This proves $(x,t)\in \Omega_{i-1,j_0}$ and hence (\ref{dinclusion}). 

The inclusion (\ref{dinclusion}) shows that $\{x: d_h^s(x,z_{i,j}) \leqslant a_T h/2\}\times [-T_i,T_i]$ is strictly contained in $\cup_{l=0}^{i-1}\cup_j\overline{\Omega}_{l,j}$, which implies that $\overline{\Omega}_{i,j}\subset \Omega_{i,j}^0$. An explicit lower bound for the distance between their boundaries is estimated in Lemma \ref{mindistance}.

For the second statement of the sublemma, by (\ref{dinclusion}),
\begin{eqnarray*}
\{\psi_{i,j}(x,t)>(T_i-a_T h)^2\}\subset \{d_h^s(x,z_{i,j}) < a_T h\}\times (-T_i,T_i) \subset \cup_{l=0}^{i-1}\cup_j\overline{\Omega}_{l,j}.
\end{eqnarray*}
The non-emptyness directly follows from the definition of $\Omega^{0}_{i,j}$.
\end{proof}

\smallskip
\noindent\textbf{Error estimate for Case 2.} Finally we show that $\overline{\Omega}=\cup_{i\geqslant 0}\cup_{j} \overline{\Omega}_{i,j}$ almost covers the domain of influence in the original manifold $M$ (see Figure \ref{figure3}). More precisely, we prove that there exists $C^{\prime}=C^{\prime}(T,D,K_1,i_0,r_0,r_g)$ such that $\Omega(C^{\prime}h)\subset \overline{\Omega}$. The idea of the proof is similar to that for Case 1, and we omit the parts of the proof identical to Case 1.

For any $(x,t)\in M\times [-T,T]-\overline{\Omega}$, one of the following two situations must happen:\\
(1) $d(x, \partial M)< h$; \\
(2) $x\in M_h$ and $d_h^s(x,z_{i,j})> \big(1-\xi(\rho_0-d_h^s(x,z_{i,j}))\big)T_i-\sqrt{t^2+9T^2 h}$ for any $z_{i,j}\,(i\geqslant 1)$. \\
The situation (1) implies that $d(x,\partial M-\Gamma)<9h$ or $d(x,\Gamma)>T-|t|-\sqrt{6h}$ by the same argument as for Case 1.

\smallskip
Now we focus on the situation (2) when $x\in M_h$. Lemma \ref{dhs}(4) yields that for any $z_{i,j}\,(i\geqslant 1)$,
\begin{equation}\label{errorcase2}
d_h(x,z_{i,j})> \big(1-\xi(\rho_0-d_h^s(x,z_{i,j}))\big)T_i-|t|-3T\sqrt{h}-2h^2.
\end{equation}
Let $z_1\in \Gamma_1$ be a point in $\Gamma_1$ such that $d_h(x,z_1)=d_h(x,\Gamma_1)$, and take a minimizing geodesic of $M_h$ from $x$ to $z_1$. Observe that this minimizing geodesic intersects with each $\Gamma_i$ at most once; otherwise it would fail to minimize the distance $d_h(x,\Gamma_1)$. Furthermore, due to the continuity of the distance function $d_h(\cdot,\Gamma_1)$, if the minimizing geodesic intersects with $\Gamma_i$, then it intersects with $\Gamma_l$ for all $1\leqslant l<i$.
Suppose the minimizing geodesic intersects with $\Gamma_i$ at $z_i\in\Gamma_i$ for $1\leqslant i\leqslant m$, and the intersection does not occur at any nonempty $\Gamma_i$ for $i>m$. Then by Sublemma \ref{sublemma1}, we have
\begin{eqnarray}\label{dhGamma1}
d_h(x,\Gamma_1)=d(x,z_1)&=&d_h(x,z_m)+\sum_{i=1}^{m-1}d_h(z_i,z_{i+1}) \nonumber \\
&\geqslant& d_h(x,z_{m})+(m-1)(\rho_0-4h).
\end{eqnarray}
We claim that $d_h(x,z_m)\leqslant \rho_0-3h$. Suppose not, and by the inequality above, we have $d_h(x,\Gamma_1)>m(\rho_0-4h)$. This implies that $\Gamma_{m+1}\neq\emptyset$ and any minimizing geodesic from $x$ to $\Gamma_1$ must intersect with $\Gamma_{m+1}$, which is a contradiction.

Since $\Gamma_m\neq\emptyset$ by assumption, the step $m$ of our procedure takes place as long as $T_m>2h$ by our stopping criterion. However if $T_m\leqslant 2h$, the procedure stops at some previous step.

\smallskip
\noindent \textbf{(i)} $T_m>2h$.

On $\Gamma_m$, we can find some $z_{m,j}$ such that $d_h(z_m,z_{m,j})<h$ since $\{z_{m,j}\}$ is an $h$-net. Then it follows that $d_h(x,z_{m,j})<\rho_0-2h$. Lemma \ref{dhs}(4) indicates that $d_h^s(x,z_{m,j})<\rho_0-h$. Hence $\xi(\rho_0-d_h^s(x,z_{m,j}))$ in (\ref{errorcase2}) vanishes. Then by (\ref{dhGamma1}),
\begin{eqnarray*}
d_h(x,\Gamma_1)&>& d_h(x,z_{m,j})-h+(m-1)(\rho_0-4h) \\
&>& T_m-|t|-3T\sqrt{h}-h-2h^2+(m-1)(\rho_0-4h) \\
&=& T_1-|t|-3mT\sqrt{h}-h-4(m-1)h-2h^2,
\end{eqnarray*}
where we used $T_m=T_1-(m-1)(\rho_0+3T\sqrt{h})$ by the definition of $T_i$. 

\smallskip
\noindent \textbf{(ii)} $T_m\leqslant 2h$.

From $T_m=T_1-(m-1)(\rho_0+3T\sqrt{h})$, we have
$$T_1\leqslant (m-1)(\rho_0+3T\sqrt{h})+2h.$$
Hence by (\ref{dhGamma1}), we still get a similar estimate as the previous situation:
\begin{eqnarray*}
d_h(x,\Gamma_1)\geqslant (m-1)(\rho_0-4h)&\geqslant& T_1-3(m-1)T\sqrt{h}-2h-4(m-1)h \\
&\geqslant& T_1-|t|-3(m-1)T\sqrt{h}-2h-4(m-1)h.
\end{eqnarray*}

\smallskip
From here, one can follow the rest of the estimates for Case 1 and obtains
$$d(x,\Gamma)>T-|t|-C(m,T,K_1)\sqrt{h}.$$
Combining these situations together, we have proved that $(x,t)\in M\times [-T,T]-\Omega(Ch)$ for $C=\max\{C(m,T,K_1)^2,9\}$ by definition (\ref{Omegaht}). Therefore, there exists $C^{\prime}=C^{\prime}(m,T,K_1)$ such that $\Omega(C^{\prime}h)\subset \overline{\Omega}$. 

The only part left is to estimate the upper bound for $m$. By assumption, $\Gamma_m\neq\emptyset$ and hence $\Gamma_m$ must be taken before $d_h(\cdot,\Gamma_1)$ reaches outside the diameter of $M_h$. Due to Lemma \ref{distances} for $M_h,M$, the diameter of $M_h$ is bounded by $6 D/5$ for sufficiently small $h$ depending only on $K_1$. Thus by the definition of $\Gamma_i$, we have
$$m\leqslant\big[\frac{6D}{\rho_0}\big]+1,$$
where $\rho_0=\min\big\{i_0/2,r_0/2,r_g/4,\pi/6K_1 \big\}$ depends only on $n,\|R_M\|_{C^1},\|S\|_{C^1},$ $i_0,r_0$.

\medskip
\noindent\textbf{Stability estimate.} With all the functions and domains we have constructed, the only part left is to apply Theorem \ref{global}. From the error estimate above, we have proved that there exists $C^{\prime}=C^{\prime}(T,D,K_1,i_0,r_0,r_g)$ such that $\Omega(C^{\prime}h)\subset \overline{\Omega}=\cup_{i\geqslant 0}\cup_{j} \overline{\Omega}_{i,j}$, where $r_g$ is a constant depending only on $n,\|R_M\|_{C^1},\|S\|_{C^1},i_0$. Recall that $\widetilde{u}$ is an extension of $u$ to $\widetilde{M}$ defined by (\ref{extensionu}). 
Theorem \ref{global} yields the following stability estimate on $\overline{\Omega}$ and hence on $\Omega(C^{\prime}h)$.
$$\|u\|_{L^2(\Omega(C^{\prime}h))}\leqslant \|\widetilde{u}\|_{L^2(\overline{\Omega})}\leqslant C \frac{\|\widetilde{u}\|_{H^1(\Omega^0)}}{\Big(\log \big(1+\frac{\|\widetilde{u}\|_{H^1(\Omega^0)}}{\|P\widetilde{u}\|_{L^2(\Omega^0)}}\big)\Big)^{\frac{1}{2}}}, $$
where $\Omega^0=\cup_{i\geqslant 0}\cup_{j} \Omega_{i,j}^0$. During the initial step, we have show $\Omega^0_{0,j}\subset M\cup \Omega_{\Gamma}$, and $\Omega^0_{i,j}$ is defined in $M\times [-T,T]$ for all $i\geqslant 1$. Hence $\Omega^0\subset (M\cup \Omega_{\Gamma})\times [-T,T]$. Since the function $x\mapsto x(\log (1+x))^{-1/2}$ is non-decreasing on $[0,+\infty)$, we have
$$\|u\|_{L^2(\Omega(C^{\prime}h))}\leqslant C \frac{\|\widetilde{u}\|_{H^1((M\cup\Omega_{\Gamma})\times[-T,T])}}{\Big(\log \big(1+\frac{\|\widetilde{u}\|_{H^1((M\cup\Omega_{\Gamma})\times [-T,T])}}{\|P\widetilde{u}\|_{L^2((M\cup\Omega_{\Gamma})\times[-T,T])}}\big)\Big)^{\frac{1}{2}}}.$$
Therefore, the desired stability estimate follows from Lemma \ref{extension} after replacing $h$ by $h/C^{\prime}$. 
The number of domains in each step is not consequential to the estimate as long as relevant quantities of $\psi_{i,j}$ are uniformly bounded. The dependency of the constant is calculated in the Appendix \ref{constants}. 

The second statement of the theorem is due to the following interpolation formula for bounded domains with locally Lipschitz boundary:
$$\|u\|_{H^{1-\theta}}\leqslant \|u\|_{L^2}^{\theta}\|u\|_{H^1}^{1-\theta},\quad \theta\in (0,1).$$
This concludes the proof of Theorem \ref{main1}.

\begin{remark}
If we define $d_h$ (\ref{dh}) with $h^{-2}$ scaling in the boundary neighborhood and require $h<T^{-1}$, then the level sets of $\psi_j$ (\ref{psi}) automatically do not intersect with $\partial M$ even without the $\xi\big(d(x,\partial M)\big)$ term. However, the extra condition $h<T^{-1}$ is not ideal and we want to choose the parameter $h$ as large as possible for a large $T$, considering the stability estimate grows exponentially in $h$. In addition, we frequently used the number $a_T=\min\{1,T^{-1}\}$ exactly for the same purpose.
\end{remark}

\begin{remark}
In the definition of $\Omega_{i,j}^0$ for Case 2, we removed the region where points are $a_T h/2$-close to the reference points, and this region is contained in the set propagated by the unique continuation from previous steps by Sublemma \ref{sublemma2}. The $h^{-1}$ scaling in the definition of $d_h$ (\ref{dh}) directly affects the order of this number $a_T h/2$. Without the scaling the order of this number would be of $h^2$.
\end{remark}

\subsection{Applications of the quantitative unique continuation}\hfill

\smallskip
Due to the trace theorem, Theorem \ref{main1} yields the following estimate on the initial value.

\begin{initial}\label{initial}
Let $M\in \mathcal{M}_n(D,K_1,K_2,i_0,r_0)$ be a compact Riemannian manifold with smooth boundary $\partial M$, and let $\Gamma$ (possibly $\Gamma=\partial M$) be a connected open subset of $\partial M$ with smooth boundary. Suppose $u\in H^2(M\times[-T,T])$ is a solution of the wave equation $Pu=0$. Assume the Cauchy data satisfy
$$u|_{\partial M\times [-T,T]}\in H^{2,2}(\partial M \times [-T,T]),\quad \frac{\partial u}{\partial \mathbf{n}} \in H^{2,2}(\partial M \times [-T,T]).$$
If
$$\|u\|_{H^1(M\times[-T,T])}\leqslant \Lambda_0,\quad \|u\|_{H^{2,2}(\Gamma\times [-T,T])}+\big\|\frac{\partial u}{\partial \mathbf{n}}\big\|_{H^{2,2}(\Gamma\times [-T,T])}\leqslant \varepsilon_0,$$
then for sufficiently small $h$, we have
$$\|u(x,0)\|_{L^2(\Omega(2h,0,3))} \leqslant C_3^{\frac{1}{3}}h^{-\frac{2}{9}}\exp(h^{-C_4 n}) \frac{\Lambda_0+h^{-\frac{1}{2}}\varepsilon_0}{\big(\log (1+h+h^{\frac{3}{2}}\frac{\Lambda_0}{\varepsilon_0})\big) ^{\frac{1}{6}}}\, ,$$
where $C_3,C_4$ are constants independent of $h$, and their dependency on geometric parameters is stated in Theorem \ref{main1}. For a fixed $t\in [-T,T]$, the domain $\Omega(h,t,m)$ is defined as follows:
\begin{equation}\label{Omegahtm}
\Omega(h,t,m)=\big\{x\in M: T-|t|-d(x,\Gamma) >h^{\frac{1}{m}},\; d(x,\partial M-\Gamma)>h^{\frac{1}{m}}\big\}.
\end{equation}
\end{initial}
\begin{proof}
Observe that $\Omega(2h,0,3)\times (-t_0,t_0)\subset \Omega(h)$ with $t_0=(\sqrt[3]{2}-1)\sqrt[3]{h}$ by definition. Then we take $\theta=1/3$ in Theorem \ref{main1} and apply the trace theorem (Theorem 6.6.1 in \cite{BL}): there exists a constant $C$ such that
\begin{eqnarray*}
\|u(x,0)\|_{L^2(\Omega(2h,0,3))} &\leqslant& C t_0^{-\frac{2}{3}}\|u(x,t)\|_{H^{\frac{2}{3}}(\Omega(2h,0,3)\times(-t_0,t_0))} \\
&\leqslant& 4C h^{-\frac{2}{9}}\|u(x,t)\|_{H^{\frac{2}{3}}(\Omega(h))}.
\end{eqnarray*}
\vspace{-4.5mm}
\end{proof}

\begin{remark}
Note that the constant $C_3$ in Corollary \ref{initial} is not exactly the same as the constant $C_3$ in Theorem \ref{main1}. However, they depend on the same set of geometric parameters. In this paper, we keep the same notation for constants if operations do not introduce any new parameter.
\end{remark}

The following independent result gives an explicit estimate on the Hausdorff measure of the boundary of the domain of influence, which shows that the region Corollary \ref{initial} does not cover has a uniformly controlled small volume.

\begin{area}\label{area}
Let $M$ be a compact Riemannian manifold with smooth boundary. For any measurable subset $\Gamma\subset \partial M$ and any $t\geqslant 0$, the following explicit estimate applies:
$$\vol_{n-1} \big(\partial M(\Gamma,t)\big)<C_5 \big(n,\|R_M\|_{C^1},\|S\|_{C^1},i_0,\vol(M),\vol(\partial M) \big),$$
where $M(\Gamma,t)$ is defined in (\ref{def-Mt}). As a consequence, the estimate above implies the following volume estimate due to the Co-Area formula. Namely, for any $t,\gamma\geqslant 0$, we have
$$\vol_n \big(M(\Gamma,t+\gamma)-M(\Gamma,t)\big)< C_5 \big(n,\|R_M\|_{C^1},\|S\|_{C^1},i_0,\vol(M),\vol(\partial M) \big)\gamma.$$
\end{area}

\begin{proof}
Denote the level set of the distance function by $\Sigma_t=\{x \in \textrm{int}(M): d(x,\Gamma)=t\}$. For any point in $\Sigma_t$, there exists a minimizing geodesic from the point to the subset $\Gamma$. These minimizing geodesics do not intersect with $\Sigma_t$ except at the initial points by definition. Moreover, they do not intersect each other in the interior of $M$, as geodesics would fail to minimize distance past a common point in the interior of $M$. Define $l(x)$ to be the infimum of the distances between a point $x\in \Sigma_t$ and the first intersection points with the boundary along all minimizing geodesics from $x$ to $\Gamma$, and to be infinity if any minimizing geodesic from $x$ to $\Gamma$ does not intersect $\partial M-\overline{\Gamma}$. 

For sufficiently small $\epsilon>0$ chosen later, denote 
$$\Sigma_t(\epsilon)=\big\{x\in \Sigma_t: \frac{\epsilon}{2}<l(x)\leqslant \epsilon \big\}.$$
Denote by $U(\Sigma_t(\epsilon))$ the set of all points on all minimizing geodesics from $\Sigma_t(\epsilon)$ to $\Gamma$ and consider the set $U(\Sigma_t(\epsilon))\cap \Sigma_{t^{\prime}}$ for $t^{\prime}\in [t-\epsilon/4,t)$. Clearly the set $U(\Sigma_t(\epsilon))\cap \Sigma_{t^{\prime}}$ does not intersect with $\partial M$ by definition. Furthermore, it is contained in the $C(n,\|R_M\|_{C^1},\|S\|_{C^1})\epsilon^2$-neighborhood of the boundary $\partial M$ if $\epsilon$ is not greater than $\epsilon_0(n,\|R_M\|_{C^1},\|S\|_{C^1},i_0)$, due to Lemma \ref{geodiff}. 

Since the distance function $d(\cdot,\Gamma)$ is Lipschitz with the Lipschitz constant 1, it is differentiable almost everywhere by Rademacher's theorem and its gradient has length at most $1$. The existence of minimizing geodesics from $\Gamma$ yields that the gradient of $d(\cdot,\Gamma)$ has length at least 1 wherever it exists. 
Hence the gradient of $d(\cdot,\Gamma)$ has unit length almost everywhere. We apply the Co-Area formula (e.g. Theorem 3.1 in \cite{F}) to the sets $U(\Sigma_t(\epsilon))\cap \Sigma_{t^{\prime}}$ with the distance function $d(\cdot,\Gamma)$. Then by Lemma \ref{areaLipschitz} and Lemma \ref{geodiff}, we have
\begin{eqnarray*}
\frac{\epsilon}{4} \vol_{n-1}(\Sigma_t(\epsilon)) &<& 5^{n-1} \int_{t-\epsilon/4}^t \vol_{n-1}\big(U(\Sigma_t(\epsilon))\cap \Sigma_{t^{\prime}}\big) dt^{\prime} \\
&=& 5^{n-1} \vol_n \Big( \bigcup_{t^{\prime}\in [t-\epsilon/4,t)} \big(U(\Sigma_t(\epsilon))\cap \Sigma_{t^{\prime}}\big) \Big) \\
&<& 5^{n-1} C(n,\|R_M\|_{C^1},\|S\|_{C^1})\epsilon^2 \vol(\partial M).
\end{eqnarray*}
Then for $\epsilon\leqslant \epsilon_0$ we get
$$\vol_{n-1}(\Sigma_t(\epsilon))<C(n,\|R_M\|_{C^1},\|S\|_{C^1},\vol(\partial M))\epsilon.$$
Hence we have an estimate on the measure of $U_t(\epsilon_0):=\{x\in \Sigma_t: l(x)\leqslant \epsilon_0\}$:
\begin{eqnarray*}
\vol_{n-1}(U_t(\epsilon_0))&=&\vol_{n-1} \big(\bigcup_{k=0}^{\infty}\Sigma_t(\epsilon_0 2^{-k}) \big)=\sum_{k=0}^{\infty} \vol_{n-1} \big(\Sigma_t(\epsilon_0 2^{-k}) \big) \\
&<& C(n,\|R_M\|_{C^1},\|S\|_{C^1},\vol(\partial M))\epsilon_0\sum_{k=0}^{\infty}2^{-k} \\
&<& C(n,\|R_M\|_{C^1},\|S\|_{C^1},i_0,\vol(\partial M)).
\end{eqnarray*}
\indent As for the other part $\Sigma_t-U_t(\epsilon_0)$, if $t>\epsilon_0$, the minimizing geodesics from the points of $\Sigma_t-U_t(\epsilon_0)$ to $\Gamma$ do not intersect the boundary within distance $\epsilon_0$. By the same argument as above, we can control the measure in question in terms of the volume of the manifold:
$$ \frac{\epsilon_0}{2} \vol_{n-1}(\Sigma_t-U_t(\epsilon_0)) < 5^{n-1}\vol(M),$$
which implies that 
$$\vol_{n-1}(\Sigma_t-U_t(\epsilon_0)) < C(n,\|R_M\|_{C^1},\|S\|_{C^1},i_0,\vol(M)).$$ 
Since the part of $\partial M(\Gamma,t)$ on the boundary is bounded by $\vol(\partial M)$, the measure estimate for $\partial M(\Gamma,t)$ follows. 

If $t\leqslant \epsilon_0$, the domain of influence is contained in the boundary normal neighborhood of width $t$. The minimizing geodesics from points of $\Sigma_t-U_t(\epsilon_0)$ to $\Gamma$ do not intersect the boundary within distance $t/2$. Then by the same argument as before, we have
$$\frac{t}{2} \vol_{n-1}(\Sigma_t-U_t(\epsilon_0))< 5^{n-1}\vol(\partial M)t,$$
which completes the measure estimate for $\partial M(\Gamma,t)$. 

The $n$-dimensional volume estimate directly follows from the measure estimate for $\partial M(\Gamma,t)$ and the Co-Area formula.
\end{proof}

Due to the Sobolev embedding theorem and Corollary \ref{initial}, we next prove Proposition \ref{wholedomain}.

\begin{proof}[Proof of Proposition \ref{wholedomain}]
Due to Corollary \ref{initial}, we only need an estimate in $M(\Gamma,T)-\Omega(2h,0,3)$. By the definition (\ref{Omegahtm}) and Proposition \ref{area}, we have
\begin{eqnarray*}
\vol\big(M(\Gamma,T)-\Omega(2h,0,3)\big)<\vol\big(M(\Gamma,T)-M(\Gamma,T-(2h)^{\frac{1}{3}}\big)+\vol(\partial M)(2h)^{\frac{1}{3}} < C h^{\frac{1}{3}}.
\end{eqnarray*}
Since $u(x,0)\in H^1(M)$, by the Sobolev embedding theorem we have for $n\geqslant 3$,
$$\|u(x,0)\|_{L^{\frac{2n}{n-2}}(M)}\leqslant C\|u(x,0)\|_{H^1(M)}\leqslant C\Lambda;$$
for $n=2$,
$$\|u(x,0)\|_{L^6(M)}\leqslant C\|u(x,0)\|_{W^{1,\frac{3}{2}}(M)}\leqslant C\Lambda.$$
Hence H\"older's inequality gives an estimate on the $L^2$-norm of $u(x,0)$ over $M(\Gamma,T)-\Omega(2h,0,3)$.
Then the proposition follows from Corollary \ref{initial}, and the regularity result for the wave equation (e.g. Theorem 2.30 in \cite{KKL}): namely,
$$\max_{t\in [-T,T]}\|u(x,t)\|_{H^1(M)}\leqslant C(T) \|u(x,0)\|_{H^1(M)}.$$
This proves Proposition \ref{wholedomain}.
\end{proof}

\section{Fourier coefficients and the multiplication by an indicator function}\label{section-projection}

In this section, we present the essential step of our reconstruction method where we compute how the Fourier coefficients of a function (with respect to the basis of eigenfunctions) change, when the function is multiplied by an indicator function of a union of balls with center points on the boundary. This step is based on the stability estimate for the unique continuation we have obtained in Section \ref{section-uc}. The results in this section will be applied to study the stability of the manifold reconstruction from boundary spectral data in the next section.

\smallskip
Let $M$ be a compact Riemannian manifold with smooth boundary $\partial M$. Given a small number $\eta>0$, we choose subsets of $\partial M$ in the following way. Suppose $\{\Gamma_i\}_{i=1}^N$ are disjoint open connected subsets of $\partial M$ satisfying
$$\partial M =\bigcup_{i=1}^N \overline{\Gamma}_i, \quad \textrm{diam}(\Gamma_i)\leqslant \eta,$$
where the diameter is measured with respect to the distance of $M$. Assume that every $\Gamma_i$ contains a ball (of $\partial M$) of radius $\eta/6$. Without loss of generality, we assume every $\partial\Gamma_i$ is smooth embedded and admits a boundary normal neighborhood of width $\eta/10$. This is because one always has the choice to propagate the unique continuation from the smaller ball of radius $\eta/6$. An error of order $\eta$ does not affect our final result. 

Let $\alpha=(\alpha_0,\alpha_1,\cdots,\alpha_N)$ with $\alpha_k\in [\eta, D]\cup\{0\}\; (k=0,\cdots,N)$ be a multi-index, where $D$ is the upper bound for the diameter of $M$. Set $\Gamma_0=\partial M$. We define the domain of influence associated with $\alpha$ by
\begin{equation}\label{Malpha}
M_{\alpha}:=\bigcup_{k=0}^N M(\Gamma_k, \alpha_k)=\bigcup_{k=0}^N \big\{ x \in M: d(x,\Gamma_k) < \alpha_k \big\}.
\end{equation}
We will only be concerned with (nonempty) domains of influence with the initial time range $\alpha_k\geqslant \eta$. Hence for sufficiently small $\eta$ explicitly depending on geometric parameters, Proposition \ref{wholedomain} applies with $h<\eta/100$, since $i_b(\overline{\Gamma}_k)\geqslant \eta/10$ for all $k\geqslant 1$ by assumption.

\smallskip
We are given a function $u\in H^3(M)$ with
\begin{equation*}\label{aprior}
\|u\|_{L^2(M)}=1,\;\;\|u\|_{H^3(M)}\leqslant \Lambda.
\end{equation*}

\begin{u0}\label{def-u0}
For a small parameter $\gamma\in (0,N^{-2})$, we can construct a function $u_0\in H^3(M)$ such that
$$u_0 |_{M_{\alpha}}=0, \;\;\; u_0 |_{M^c_{\alpha+\gamma}}=u, \;\;\; \|u_0\|_{L^2(M)}\leqslant 1,$$
\begin{equation}\label{Hnorm}
\|u_0\|_{H^s(M)}\leqslant C_0 \Lambda \gamma^{-s}, \textrm{ for }s\in [1,3],
\end{equation}
where $\alpha+\gamma=(\alpha_0+\gamma,\alpha_1+\gamma,\cdots,\alpha_N+\gamma)$, and $C_0$ is a constant explicitly depending on geometric parameters. 
\end{u0}
\begin{proof}
Let $\{x_l\}$ be a maximal $\gamma/2$-separated set in $M$, and $\{\phi_l\}$ be a partition of unity subordinate to the open cover $\{B_{\gamma/2}(x_l)\}$ of $M$ such that $\|\phi_l\|_{C^s}\leqslant C \gamma^{-s}$. Then the desired function $u_0$ can be defined as
\begin{equation}\label{def-u0-partition}
u_0(x)=\sum_{\textrm{supp}(\phi_l)\cap M_{\alpha}=\emptyset} \phi_l(x) u(x)\, ,\quad x\in M.
\end{equation}
The first three conditions are clearly satisfied. 

To prove the $H^s$-norm condition, we only need to show that the number of nonzero terms in the sum (\ref{def-u0-partition}) is uniformly bounded. Given an arbitrary point $x\in M$, any $B_{\gamma/2}(x_l)$ with $\phi_l(x)\neq 0$ is contained in $B_{\gamma}(x)$. By the definition of a $\gamma/2$-separated set, $\{B_{\gamma/4}(x_l)\}$ do not intersect with each other. Hence it suffices to estimate the number of disjoint balls of radius $\gamma/4$ in a ball of radius $\gamma$. For sufficiently small $\gamma$, the volume of a ball of radius $\gamma$ is bounded from both sides by $C\gamma^n$, which yields that the maximal number of balls is bounded by a constant independent of $\gamma$. To obtain an explicit estimate, it is convenient to work in a Riemannian extension of $M$, for instance in $\widetilde{M}$ defined in Lemma \ref{extensionmetric}. Then an explicit estimate for the maximal number follows from Lemma \ref{distances} and (\ref{Jacobian}).
\end{proof}

Note that due to Proposition \ref{area}, we have
\begin{equation}\label{layervolume}
\vol(M_{\alpha+\gamma}-M_{\alpha})< (N+1)C_5\gamma < 2C_5\gamma^{\frac{1}{2}}.
\end{equation}

\subsection{Approximation results with spectral data without error}\hfill

\smallskip
Suppose the first $J$ Neumann boundary spectral data $\{\lambda_j,\varphi_j|_{\partial M}\}_{j=1}^{J}$ are known without error. Let $u\in H^3(M)$ be a given function with $\|u\|_{L^2(M)}=1$ and $\|u\|_{H^3(M)}\leqslant \Lambda$. Let $u_0$ be defined in Lemma \ref{def-u0}. We define $u_{J}$ to be the projection of $u_0$ onto the first $J$ eigenspaces $\mathcal{V}_{J}= \textrm{span}\{\varphi_1,\cdots,\varphi_{J}\}\subset C^{\infty}(M)$ with respect to the $L^2(M)$-norm:
\begin{equation}\label{u2}
u_{J}=\sum_{j=1}^{J}\langle u_0,\varphi_j\rangle \varphi_j \in \mathcal{V}_{J}.
\end{equation}

We consider the following initial value problem for the wave equation with the Neumann boundary condition:
\begin{eqnarray*}
\partial_t^2 W -\Delta_g W &=& 0,\quad \textrm{  on } \textrm{int}(M)\times \mathbb{R},  \\
\frac{\partial W}{\partial\mathbf{n}}\big|_{\partial M\times \mathbb{R}} &=& 0, \quad \partial_t W|_{t=0}=0,\\
W|_{t=0}&=& v.
\end{eqnarray*}
Denote by $W(v)$ the solution of the wave equation above with the initial value $v$. Then we define $\mathcal{U}$ to be the set of initial values $v\in\mathcal{V}_{J}$ for which the corresponding waves $W(v)$ are small at all $\Gamma_{k}\times[-\alpha_k,\alpha_k]$: namely,
\begin{equation}\label{Udef}
\mathcal{U}(J,\Lambda,\gamma,\varepsilon_1)=\bigcap_{k=0}^ N \big\{v\in \mathcal{V}_{J}: \|v\|_{H^1(M)}\leqslant 3C_0 \Lambda\gamma^{-3},\; \|W(v)\|_{H^{2,2}(\Gamma_{k}\times [-\alpha_k,\alpha_k])}\leqslant \varepsilon_1 \big\}.
\end{equation}
When the parameters $J,\Lambda,\gamma,\varepsilon_1$ are clearly specified in a certain context, we denote this set by simply $\mathcal{U}$ for short. 

Note that since functions in $\mathcal{V}_J$ are smooth on $M$, the wave $W(v)$ for $v\in \mathcal{V}_J$ is also smooth and hence its $H^{2,2}$-norm is well-defined. Given the Fourier coefficients of $v\in \mathcal{V}_{J}$, the conditions of $\mathcal{U}$ can be checked only using the boundary spectral data. In fact, if a function $v$ has the form $v=\sum_{j=1}^{J}v_j \varphi_j$, then $\|v\|_{H^1(M)}=\sum_{j=1}^{J}(1+\lambda_j)v_j^2$, and the wave $W(v)$ over $\partial M$ is given by
\begin{equation}\label{waveboundary}
W(v)(x,t) \big|_{\partial M\times \mathbb{R}}= \sum_{j=1}^{J} v_j \cos(\sqrt{\lambda_j} t) \varphi_j(x)|_{\partial M}.
\end{equation}

For convenience, we use the following equivalent Sobolev norm (e.g. Theorem 2.22 in \cite{KKL}) for a function $v\in H^s(M)$ with the Fourier expansion $v=\sum_{j=1}^{\infty}v_j\varphi_j$:
\begin{equation}\label{Hkdef}
\|v\|_{H^{s}(M)}^2=\sum_{j=1}^{\infty} (1+\lambda_j^s) v_j^2,\textrm{ for }s\in [1,3].
\end{equation}

\begin{smallinitial}\label{smallinitial}
Let $u\in H^3(M)$ be a given function with $\|u\|_{H^3(M)}\leqslant \Lambda$, and $u_0,u_{J}$ be defined in Lemma \ref{def-u0} and (\ref{u2}). Then for any $\varepsilon_1>0$, there exists $J_0=J_0(D,\Lambda,\gamma,\varepsilon_1)$ such that $u_{J}\in \mathcal{U}(J,\Lambda,\gamma,\varepsilon_1)$ for any $J\geqslant J_0$.
\end{smallinitial}

\begin{proof}
Assume $J$ is sufficiently large such that $\lambda_J>1$. Suppose $u_0,u_J$ have expansions:
$$u_0=\sum_{j=1}^{\infty} d_j \varphi_j, \quad u_{J}=\sum_{j=1}^{J} d_j \varphi_j \in \mathcal{V}_{J}. $$
By (\ref{Hkdef}) we know
\begin{eqnarray}\label{H3}
\|u_0\|_{H^3(M)}^2 \geqslant \sum_{j=J+1}^{\infty} d_j^2 \lambda_j^3 \geqslant \lambda_{J} \sum_{j=J+1}^{\infty} d_j^2 \lambda_j^2,
\end{eqnarray}
and hence by (\ref{Hnorm}),
\begin{equation}\label{H2error}
\|u_0-u_{J}\|_{H^2(M)}^2 =  \sum_{j=J+1}^{\infty} (1+\lambda_j^2) d_j^2 \leqslant 2\sum_{j=J+1}^{\infty} \lambda_j^2 d_j^2 \leqslant 2C_0^2 \Lambda^2 \lambda_{J}^{-1}\gamma^{-6}.
\end{equation}
As a consequence, $u_{J}$ satisfies the $H^1$-norm condition of $\mathcal{U}$ (\ref{Udef}):
\begin{eqnarray*}
\|u_{J}\|_{H^1(M)} &\leqslant& \|u_0\|_{H^1(M)}+\|u_0-u_{J}\|_{H^1(M)} \\
&\leqslant& C_0\Lambda \gamma^{-1}+\sqrt{2}C_0 \Lambda \lambda_{J}^{-1/2} \gamma^{-3} < 3C_0\Lambda \gamma^{-3}.
\end{eqnarray*}

Next we show that $u_{J}$ also satisfies the $H^{2,2}$-norm condition of $\mathcal{U}$ (\ref{Udef}) for sufficiently large $J$. This condition is trivially satisfied when $\alpha_k=0$.
Due to the finite speed propagation of waves, the condition $u_0 |_{M_{\alpha}}=0$ implies that $W(u_0)|_{\Gamma_k\times (-\alpha_k,\alpha_k)}=0$ for all $k$ with $\alpha_k\neq 0$. Thus it suffices to show that $W(u_0)-W(u_{J})$ has small $H^{2,2}$-norm on $\partial M\times [-D,D]$. 

Since $u_0\in H^3(M)$, the regularity result for the wave equation (e.g. Theorem 2.45 in \cite{KKL}) shows that
$$W(u_0)\big|_{M\times [-D,D]} \in C([-D,D];H^3(M))\cap C^3([-D,D];L^2(M)).$$
Hence from (\ref{waveboundary}), we have
$$\big(W(u_0)-W(u_{J})\big)(x,t) \big|_{\partial M\times [-D,D]}=\sum_{j=J+1}^{\infty} d_j\cos(\sqrt{\lambda_j}t)\varphi_j(x)|_{\partial M}.$$
Then the trace theorem and (\ref{Hkdef}) imply that
\begin{eqnarray*}
\|W(u_0)-W(u_{J})\|^2_{H^{2}(\partial M)} &\leqslant& C\|W(u_0)-W(u_{J})\|^2_{H^{\frac{11}{4}}(M)} \\
&=& C\sum_{j=J+1}^{\infty}(1+\lambda_j^{\frac{11}{4}})d_j^2\cos^2(\sqrt{\lambda_j}t) \\
&\leqslant& 2C\sum_{j=J+1}^{\infty} d_j^2 \lambda_j^{\frac{11}{4}}\leqslant C(\Lambda)\lambda_{J}^{-\frac{1}{4}}\gamma^{-6},
\end{eqnarray*}
where the last inequality is due to a similar estimate to (\ref{H3}).
For the time derivatives, the trace theorem and (\ref{Hkdef}) imply
\begin{eqnarray*}
\|\partial_t^2 W(u_0)-\partial_t^2 W(u_{J})\|^2_{L^2(\partial M)} &\leqslant& C \|\partial_t^2 W(u_0)-\partial_t^2 W(u_{J})\|^2_{H^{\frac{3}{4}}(M)} \\
&=& C\sum_{j=J+1}^{\infty}(1+\lambda_j^{\frac{3}{4}})d_j^2\lambda_j^2 \cos^2(\sqrt{\lambda_j}t) \\
&\leqslant& 2C\sum_{j=J+1}^{\infty}d_j^2\lambda_j^{\frac{11}{4}} \leqslant C(\Lambda)\lambda_{J}^{-\frac{1}{4}}\gamma^{-6}.
\end{eqnarray*}
Similarly by using (\ref{H3}),
\begin{eqnarray*}
\|\partial_t W(u_0)-\partial_t W(u_{J})\|^2_{L^2(\partial M)} \leqslant C(\Lambda)\lambda_{J}^{-1}\gamma^{-6}.
\end{eqnarray*}
Hence by the definition of $H^{2,2}$-norm (\ref{H21}),
\begin{eqnarray*}
\|W(u_0)-W(u_{J})\|^2_{H^{2,2}(\partial M\times [-D,D])} &\leqslant& 2D\,C(\Lambda)(2\lambda_{J}^{-\frac{1}{4}}\gamma^{-6}+\lambda_{J}^{-1}\gamma^{-6}) \\
&\leqslant& C(D,\Lambda)\lambda_{J}^{-\frac{1}{4}}\gamma^{-6}.
\end{eqnarray*}
Therefore for all $k=0,\cdots,N$ with $\alpha_k\neq 0$, we have
\begin{eqnarray*}
\|W(u_J)\|^2_{H^{2,2}(\Gamma_k\times [-\alpha_k,\alpha_k])} = \|W(u_0)-W(u_J)\|^2_{H^{2,2}(\Gamma_k\times [-\alpha_k,\alpha_k])} \leqslant C(D,\Lambda)\lambda_{J}^{-\frac{1}{4}}\gamma^{-6}.
\end{eqnarray*}
For any $\varepsilon_1>0$, choose sufficiently large $J$ such that $\lambda_{J}\geqslant C(D,\Lambda)\gamma^{-24}\varepsilon_1^{-8}$ and the lemma follows.
\end{proof}

\begin{remark}
The choice of $J_0$ in Lemma \ref{smallinitial} also depends on geometric parameters, which is brought in when applying the trace theorem. Those relevant parameters are part of the parameters we considered in Section \ref{section-uc}, so we omit them in this section for brevity. The same goes for the next two propositions, where the dependency on geometric parameters is brought in when applying Proposition \ref{wholedomain}.
\end{remark}

We prove the following approximation result for finite spectral data.

\begin{projection}\label{projection}
Let $u\in H^3(M)$ be a given function with $\|u\|_{L^2(M)}=1$ and $\|u\|_{H^3(M)}\leqslant \Lambda$. Let $\alpha=(\alpha_0,\cdots,\alpha_N)$, $\alpha_k\in [\eta, D]\cup\{0\}$ be given, and $M_{\alpha}$ be defined in (\ref{Malpha}). Then for any $\varepsilon>0$, there exists sufficiently large $J=J(D,N,\Lambda,\eta,\varepsilon)$, such that by only knowing the first $J$ Neumann boundary spectral data $\{\lambda_j,\varphi_j|_{\partial M}\}_{j=1}^J$ and the first $J$ Fourier coefficients $\{a_j\}_{j=1}^J$ of $u$, we can find $\{b_j\}_{j=1}^{J}$ and $u^a=\sum_{j=1}^{J}b_j \varphi_j$, such that
$$\|u^a-\chi_{M_{\alpha}}u\|_{L^2(M)} <\varepsilon,$$
where $\chi$ denotes the characteristic function.
\end{projection}

\begin{proof}
We consider the following minimization problem in $\mathcal{U}(J,\Lambda,\gamma,\varepsilon_1)$ (denoted by $\mathcal{U}$ from now on) defined in (\ref{Udef}), where the parameters $J,\gamma,\varepsilon_1$ will be determined later. Let $u_{min}\in\mathcal{U}$ be the solution of the minimization problem
\begin{equation}\label{minimization}
\|u_{min}-u\|_{L^2(M)}=\min_{w\in \mathcal{U}} \|w-u\|_{L^2(M)}.
\end{equation}
Observe that given the first $J$ Fourier coefficients of $u$, finding the minimum of the norm $\|w-u\|_{L^2(M)}$ is equivalent to finding the minimum of a polynomial in terms of the ($J$ number of) Fourier coefficients of $w$. Since the conditions of $\mathcal{U}$ (\ref{Udef}) can be checked with finite boundary spectral data by (\ref{waveboundary}) and (\ref{Hkdef}), the minimization problem transforms into a polynomial minimization problem in a bounded domain in $\mathbb{R}^J$ (the space of Fourier coefficients). Hence the Fourier coefficients of the minimizer $u_{min}$ are solvable by only using the finite spectral data.

\smallskip
Next, we investigate what properties this minimizer $u_{min}$ satisfies. By Proposition \ref{wholedomain} and the fact that the Neumann boundary condition is imposed, $w\in\mathcal{U}$ implies that $\|w\|_{L^2(M(\Gamma_k,\alpha_k))}< \varepsilon_2(h,\Lambda,\eta,\gamma,\varepsilon_1)$ for all $k=0,1,\cdots,N$ with $\alpha_k\neq 0$, where
\begin{equation}\label{epsilon20}
\varepsilon_2=C_3^{\frac{1}{3}}h^{-\frac{2}{9}}\exp(h^{-C_4 n}) \frac{\Lambda\gamma^{-3}+h^{-\frac{1}{2}}\varepsilon_1}{\big(\log (1+h^{\frac{3}{2}}\gamma^{-3}\frac{\Lambda}{\varepsilon_1})\big) ^{\frac{1}{6}}}+C_5\Lambda\gamma^{-3} h^{\frac{1}{3n+3}}.
\end{equation}
Hence,
$$\|w\|_{L^2(M_{\alpha})} < (N+1)\varepsilon_2.$$
Then for any $w\in \mathcal{U}$ and in particular for $w=u_{min}$,
\begin{eqnarray} \label{lower}
\|w-u\|_{L^2(M)}^2 &=&  \|w-u\|_{L^2(M_{\alpha})}^2+  \|w-u\|_{L^2(M_{\alpha}^c)}^2 \nonumber \\
&>& \|u\|^2_{L^2(M_{\alpha})}-4N\varepsilon_2 +  \|w-u\|^2_{L^2(M_{\alpha}^c)}.
\end{eqnarray}

On the other hand the following estimate holds for $u_{J}$:
\begin{eqnarray*}
\|u_{J}-u\|_{L^2(M)}^2 &\leqslant& (\|u_{J}-u_0\|_{L^2(M)} + \|u_0-u\|_{L^2(M)})^2 \\
&\leqslant& \|u_{J}-u_0\|^2_{L^2(M)}+4\|u_{J}-u_0\|_{L^2(M)}  + \|u_0-u\|^2_{L^2(M)} \\
&\leqslant & C(\Lambda) \lambda_{J}^{-\frac{1}{2}}\gamma^{-2}+ \|u\|^2_{L^2(M_{\alpha})}+ \|u_0-u\|^2_{L^2(M_{\alpha+\gamma}-M_{\alpha})} ,
\end{eqnarray*}
where the last inequality is due to an estimate for $\|u_{J}-u_0\|_{L^2}$ similar to \eqref{H2error}, and the definition of $u_0$. The definition of partition of unity in (\ref{def-u0-partition}), the Sobolev embedding theorem (see the proof of Proposition \ref{wholedomain}) and (\ref{layervolume}) yield that
$$\|u_0-u\|_{L^2(M_{\alpha+\gamma}-M_{\alpha})}\leqslant \|u\|_{L^2(M_{\alpha+\gamma}-M_{\alpha})} < 2C_5\Lambda \gamma^{\frac{1}{2\max\{n,3\}}}.$$
Hence,
$$\|u_{J}-u\|_{L^2(M)}^2 < C(\Lambda) \lambda_{J}^{-\frac{1}{2}}\gamma^{-2}+ \|u\|^2_{L^2(M_{\alpha})}+ 4C_5^2\Lambda^2 \gamma^{\frac{1}{n+1}}.$$
For sufficiently large $J=J(D,\Lambda,\gamma,\varepsilon_1)$, we have $u_{J}\in \mathcal{U}$ by Lemma \ref{smallinitial}. This indicates that the minimizer $u_{min}$ also satisfies
\begin{equation} \label{upper}
\|u_{min}-u\|_{L^2(M)}^2 < C(\Lambda) \lambda_{J}^{-\frac{1}{2}}\gamma^{-2}+ \|u\|^2_{L^2(M_{\alpha})}+ 4C_5^2\Lambda^2 \gamma^{\frac{1}{n+1}}.
\end{equation}

Combining the two inequalities (\ref{lower}) and (\ref{upper}), we have
$$\|u_{min}-u\|^2_{L^2(M_{\alpha}^c)} < 4N\varepsilon_2 +C(\Lambda) \lambda_{J}^{-\frac{1}{2}}\gamma^{-2}+ 4C_5^2\Lambda^2\gamma^{\frac{1}{n+1}}.$$
The fact that $\|u_{min}\|_{L^2(M_{\alpha})}< N\varepsilon_2$ implies that
\begin{eqnarray*}
\|\chi_{M_{\alpha}}u-(u-u_{min})\|^2_{L^2(M)}&=&\|u_{min}-\chi_{M_{\alpha}^c}u\|^2_{L^2(M)} \\
&=& \|u_{min}-\chi_{M_{\alpha}^c}u\|^2_{L^2(M_{\alpha}^c)}+\|u_{min}\|^2_{L^2(M_{\alpha})} \\
&<& 4N\varepsilon_2 +C(\Lambda) \lambda_{J}^{-\frac{1}{2}}\gamma^{-2}+4C_5^2\Lambda^2\gamma^{\frac{1}{n+1}} +4N^2\varepsilon_2^2 .
\end{eqnarray*}

From our discussion at the beginning of this proof, we know the Fourier coefficients of $u_{min}$ is solvable. Suppose we have found a minimizer $u_{min}=\sum_{j=1}^{J} c_j \varphi_j$. Since the first $J$ Fourier coefficients of $u$ are given as $a_j$, we can replace the function $u-u_{min}$ in the last inequality by $\sum_{j=1}^J a_j \varphi_j-u_{min}$ and the error in $L^2$-norm is controlled by $\Lambda\lambda_{J}^{-1/2}$. Hence by the Cauchy-Schwarz inequality, we obtain
\begin{equation}\label{ualast}
\big\|\chi_{M_{\alpha}}u-\sum_{j=1}^{J}(a_j-c_j)\varphi_j \big\|^2_{L^2(M)} < 8N\varepsilon_2+8N^2\varepsilon_2^2 +C(\Lambda) \lambda_{J}^{-\frac{1}{2}}\gamma^{-2}+ 8C_5^2\Lambda^2\gamma^{\frac{1}{n+1}},
\end{equation}
which makes $u^a :=\sum_{j=1}^{J} b_j \varphi_j$ with $b_j=a_j-c_j$ our desired function. 

Finally, we determine the relevant parameters. For any $\varepsilon>0$, we first choose and fix $\gamma$ such that the last (\ref{ualast}) term $8C_5^2\Lambda^2\gamma^{\frac{1}{n+1}}= \varepsilon^2/4$, and choose sufficiently large $J$ such that the third term is small than $\varepsilon^2/4$. Then we choose $\varepsilon_2$ so that the first two terms $8N\varepsilon_2+8N^2\varepsilon_2^2=\varepsilon^2/4$. Next we determine $\varepsilon_1$. We choose and fix $h<\eta/100$ such that the second term in (\ref{epsilon20}) is equal to $\varepsilon_2/2$, and choose $\varepsilon_1$ such that the first term in (\ref{epsilon20}) is equal to $\varepsilon_2/2$. By Lemma \ref{smallinitial}, there exists sufficiently large $J$ such that $u_{J}\in \mathcal{U}$, which validates all the estimates. The proposition is proved.
\end{proof}

\subsection{Approximation results with spectral data with error} \hfill

\smallskip
Now suppose that not only do we not know all the spectral data, we also only know them up to an error. More precisely, suppose we are given a set of data $\{\lambda^a_j,\varphi^a_j|_{\partial M}\}$ which is a $\delta$-approximation of the Neumann boundary spectral data, where $\lambda_j^a\in \mathbb{R}_{\geqslant 0}$ and $\varphi_j^a|_{\partial M}\in C^2(\partial M)$. By Definition \ref{deferror}, there exists a choice of Neumann boundary spectral data $\{\lambda_j,\varphi_j|_{\partial M}\}_{j=1}^{\infty}$, such that for all $j\leqslant \delta^{-1}$,
\begin{equation}\label{error-condition}
\big|\sqrt{\lambda_j}-\sqrt{\lambda_j^a}\big|<\delta,\quad \|\varphi_j - \varphi_j^a \|_{C^{0,1}(\partial M)}+ \big\|\nabla_{\partial M}^2 (\varphi_j- \varphi_j^a)|_{\partial M} \big\|< \delta.
\end{equation} 

Since $\varphi_j^{a}\in C^2(\partial M)$ by assumption, the bound on the $C^{0,1}$-norm above yields
\begin{equation}\label{close-C1}
 \|\varphi_j - \varphi_j^a \|_{C^{0}(\partial M)}+ \big|\nabla(\varphi_j - \varphi_j^a)|_{\partial M} \big|<\delta, \; \textrm{ for }j\leqslant \delta^{-1}.
\end{equation}
In a local coordinate $(x^1,\cdots,x^{n-1})$ on $\partial M$, for any $f\in C^2(\partial M)$, we have the formula
$$\big(\nabla_{\partial M}^2 f\big)(\frac{\partial}{\partial x^{k}},\frac{\partial}{\partial x^{l}})=\frac{\partial^2 f}{\partial x^{k} \partial x^l} -\sum_{i=1}^{n-1} \Gamma_{kl}^i \frac{\partial f}{\partial x^i}, \quad k,l=1,\cdots,n-1.$$
Furthermore, we can choose to work in the geodesic normal coordinate. Then the norm of the second covariant derivative (the Hessian), the formula above and (\ref{coorb}) yield a bound $C\delta$ on the second derivative of $(\varphi_j- \varphi_j^a)|_{\partial M}$:
\begin{equation}\label{close-C2}
\big|\frac{\partial^2}{\partial x^k \partial x^l} (\varphi_j- \varphi_j^a)|_{\partial M}\big| < C\delta, \; \textrm{ for }j\leqslant \delta^{-1}, \;\, k,l=1,\cdots,n-1.
\end{equation}

We prove the following approximation result analogous to Proposition \ref{projection}.

\begin{measureerror}\label{measureerror}
Let $u\in H^3(M)$ be a given function with $\|u\|_{L^2(M)}=1$ and $\|u\|_{H^3(M)}\leqslant \Lambda$. Let $\alpha=(\alpha_0,\cdots,\alpha_N)$, $\alpha_k\in [\eta, D]\cup\{0\}$ be given, and $M_{\alpha}$ be defined in (\ref{Malpha}). Then for any $\varepsilon>0$, there exists sufficiently large $J=J(D,N,\Lambda,\eta,\varepsilon)$ such that the following holds. \\
There exists $\delta=\delta(D, \vol(\partial M),N,\Lambda,J,\eta,\varepsilon)\leqslant J^{-1}$ such that by knowing a $\delta$-approximation $\{\lambda^a_j,\varphi^a_j|_{\partial M}\}$ of the Neumann boundary spectral data, and knowing the first $J$ Fourier coefficients $\{a_j\}_{j=1}^J$ of $u$, we can find $\{b_j\}_{j=1}^{J}$ and $u^a=\sum_{j=1}^{J}b_j \varphi_j$, such that
$$\|u^a-\chi_{M_{\alpha}}u\|_{L^2(M)} <\varepsilon.$$
Here the known Fourier coefficients of $u$ are with respect to $\{\varphi_j\}$ which is a choice of orthonormalized eigenfunctions satisfying (\ref{error-condition}) for $\{\lambda^a_j,\varphi^a_j|_{\partial M}\}$.
\end{measureerror}

\begin{proof}
Since we only know an approximation of the boundary spectral data, an error appears when we determine if a function belongs to the space $\mathcal{U}$ (\ref{Udef}) in the minimization problem (\ref{minimization}). The norms appeared in the conditions of $\mathcal{U}$ can be written in terms of the Fourier coefficients and boundary spectral data. However in this case, the actual spectral data are unknown and we can only check these norm conditions with a given approximation of the spectral data. First we need to estimate how these conditions change when the spectral data are perturbed. 

For a function $v(x)=\sum_{j=1}^{J}v_j\varphi_j(x)$ with $\sum_{j=1}^{J}v_j^2\leqslant 1$, the error for the $H^1$-norm condition of $\mathcal{U}$ is
\begin{equation}\label{H1error}
\Big|\|v\|^2_{H^1(M)}- \sum_{j=1}^{J} (1+\lambda_j^a)v_j^2\Big| =\sum_{j=1}^{J}|\lambda_j-\lambda_j^a| v_j^2 < (2\sqrt{\lambda_J}+\delta)\delta.
\end{equation}

For the $H^{2,2}$-norm condition of $\mathcal{U}$, from (\ref{waveboundary}) we know
$$
W(v)(x,t)|_{\partial M\times \mathbb{R}}= \sum_{j=1}^{J} v_j \cos(\sqrt{\lambda_j} t) \varphi_j(x)|_{\partial M}. 
$$
To check if this condition is satisfied, we can only use the approximate spectral data:
$$W^a (v)(x,t)|_{\partial M\times \mathbb{R}}= \sum_{j=1}^{J} v_j \cos(\sqrt{\lambda^a_j} t) \varphi_j^a(x)|_{\partial M}. $$
In fact, we are only concerned with a finite time range $t\in [-D,D]$. Since 
$$\big|\cos(\sqrt{\lambda_j} t)-\cos(\sqrt{\lambda_j^a} t)\big|\leqslant |\sqrt{\lambda_j} t-\sqrt{\lambda_j^a} t| <D \delta,$$
we have the following estimate on the error:
\begin{eqnarray*}
\|W(v)-W^a(v)\|_{H^2(\partial M)} &\leqslant& \|\sum_{j=1}^{J} v_j \cos(\sqrt{\lambda_j} t) \varphi_j-\sum_{j=1}^{J} v_j \cos(\sqrt{\lambda_j^a} t) \varphi_j\|_{H^2(\partial M)} \\
&+& \|\sum_{j=1}^{J} v_j \cos(\sqrt{\lambda_j^a} t) \varphi_j-\sum_{j=1}^{J} v_j \cos(\sqrt{\lambda_j^a} t) \varphi_j^a\|_{H^2(\partial M)} \\
&\leqslant&  D\delta \sum_{j=1}^{J}|v_j|\|\varphi_j\|_{H^2(\partial M)}+ \sum_{j=1}^{J} |v_j|\|\varphi_j-\varphi_j^a\|_{H^2(\partial M)} \\
&<& D\delta \sum_{j=1}^{J}\|\varphi_j\|_{H^2(\partial M)} +CJ\delta \sqrt{\vol(\partial M)}\; ,
\end{eqnarray*}
where the last inequality is due to (\ref{close-C1}) and (\ref{close-C2}). By the trace theorem and (\ref{Hkdef}), we know
$$\|\varphi_j\|_{H^2(\partial M)}^2\leqslant C\|\varphi_j\|_{H^3(M)}^2=C(1+\lambda_j^3),$$
and hence we obtain
\begin{equation*}
\|W(v)-W^a(v)\|_{H^2(\partial M)} < C(D,\vol(\partial M))J\lambda_{J}^{\frac{3}{2}}\delta.
\end{equation*}
Similarly for the time derivatives, we have
\begin{equation*}
\|\partial_t W(v)-\partial_t W^a(v)\|_{L^2(\partial M)} < C(D,\vol(\partial M))J\lambda_{J}\delta,
\end{equation*}
and 
$$\|\partial_t^2 W(v)-\partial_t^2 W^a(v)\|_{L^2(\partial M)} < C(D,\vol(\partial M))J\lambda_{J}^{\frac{3}{2}}\delta.$$
Therefore by definition (\ref{H21}), for some $C_0^{\prime}=C_0^{\prime}(D,\vol(\partial M))$, we have
\begin{equation}\label{H21error}
\|W(v)-W^a(v)\|_{H^{2,2}(\partial M\times [-D,D])}< C_0^{\prime}J\lambda_{J}^{\frac{3}{2}}\delta.
\end{equation}

\smallskip
Now following the proof of Proposition \ref{projection}, we still consider the minimization problem (\ref{minimization}), however in a perturbed space of $\mathcal{U}$. We define an approximate space $\mathcal{U}^a$ of $\mathcal{U}$ as follows:
\begin{eqnarray*}
\mathcal{U}^{a}=\bigcap_{k=0}^ N \Big\{v=\sum_{j=1}^{J}v_j\varphi_j :& \sum_{j=1}^{J}v_j^2\leqslant 1,\; 
\sum_{j=1}^{J} (1+\lambda_j^a)v_j^2\leqslant 9C_0^2 \Lambda^2\gamma^{-6}+3\lambda_J^{\frac{1}{2}}\delta,\\
&\|W^a (v)\|_{H^{2,2}(\Gamma_{k}\times [-\alpha_k,\alpha_k])}\leqslant \varepsilon_1+C_0^{\prime}J\lambda_{J}^{\frac{3}{2}}\delta \, \Big\}.
\end{eqnarray*}
Clearly this space $\mathcal{U}^a$ can be determined with only Fourier coefficients and the given approximation $\{\lambda^a_j,\varphi^a_j|_{\partial M}\}$ of the boundary spectral data. Then we consider the minimization problem (\ref{minimization}) with the space $\mathcal{U}$ replaced by $\mathcal{U}^a$. Hence this perturbed minimization problem is solvable by only using the given approximation of the spectral data.

By Lemma \ref{smallinitial}, there exists sufficiently large $J$ such that $u_J\in\mathcal{U}$, and it follows from (\ref{H1error}) and (\ref{H21error}) that $u_{J}\in \mathcal{U}^a$. Then one can follow the rest of the proof for Proposition \ref{projection}. The only part changed is $\varepsilon_2$, since the actual $H^1$ and $H^{2,2}$ norms of $v\in\mathcal{U}^a$ differ from the original conditions of $\mathcal{U}$. More precisely, for any $v\in\mathcal{U}^a$, again by (\ref{H1error}) and (\ref{H21error}), we have
$$\|v\|_{H^1(M)}< \sqrt{9C_0^2 \Lambda^2\gamma^{-6}+6\lambda_J^{\frac{1}{2}}\delta}<3C_0 \Lambda\gamma^{-3}+3\lambda_J^{\frac{1}{4}}\sqrt{\delta},$$
$$\|W (v)\|_{H^{2,2}(\Gamma_{k}\times [-\alpha_k,\alpha_k])} < \varepsilon_1+2C_0^{\prime}J\lambda_{J}^{\frac{3}{2}}\delta.$$
Therefore following the proof of Proposition \ref{projection}, for $\delta<\lambda_J^{-1}$, one obtains an estimate almost the same as (\ref{ualast}) with $\varepsilon_2(\delta)$:
\begin{equation}\label{ualasterror}
\big\|\chi_{M_{\alpha}}u-\sum_{j=1}^{J}(a_j-c_j)\varphi_j \big\|^2_{L^2(M)} < 8N\varepsilon_2(\delta)+8N^2\varepsilon_2^2(\delta) +C(\Lambda) \lambda_{J}^{-\frac{1}{2}}\gamma^{-2}+ 8C_5^2\Lambda^2\gamma^{\frac{1}{n+1}},
\end{equation}
where $c_j$ is the $j$-th Fourier coefficient of a minimizer, and
\begin{equation*}
\varepsilon_2(\delta)=C_3^{\frac{1}{3}}h^{-\frac{2}{9}}\exp(h^{-C_4 n}) \frac{\Lambda\gamma^{-3}+h^{-\frac{1}{2}}(\varepsilon_1+2C_0^{\prime}J\lambda_{J}^{\frac{3}{2}}\delta)}{\bigg(\log \big(1+h^{\frac{3}{2}}\gamma^{-3}\frac{\Lambda}{\varepsilon_1+2C_0^{\prime}J\lambda_{J}^{\frac{3}{2}}\delta}\big)\bigg) ^{\frac{1}{6}}}+C_5\Lambda\gamma^{-3} h^{\frac{1}{3n+3}}.
\end{equation*}

Finally we determine the relevant parameters. For any $\varepsilon>0$, we first choose and fix $\gamma,\varepsilon_2(0),\varepsilon_1$ such that the right hand side of $(\ref{ualasterror})$ with $\delta=0$ is equal to $3\varepsilon^2/4$ in the same way as in Proposition \ref{projection}. By Lemma \ref{smallinitial} we choose and fix sufficiently large $J$ such that $u_{J}\in \mathcal{U}$, which validates all the estimates if we restrict $\delta\leqslant J^{-1}$. At last we choose sufficiently small $\delta<\lambda_J^{-1}$ such that
$$N\varepsilon_2(\delta)+N^2\varepsilon_2^2(\delta)-N\varepsilon_2(0)-N^2\varepsilon_2^2(0)<\frac{\varepsilon^2}{32},$$
and then the proposition follows.
\end{proof}

\begin{remark}\label{projection-partial}
We point out that in Proposition \ref{projection} and \ref{measureerror}, it suffices to know the boundary data on $\cup_{\alpha_i>0} \Gamma_i$ to obtain the estimate for $M_{\alpha}$ with $\alpha_0=0$. This may be useful when only partial boundary spectral data (measured only on a part of the boundary) are known.
\end{remark}

\section{Approximations to boundary distance functions} \label{section-appro}

Let $M$ be a compact Riemannian manifold with smooth boundary $\partial M$. For $x\in M$, the \emph{boundary distance function} $r_x:\partial M\to \mathbb{R}$ is defined by
$$r_x(z)=d(x,z),\quad z\in\partial M.$$
Then the boundary distance functions define a map $\mathcal{R}: M\to L^{\infty}(\partial M)$ by $\mathcal{R}(x)=r_x$. It is known that the map $\mathcal{R}$ is a homeomorphism and the metric of the manifold can be reconstructed from its image $\mathcal{R}(M)$ (e.g. Section 3.8 in \cite{KKL}). Furthermore, the reconstruction is stable (Theorem \ref{2007}). Therefore, to construct a stable approximation of the manifold from boundary spectral data, we only need to construct a stable approximation to the boundary distance functions $\mathcal{R}(M)$. In this section, we construct an approximation to the boundary distance functions through slicing procedures.

\medskip
Given $\eta>0$, let $\{\Gamma_i\}_{i=1}^N$ be a partition of the boundary $\partial M$ into disjoint open connected subsets satisfying the assumptions at the beginning of Section \ref{section-projection}: $\textrm{diam}(\Gamma_i)\leqslant \eta$ and every $\Gamma_i$ contains a ball (of $\partial M$) of radius $\eta/6$, where the diameter is measured with respect to the distance of $M$. We can also choose $\Gamma_i$ to be the closure of these open sets. For example, one can choose $\Gamma_i$ to be the Voronoi regions corresponding to a maximal $\eta/2$-separated set on $\partial M$ with respect to the intrinsic distance $d_{\partial M}$ of $\partial M$. It is straightforward to check that these Voronoi regions satisfy our assumptions with 
\begin{equation}\label{boundN}
N\leqslant C(n,\vol(\partial M))\eta^{-n+1}.
\end{equation}

The approximation results in Section \ref{section-projection} enable us to approximate the volume on $M$ by only knowing an approximation of the Neumann boundary spectral data.

\begin{volume}\label{volume}
Let $\alpha=(\alpha_0,\cdots,\alpha_N)$, $\alpha_k\in [\eta, D]\cup\{0\}$ be given, and $M_{\alpha}$ be defined in (\ref{Malpha}). Then for any $\varepsilon>0$, there exists sufficiently small $\delta=\delta(\eta,\varepsilon)$, such that by only knowing a $\delta$-approximation $\{\lambda^a_j,\varphi^a_j|_{\partial M}\}$ of the Neumann boundary spectral data, we can compute a number $vol^{a}(M_{\alpha})$ satisfying 
$$\big|vol^{a}(M_{\alpha})-\vol(M_{\alpha}) \big|<\varepsilon.$$
\end{volume}

\begin{proof}
Recall that $\varphi_1=\vol(M)^{-1/2}$ on $M$ and it follows that
$$\|\chi_{M_{\alpha}}\varphi_1\|^2_{L^2(M)}=\frac{\vol(M_\alpha)}{\vol(M)}.$$
Since the eigenspace with respect to $\lambda_1=0$ is 1-dimensional, the Fourier coefficients of $\varphi_1$ with respect to any choice of orthonormalized Neumann eigenfunctions are $(1,0,\cdots,0,\cdots)$.
Apply Proposition \ref{measureerror} to $u=\varphi_1$, and we obtain the Fourier coefficients of $u^{a}=\sum_{j=1}^{J}b_j\varphi_j$ for sufficiently large $J$, and that the $L^2$-norm of $u^a$ approximates $\|\chi_{M_{\alpha}}\varphi_1\|_{L^2(M)}$. Therefore $\sum_{j=1}^{J}b_j^2$ approximates $\vol(M_\alpha)/\vol(M)$, and equivalently $\vol(M)\sum_{j=1}^{J}b_j^2$ approximates $\vol(M_\alpha)$. If $\vol(M)$ is known, then $\vol(M)\sum_{j=1}^{J}b_j^2$ is the number we are looking for.

However, we do not exactly know $\vol(M)$ since we do not exactly know the first eigenfunction; we only know an approximation of $\vol(M)$ in terms of the first approximate eigenfunction $\varphi_1^a$. More precisely,
$$\delta>\|\varphi_1-\varphi_1^a\|_{C^0(\partial M)} \geqslant \big| \vol(M)^{-\frac{1}{2}}- \|\varphi_1^a\|_{C^0(\partial M)}\big|. $$
Hence an approximate volume can be defined in the following way:
$$vol^{a}(M_{\alpha}):=\|\varphi^a_1\|^{-2}_{C^0(\partial M)}\sum_{j=1}^{J}b_j^2\, ,$$
and then it satisfies the statement of the lemma.
\end{proof}

Besides the conditions we discussed earlier for the partition $\{\Gamma_i\}$, we need to further restrict the choice of the partition. We start with the following independent lemma regarding the \emph{boundary distance coordinates}. One may refer to Section 2.1.21 in \cite{KKL} for a brief introduction on this subject. This type of coordinates will be used to reconstruct the inner part (bounded away from the boundary) of the manifold.

\begin{coordinate}\label{coordinate}
Let $M\in \mathcal{M}_n(D,K_1,K_2,i_0)$. Then there exist a constant $L$ and boundary points $\{z_i\}_{i=1}^L$, $z_i\in \partial M$ such that the following two properties hold.

\noindent (1) \,For any $x\in M$ with $d(x,\partial M)\geqslant i_0/2$, there exist $n$ boundary points $\{z_{i_1(x)},\cdots,$ $z_{i_n(x)}\}\subset \{z_i\}_{i=1}^L$, such that the distance functions $\big(d(\cdot,z_{i_1(x)}),\cdots,$ $d(\cdot,z_{i_n(x)})\big)$ define a bi-Lipschitz local coordinate in a neighborhood of $x$.

\noindent (2) \,The map $\Phi_L:M\to \mathbb{R}^{L}$ defined by 
$$\Phi_L(x)=\big(d(x,z_1),\cdots,d(x,z_L)\big)$$ 
is bi-Lipschitz on $\{x\in M: d(x,\partial M)\geqslant i_0/2\}$, where the Lipschitz constant and $L$ depend only on $n,D,K_1,K_2,i_0,\vol(\partial M)$. 

\smallskip
Furthermore, the boundary points $\{z_i\}_{i=1}^L$ can be chosen as any $r_L$-maximal separated set on $\partial M$, where $r_L<i_0/8$ is a constant depending only on $n,D,K_1,K_2,i_0$.
\end{coordinate}

\begin{proof}
Given $x\in M$ with $d(x,\partial M)\geqslant i_0/2$, let $z\in \partial M$ be a nearest boundary point: i.e. $d(x,z)=d(x,\partial M)$. Then it follows that $z$ is not conjugate to $x$ along the minimizing geodesic from $x$ to $z$. That is to say, the differential $d\exp_x |_v$ is non-degenerate, where $\exp_x$ denotes the exponential map of $M$ and $v=\exp_x^{-1}(z)$. Hence by the Inverse Function Theorem, there exists a neighborhood of $(x,v)\in TM$ (with respect to the Sasaki metric on the tangent bundle), such that the exponential map is a diffeomorphism to a neighborhood of $z$. Furthermore, one can find a uniform radius $r_1$ depending on $n,D,K_1,K_2,i_0$ for the size of these neighborhoods (Lemma 4 in \cite{KKL2}).

We take $\{z_i\}$ to be an $r_2$-net on $\partial M$ (with respect to the intrinsic distance $d_{\partial M}$ of $\partial M$), where the parameter $r_2<r_1/8$ is determined later. By definition, there exists $z_1\in \{z_i\}$ such that $d_{\partial M}(z,z_1)<r_2$. Then we search for $n-1$ points $z_2,\cdots,z_n$ such that ${}_{\partial M}{\exp}_{z_1}^{-1}(z_j)$ (for $j=2,\cdots,n$) form a basis in $T_{z_1}(\partial M)$, where  ${}_{\partial M}{\exp}$ denotes the exponential map of $\partial M$. We claim that this is possible for sufficiently small $r_2$ explicitly depending on $r_1,n,K_1$. This claim can be proved as follows. Take $v_2,\cdots,v_n$ to be an orthonormal basis of $T_{z_1}(\partial M)$, and consider the points $z_{j}^{\prime}={}_{\partial M}{\exp}_{z_1} (s v_j)\in \partial M$ for a fixed $s\in (r_1/4,r_1/2)$. By definition of $r_2$-net, there exists points $z_2,\cdots,z_n\in \{z_i\}$ such that $d_{\partial M}(z_j^{\prime},z_j)<r_2$ (for $j=2,\cdots,n$). We consider the triangle with the vertices $z_1,z_j^{\prime},z_j$. Since the lengths of the sides $z_1 z_j^{\prime}$ and $z_1 z_j$ are at least $r_1/8$, then for suffciently small $r_2$ explicitly depending on $K_1$, the angle of the triangle at $z_1$ is small (by Toponogov's Theorem) and therefore ${}_{\partial M}{\exp}_{z_1}^{-1}(z_j)$ (for $j=2,\cdots,n$) also form a basis. Then by the same argument as Lemma 2.14 in \cite{KKL}, one can show $z_1,z_2,\cdots,z_n$ are the desired boundary points, from which a boundary distance coordinate is admitted in a neighborhood of $x$.

From now on, we choose $\{z_i\}_{i=1}^L$ to be a maximal $r_2$-separated set on $\partial M$, which is indeed an $r_2$-net by maximality. The cardinality $L$ of this net is bounded by $C(n,\vol(\partial M))r_2^{-n+1}$. The bi-Lipschitzness of the boundary distance coordinates follows from the fact that the differential of the exponential map is uniformly bounded in the relevant domain by a constant depending on $n,D,K_1,K_2,i_0$ (Lemma 3 and Proposition 1 in \cite{KKL2}). This concludes the proof for the first part of the lemma.

\smallskip
Next we prove the second part of the lemma. We claim that there exists $r_3>0$, such that $\Phi_L$ with respect to any maximal $r_3$-separated set on $\partial M$ is bi-Lipschitz on $\{x\in M: d(x,\partial M)\geqslant i_0/2\}$. Note that $\Phi_L$ is automatically Lipschitz with the Lipschitz constant $\sqrt{L}$ by the triangle inequality. Suppose there exist a sequence of manifolds $M_k\in \mathcal{M}_n(D,K_1,i_0)$ and points $x_k,y_k\in \{x\in M_k: d(x,\partial M_k)\geqslant i_0/2\}$, such that 
$$\frac{|\Phi_{L,k}(x_k)-\Phi_{L,k}(y_k)|}{d_{M_k}(x_k,y_k)}\to 0,\textrm{ as }k\to \infty,$$ 
where $\Phi_{L,k}$ is defined with respect to some maximal $1/k$-separated set on $\partial M_k$. The pre-compactness of $\mathcal{M}_n(D,K_1,i_0)$ (Theorem 3.1 in \cite{AKKLT}) yields that there exists a converging subsequence of $M_k$ to a limit $M$ in $C^1$-topology. 
We choose converging subsequences of $x_k,y_k$ to limit points $x,y\in M$. The assumption implies that $\Phi_{L}(x)=\Phi_{L}(y)$ with respect to a dense subset of $\partial M$. Due to the fact that the boundary distance map $\mathcal{R}$ is a homeomorphism (Lemma 3.30 in \cite{KKL}), it follows that $x=y$. Moreover, we have $d(x,\partial M)\geqslant i_0/2$. However, for sufficiently large $k$ such that $x_k,y_k\in B_{r_1}(x)$, the points $x_k,y_k$ lie in the same boundary distance coordinate neighborhood by the first part of the lemma, on which $\Phi_{L,k}$ is locally bi-Lipschitz with a uniformly bounded Lipschitz constant. This is a contradiction to the assumption. Therefore there exists some $r_3>0$ depending on $n,D,K_1,i_0$, such that $\Phi_L$ with respect to any maximal $r_3$-separated set on $\partial M$ is bi-Lipschitz.

Finally, we further restrict $\{z_i\}_{i=1}^L$ to be a maximal $\min\{r_1,r_2,r_3\}$-separated set on $\partial M$. Hence the cardinality $L$ satisfies
$$L\leqslant C(n,\vol(\partial M))\min\{r_1,r_2,r_3\}^{-n+1},$$ 
which depends only on $n,D,K_1,K_2,i_0,vol(\partial M)$. We denote $r_L=\min\{r_1,r_2,r_3\}$ which depends on $n,D,K_1,K_2,i_0$.
\end{proof}

\smallskip
\noindent \textbf{Choice of partition.}
Let $\eta>0$ be given. We choose boundary points $\{z_i\}_{i=1}^N$ and a partition $\{\Gamma_i\}_{i=1}^N$ of $\partial M$ as follows. Let $\{z_1,\cdots,z_L\}$ be the boundary points determined in Lemma \ref{coordinate}, and then we add $N-L$ number of boundary points such that $\{z_{1},\cdots,z_{N}\}$ is a maximal $\eta/2$-separated set on $\partial M$. This is possible 
because $\{z_1,\cdots,z_L\}$ can be chosen as any $r_L$-maximal separated set on $\partial M$, with $r_L$ being a uniform constant independent of $\eta$.
We take $\{\Gamma_i\}_{i=1}^N$ to be a partition of $\partial M$ (e.g. Voronoi regions corresponding to $\{z_i\}_{i=1}^N$) satisfying the assumptions at the beginning of this section: $\textrm{diam}(\Gamma_i)\leqslant \eta$, $z_i\in \Gamma_i$, and every $\Gamma_i$ contains a ball (of $\partial M$) of radius $\eta/6$. The cardinality $N$ of the partition is bounded above by (\ref{boundN}).

\begin{def-section5}\label{Def-Mbeta}
Let $\eta>0$ be given. For multi-indices $\beta$ of the form $\beta=(\beta_0,\beta_1,\cdots,\beta_N)$ with $\beta_0\in\{0,1\},\,\beta_1,\cdots,\beta_N \in\mathbb{N}$, we consider the following two types of sub-domains (see Figure \ref{slicing}).

\smallskip
(1) \,Given a multi-index $\beta=(0,\beta_1,\cdots,\beta_N)$, we define a slicing of the manifold by
\begin{equation}\label{Mbeta}
M_{\beta}^{\ast}=\bigcap_{i:\,\beta_i>0} \big\{x\in M: \, d(x,\Gamma_i)\in [\beta_i\eta-2\eta,\beta_i\eta) \,\big\}.
\end{equation}
We also consider the following modified multi-index by setting specific components zero: 
$$\beta\langle l\rangle:=(0,\beta_1,\cdots,\beta_L,0,\cdots,0,\beta_l,0,\cdots,0),\quad l\in \{L+1,\cdots,N\}.$$

(2) \,Given a multi-index $\beta=(1,\beta_1,\cdots,\beta_N)$, we define a modified multi-index by
$$\beta[k,i]:=(1,0,\cdots,0, \beta_k,0,\cdots,0,\beta_i,0,\cdots,0),\quad k\neq i.$$
In other words, $\beta[k,i]$ can only have nonzero $k$-th and $i$-th components besides the $0$-th component. Then we define the following sub-domain:
\begin{equation}\label{Mki}
M_{\beta[k,i]}^{\ast}=\big\{x\in M: \, d(x,\partial M)\geqslant \beta_k\eta-2\eta,\, d(x,\Gamma_k)<\beta_k\eta,\, d(x,\Gamma_i)\in [\beta_i\eta-2\eta,\beta_i\eta)\, \big\}.
\end{equation}
\end{def-section5}

\begin{figure}[h]
\includegraphics[scale=0.45]{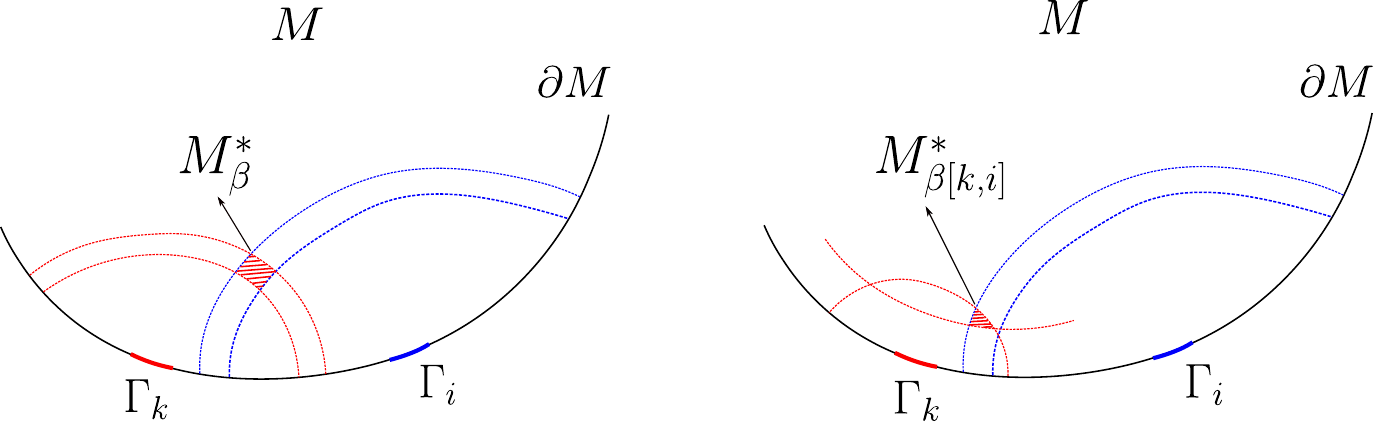}
\caption{Sub-domains from two subsets of the boundary. The former type is used to reconstruct the inner part of the manifold, while the latter type is used to reconstruct the boundary normal neighborhood.}
\label{slicing}
\end{figure}

By definition (\ref{Mbeta}), we only slice the manifold from $\Gamma_i$ if $\beta_i > 0$. Hence $M_{\beta}^{\ast}\subset  M_{\beta\langle l\rangle}^{\ast}$ for any $l\in \{L+1,\cdots,N\}$. Since the diameter of the manifold is bounded above by $D$, it suffices to consider a finite number of choices $\beta_i\leqslant 2+D/\eta$ for each $\beta_i$. Notice that we always use a fixed number (independent of $\eta$) of $\Gamma_i$ to slice the manifold. This keeps the total number of slicings from growing too large as $\eta$ gets small.
 
Similar to Lemma \ref{volume}, we can also evaluate approximate volumes for $\vol(M_{\beta\langle l\rangle}^{\ast})$, $\vol(M_{\beta[k,i]}^{\ast})$, and the error can be made as small as needed given sufficient boundary spectral data.

\begin{volumebeta}\label{volumebeta}
Let $\eta>0$ be given, and $M_{\beta\langle l\rangle}^{\ast},M_{\beta[k,l]}^{\ast}$ be defined in Definition \ref{Def-Mbeta}. Then for any $\varepsilon>0$, there exists sufficiently small $\delta=\delta(\eta,\varepsilon)$, such that by only knowing a $\delta$-approximation $\{\lambda^a_j,\varphi^a_j|_{\partial M}\}$ of the Neumann boundary spectral data, we can compute numbers $vol^{a}(M_{\beta\langle l \rangle}^{\ast})$, $vol^a(M_{\beta[k,i]}^{\ast})$ satisfying 
$$\big|vol^{a}(M_{\beta\langle l \rangle}^{\ast})-\vol(M_{\beta\langle l \rangle}^{\ast}) \big|<2^{L+1}\varepsilon,\; \textrm{ for any }l\in \{L+1,\cdots, N\},$$
and
$$\big|vol^{a}(M_{\beta[k,i]}^{\ast})-\vol(M_{\beta[k,i]}^{\ast}) \big|<4\varepsilon, \;\textrm{ for any }i\neq k,$$
where $L$ is a uniform constant independent of $\eta$ determined in Lemma \ref{coordinate}.
\end{volumebeta}

\begin{proof}
Observe that for any $\beta=(0,\beta_1,\cdots,\beta_N)$ with $\beta_1,\cdots,\beta_N>0$, the sub-domain $M_{\beta}^{\ast}$ can be obtained as a finite number of unions, intersections and complements of the sub-domains $M_{\alpha}$ of the form (\ref{Malpha}) with $\alpha_0=0$. More precisely,
\begin{eqnarray*}
M_{\beta}^{\ast} &=& \bigcap_{i=1}^N \big( M(\Gamma_i,\beta_i\eta)- M(\Gamma_i,\beta_i\eta-2\eta) \big) \\
&=& \bigcap_{i=1}^N M(\Gamma_i,\beta_i\eta) -\bigcup_{i=1}^N M(\Gamma_i,\beta_i\eta-2\eta).
\end{eqnarray*}
Then the volume of $M_{\beta}^{\ast}$ can be written in terms of the volumes of $M_{\alpha}$ with $\alpha_0=0$ through the following operations. For any $n$-dimensional Hausdorff measurable subset $\Omega_1,\Omega_2\subset M$,
$$\vol(\Omega_1-\Omega_2)=\vol(\Omega_1\cup \Omega_2)-\vol(\Omega_2);$$
$$\vol(\Omega_1\cap\Omega_2) = \vol(\Omega_1)+\vol(\Omega_2)-\vol(\Omega_1\cup\Omega_2).$$
Moreover, for any multi-indices $\alpha, \alpha^{\prime}$,
$$\vol(M_{\alpha}\cup M_{\alpha^{\prime}})=\vol(M_{\alpha_{max}}), \;\textrm{ where } (\alpha_{max})_i=\max\{\alpha_i,\alpha_i^{\prime}\}.$$
Therefore the approximate volume $vol^a(M_{\beta}^{\ast})$ for $M_{\beta}^{\ast}$ can be defined by replacing the volumes of $M_{\alpha}$ in the expansion with the approximate volume $vol^a(M_{\alpha})$. 

On the other hand, for a multi-index of the form $\beta[k,i]$, we have
\begin{eqnarray*}
M_{\beta[k,i]}^{\ast} = M(\Gamma_k,\beta_k\eta)\cap M(\Gamma_i,\beta_i\eta)-M(\partial M,\beta_k\eta-2\eta)\cup M(\Gamma_i,\beta_i\eta-2\eta).
\end{eqnarray*}
Recall that the volume information from the whole boundary $\partial M$ is incorporated in the $\alpha_0$ component of the multi-index $\alpha$. Thus the volume of $M_{\beta[k,l]}^{\ast}$ can be written in terms of the volumes of $M_{\alpha}$ with $\alpha_0\geqslant 0$.

For a multi-index of the form $\beta\langle l\rangle$, the total number of volume terms of $M_{\alpha}$ in $\vol(M_{\beta\langle l \rangle}^{\ast})$ is at most $2^{L+1}$. For a multi-index of the form $\beta[k,i]$, the total number of volume terms of $M_{\alpha}$ in $\vol(M_{\beta[k,i]}^{\ast})$ is at most 4. Then the error estimates directly follow from Lemma \ref{volume}.
\end{proof}

Now we are in place to define an approximation to the boundary distance functions $\mathcal{R}(M)$. We consider the following candidate.

\begin{def-section5}\label{Rast}
Let $\eta,\,\varepsilon>0$ be given. For a multi-index $\beta=(\beta_0,\beta_1,\cdots,\beta_N)$ with $\beta_0\in\{0,1\},\,\beta_1,\cdots,\beta_N \in\mathbb{N}_+$, if either of the following two situations happens, we associate with this $\beta$ a piecewise constant function $r_{\beta}\in L^{\infty}(\partial M)$ defined by
$$r_{\beta}(z)=\beta_i \eta, \;\textrm{ if }z\in \Gamma_i.$$

\begin{enumerate}[(1)]
\item $\beta_0=0$; $\beta_i\eta>i_0/2$ for all $i=1,\cdots,N$, and $vol^a(M_{\beta\langle l\rangle}^{\ast})\geqslant \varepsilon$ for all $l=L+1,\cdots,N$.

\item $\beta_0=1$; there exists $k\in \{1,\cdots,N\}$, such that $\beta_k\eta\leqslant i_0/2$ and $vol^a(M_{\beta[k,i]}^{\ast})\geqslant \varepsilon$ for all $i=1,\cdots,N$ with $i\neq k$.
\end{enumerate}

We test all multi-indices $\beta$ up to $\beta_i\leqslant 2+D/\eta$ for each $\beta_i$, and denote the set of all functions $r_{\beta}$ chosen this way by $\mathcal{R}^{\ast}_{\varepsilon}$. 
\end{def-section5}

Intuitively, the first situation in Definition \ref{Rast} describes a small neighborhood in the interior of the manifold away from the boundary. The second situation describes a small neighborhood near the boundary with the help of the boundary normal neighborhood. We prove that $\mathcal{R}^{\ast}_{\varepsilon}$ is an approximation to the boundary distance functions $\mathcal{R}(M)$ for sufficiently small $\varepsilon$.

\begin{approximation}\label{approximation}
Let $M\in \mathcal{M}_n(D,K_1,K_2,i_0,r_0)$. For any $\eta>0$, there exists $\varepsilon=\varepsilon(\eta)$ and sufficiently small $\delta=\delta(\eta)$, such that by only knowing a $\delta$-approximation $\{\lambda^a_j,\varphi^a_j|_{\partial M}\}$ of the Neumann boundary spectral data, we can construct a set $\mathcal{R}^{\ast}_{\varepsilon}\subset L^{\infty}(\partial M)$ such that
$$d_H (\mathcal{R}^{\ast}_{\varepsilon},\mathcal{R}(M)) \leqslant C_6\sqrt{\eta},$$
where $d_H$ denotes the Hausdorff distance between subsets of the metric space $L^{\infty}(\partial M)$, and the constant $C_6$ depends only on $n,D,K_1,K_2,i_0,\vol(\partial M)$.
\end{approximation}

\begin{proof}
Let $\eta<\min\{1,i_0/8\}$. Given any $x\in M$, take a point $x^{\prime}\in M$ such that $d(x,x^{\prime})\leqslant\eta$ and $d(x^{\prime},\partial M)\geqslant \eta$. Clearly there exist positive integers $\beta_i>0$ such that $d(x^{\prime},\Gamma_i)\in [\beta_i\eta-2\eta,\beta_i\eta)$ for all $i=1,\cdots,N$. In fact, there are two choices for each $\beta_i$, and we choose the one satisfying $d(x^{\prime},\Gamma_i)\in [\beta_i\eta-3\eta/2,\beta_i\eta-\eta/2)$ for all $i$. In particular, we see that each $\beta_i$ satisfies $\beta_i\eta-2\eta\leqslant D$.

If $\beta_i\eta>i_0/2$ for all $i=1,\cdots,N$, then we consider the multi-index $\beta=(0,\beta_1,\cdots,\beta_N)$. It follows from the triangle inequality that $B_{\eta/2}(x^{\prime})\subset M_{\beta}^{\ast}$. Since $B_{\eta/2}(x^{\prime})$ does not intersect $\partial M$, we have $\vol(M_{\beta}^{\ast})> \vol(B_{\eta/2}(x^{\prime}))\geqslant c_n\eta^n$ for sufficiently small $\eta$, which implies that $\vol(M_{\beta\langle l\rangle}^{\ast})>c_n \eta^n$ for all $l=L+1,\cdots,N$. We denote 
\begin{equation}\label{epsilon-star}
\varepsilon_{\ast}=c_n\eta^n/2, 
\end{equation}
and set $\varepsilon=2^{-L-1}\varepsilon_{\ast}$ in Lemma \ref{volumebeta}. Then we consider the set of functions $\mathcal{R}^{\ast}_{\varepsilon_{\ast}}$. Since $vol^a(M_{\beta\langle l\rangle}^{\ast})>c_n\eta^n-\varepsilon_{\ast}=\varepsilon_{\ast}$ by Lemma \ref{volumebeta}, we have $r_{\beta}\in \mathcal{R}^{\ast}_{\varepsilon_{\ast}}$ by the first situation in Definition \ref{Rast}. Then by the condition $\textrm{diam}(\Gamma_i)\leqslant \eta$ and the triangle inequality, we have 
\begin{equation}\label{xtobeta}
\|r_x-r_{\beta}\|_{L^{\infty}(\partial M)}\leqslant \|r_x-r_{x^{\prime}}\|_{L^{\infty}(\partial M)}+\|r_{x^{\prime}}-r_{\beta}\|_{L^{\infty}(\partial M)}\leqslant\eta+2\eta=3\eta.
\end{equation}

If there exists $k\in \{1,\cdots,N\}$ such that $\beta_k\eta\leqslant i_0/2$, then we consider the multi-index $\beta=(1,\beta_1,\cdots,\beta_N)$. Without loss of generality, assume $k$ is the index such that $\beta_k=\min_{i> 0} \beta_i$.
Hence
$$d(x^{\prime},\partial M)= \min \big\{d(x',\Gamma_1),\cdots,d(x',\Gamma_N) \big\}\geqslant \beta_k\eta-3\eta/2,$$
which shows $x^{\prime}\in M_{\beta[k,i]}^{\ast}$ for all $i=1,\cdots,N$ with $i\neq k$ by definition (\ref{Mki}). Moreover, we also have $B_{\eta/2}(x^{\prime})\subset M_{\beta[k,i]}^{\ast}$ for all $i$. Thus by choosing the same $\varepsilon_{\ast}$ and $\varepsilon$ as the previous case, we have $r_{\beta}\in \mathcal{R}^{\ast}_{\varepsilon_{\ast}}$ by the second situation in Definition \ref{Rast}, and (\ref{xtobeta}) still holds. This concludes the proof for one direction.

\smallskip
On the other hand, given any $r_{\beta}\in \mathcal{R}^{\ast}_{\varepsilon_{\ast}}$, Definition \ref{Rast} and Lemma \ref{volumebeta} indicates that either $\vol(M_{\beta\langle l\rangle}^{\ast})>0$ for all $l=L+1,\cdots,N$, or there exists $k$ such that $\vol(M_{\beta[k,i]}^{\ast})>0$ for all $i$. Recall that $\beta_1,\cdots,\beta_N>0$ by definition.

\noindent \textbf{(i)} The first situation allows us to pick an arbitrary point $x_l$ in every $M_{\beta\langle l\rangle}^{\ast}$. Then by $\textrm{diam}(\Gamma_i)\leqslant\eta$ and the triangle inequality, we have
\begin{equation}\label{tril}
\|r_{\beta}-r_{x_l}\|_{L^{\infty}(\Gamma_1\cup\cdots\cup \Gamma_L\cup\Gamma_l)}\leqslant 3\eta, \;\textrm{ for any }l\in \{L+1,\cdots,N\}.
\end{equation}
Notice that all $x_l$ are in fact bounded away from the boundary. More precisely, for any $x_l$, we know from Definition \ref{Rast} that
$$d(x_l,\Gamma_i)\geqslant \beta_i\eta-2\eta > i_0/2-2\eta>i_0/4, \;\textrm{ for all }i=1,\cdots, L.$$
Since the boundary points $\{z_i\}_{i=1}^L$ can be chosen as an $r_L$-maximal separated set on $\partial M$, where $r_L<i_0/8$ is a uniform constant independent of $\eta$ (Lemma \ref{coordinate}), we have for any $x_l$,
$$d(x_l,\partial M)>i_0/8.$$
Hence for any other $j\in \{L+1,\cdots,N\}$ with $j\neq l$, Lemma \ref{coordinate} yields that 
$$d(x_l,x_j)\leqslant C(n,D,K_1,i_0)|\Phi_L(x_l)-\Phi_L(x_j)| \leqslant C\sqrt{L}\,\eta,$$
where $\Phi_L(\cdot)=\big(d(\cdot,z_1),\cdots,d(\cdot,z_L)\big)$. Then it follows from the triangle inequality and (\ref{tril}) that
$$\|r_{\beta}-r_{x_l}\|_{L^{\infty}(\Gamma_j)}\leqslant \|r_{\beta}-r_{x_j}\|_{L^{\infty}(\Gamma_j)}+\|r_{x_j}-r_{x_l}\|_{L^{\infty}(\Gamma_j)} \leqslant  (C\sqrt{L}+3)\eta.$$
Thus by ranging $j\neq l$ over $\{L+1,\cdots,N\}$, we obtain
\begin{equation*}\label{rbeta}
\|r_{\beta}-r_{x_l}\|_{L^{\infty}(\partial M)}\leqslant (C\sqrt{L}+3)\eta.
\end{equation*}

\noindent \textbf{(ii)} The second situation allows us to pick an arbitrary point $x_i$ in every $M_{\beta[k,i]}^{\ast}$. Observe from Definition \ref{Rast} that for any $x_i$, we have 
$$d(x_i,\partial M)\leqslant d(x_i,\Gamma_k) < \beta_k\eta\leqslant i_0/2.$$
The fact that $d(x,\Gamma_k)\geqslant d(x,\partial M)$ implies that
$$\|r_{\beta}-r_{x_i}\|_{L^{\infty}(\Gamma_k\cup\Gamma_i)}\leqslant 2\eta.$$
For any other $j\in\{1,\cdots,N\}$ with $j\neq k,i$, we have
$$d(x_i,x_j)\leqslant C\sqrt{\eta}.$$
This is due to the fact that the diameter of the sub-domain $\{x\in M: \, d(x,\partial M)\geqslant \beta_k\eta-2\eta,\, d(x,\Gamma_k)<\beta_k\eta \}$ for $\beta_k\eta\leqslant i_0/2$ is bounded above by $C\sqrt{\eta}$.
Hence by ranging $j\neq k,i$ over $\{1,\cdots,N\}$, we obtain
\begin{equation*}\label{rbeta}
\|r_{\beta}-r_{x_i}\|_{L^{\infty}(\partial M)}\leqslant C\sqrt{\eta}+2\eta.
\end{equation*}
\vspace{-4mm}
\end{proof}

\begin{remark}\label{remark-thirdlog}
We only used a fixed number (independent of $\eta$) of subsets of the boundary to slice the manifold, so that the total number of slicings does not grow too large as $\eta$ gets small. To reconstruct the inner part of the manifold, we used $L+1$ subsets with $L$ being a uniform constant (however not explicit). Near the boundary, we took advantage of the boundary normal neighborhood and essentially only used two subsets. Instead if we use all $N$ subsets to slice the manifold, it would result in a third logarithm in Theorem \ref{stability}. 
\end{remark}

\begin{remark}\label{approximation-partial}
By virtue of Remark \ref{projection-partial}, the approximate volume for $M_{\alpha}$ with $\alpha_0=0$ in Lemma \ref{volume} can be found by only knowing the boundary data on $\cup_{\alpha_i>0} \Gamma_i$. This implies that the approximate volume for $M_{\beta}^{\ast}$ (with $\beta_0=0$) in Lemma \ref{volumebeta} can be found by only knowing the boundary data on $\cup_{\beta_i>0} \Gamma_i$. Thus in a similar but simpler way as Definition \ref{Rast} and Proposition \ref{approximation}, one can define an approximation to $\mathcal{R}(M)$ restricted on a part of the boundary using partial boundary spectral data. Furthermore in the case of partial data, a similar calculation as in Appendix \ref{constants} yields a $\log$-$\log$-$\log$ estimate on the stability of the reconstruction of $\mathcal{R}(M)$.
\end{remark}

The following result shows that the reconstruction of a manifold from $\mathcal{R}(M)$ is stable.

\begin{2007}(Theorem 1 in \cite{KKL2})\label{2007}
Let $M$ be a compact Riemannian manifold with smooth boundary. Suppose $\mathcal{R}^{\ast}$ is an $\eta$-approximation to the boundary distance functions $\mathcal{R}(M)$ for sufficiently small $\eta$. Then one can construct a finite metric space $X$ directly from $\mathcal{R}^{\ast}$ such that
$$d_{GH}(M,X)<C_7(n,D,K_1,K_2,i_0)\, \eta^{\frac{1}{36}},$$
where $d_{GH}$ denotes the Gromov-Hausdorff distance between metric spaces.
\end{2007}

\medskip
Finally we prove the main results Theorem \ref{stability} and Theorem \ref{Cor1}.

\begin{proof}[Proof of Theorem \ref{stability}]
The estimate directly follows from Proposition \ref{approximation} and Theorem \ref{2007}. The dependency of constants is derived in Appendix \ref{constants}.

The only part left is to find an upper bound for $\vol(\partial M),\vol(M)$ in terms of other geometric parameters. Due to Corollary 2(b) in \cite{KKL2}, the (intrinsic) diameter of $\partial M$ is uniformly bounded by a constant depending on $n,D,\|R_M\|_{C^1},\|S\|_{C^2},i_0$, however not explicitly. Then by the volume comparison theorem for $\partial M$, $\vol(\partial M)$ is uniformly bounded by the same set of parameters. As for $\vol(M)$, the manifold $M$ is covered by harmonic coordinate charts with the total number of charts bounded (not explicitly) by a constant depending on $n,D,\|R_M\|_{C^1},\|S\|_{C^2},i_0$ (Theorem 3 in \cite{KKL2}). Away from the boundary, the volumes of balls of a small radius are uniformly bounded. Near the boundary, we can use the boundary normal neighborhood of $\partial M$ since $\vol(\partial M)$ is already shown to be bounded. Hence $\vol(M)$ is uniformly bounded by the same set of parameters.
\end{proof}

\begin{proof}[Proof of Theorem \ref{Cor1}]
We take the first $\delta^{-1}$ Neumann boundary spectral data of $M_2$, and by Definition \ref{deferror}, this set of finite data (without error) is a $\delta$-approximation of the Neumann boundary spectral data of $M_2$. By Proposition \ref{approximation}, we can construct an approximation to $\mathcal{R}(M_2)$. On the other hand, the finite spectral data of $M_2$ is $\delta$-close to the Neumann boundary spectral data of $M_1$ by Definition \ref{deferror}, since the Neumann boundary spectral data of $M_1$ and $M_2$ are $\delta$-close by assumption. Then from the pull-back of the finite spectral data of $M_2$ via the boundary isometry, we can construct an approximation to $\mathcal{R}(M_1)$. Since the boundary isometry (diffeomorphism) preserves Riemannian metrics on the boundaries, the pull-back of the finite spectral data via the boundary isometry produces an isometric approximation to the boundary distance functions. Hence Theorem \ref{Cor1} follows from Corollary 1 in \cite{KKL2}.
\end{proof}

\section{Technical lemmas}\label{auxiliary}

This section contains the proofs of several lemmas used in Section \ref{section-uc}. Some of the lemmas in this section, especially Lemma \ref{dd}, are important technical results, and we prove them here without interrupting the structure of the main proof. Some other lemmas are known facts. We did not find precise references for them, so we present short proofs here.

\begin{riccati}\label{riccati}
Let $(M,g)\in \mathcal{M}_n(D,K_1,K_2,i_0)$. Denote by $S_{\rho}$ the second fundamental form of the equidistant hypersurface in $M$ defined by the level set $d(\cdot,\partial M)=\rho$ for $\rho<i_0$. Then there exists a uniform constant $r_b$ explicitly depending only on $K_1,i_0$, such that for any $\rho\leqslant r_b$, we have $\|S_{\rho}\|\leqslant 2K_1$.

Moreover, if the metric components satisfy (\ref{coorb}) with respect to a coordinate chart in a ball $U$ of $\partial M$, then the metric components with respect to the boundary normal coordinate in $U\times [0,r_b]$ satisfy
$$\|g_{ij}\|_{C^1}\leqslant C(n,\|R_M\|_{C^1},\|S\|_{C^1}),\; \|g_{ij}\|_{C^4}\leqslant C(n,K_1,K_2,i_0),\; \forall \, 1\leqslant i,j\leqslant n.$$
\end{riccati}
\begin{proof}
At an arbitrary point $z\in \partial M$, take an arbitrary unit vector $V$ in $T_z (\partial M)$ and extend it to $V(\rho)\in T_{\gamma_{z,\textbf{n}}(\rho)} M$ ($\rho<i_0$) via the parallel translation along $\gamma_{z,\textbf{n}}$, where $\gamma_{z,\textbf{n}}$ denotes the geodesic of $M$ from $z$ with the initial normal vector $\textbf{n}$ at $z$. We still use the notation $S_{\rho}$ to denote the shape operator of the equidistant hypersurface with distance $\rho$ from $\partial M$. Consider the following function
$$\kappa_V (\rho)=\langle S_{\rho} (V(\rho)),V(\rho) \rangle_{g}.$$
The bound on the second fundamental form of $\partial M$ indicates $|\kappa_V(0)|\leqslant K_1$. For convenience, we omit the evaluation at $\rho$ and use $V$ to denote the vector field $V(\rho)$. 

Since $V$ is a parallel vector field with respect to the normal vector field $\frac{\partial}{\partial \rho}$ (or simply $\partial_{\rho}$), we have
\begin{equation*}
\frac{d }{d\rho} \kappa_V = \langle  \nabla_{\partial_{\rho}} (S_{\rho} V) ,V \rangle+ \langle S_{\rho} V, \nabla_{\partial_{\rho}} V\rangle = \langle  (\nabla_{\partial_{\rho}} S_{\rho}) V,V \rangle.
\end{equation*}
Then the Riccati equation (e.g. Theorem 2 in \cite{PP}, p44) leads to the following formula:
\begin{equation}\label{kappa-riccati}
\frac{d }{d\rho} \kappa_V  =  -\langle  S_{\rho}^2 V,V \rangle +R_M(V,\partial_{\rho},V,\partial_{\rho}) .
\end{equation}
Due to the fact that $S_{\rho}$ is symmetric and $|V|=1$, we have
$$\langle  S_{\rho}^2 V,V \rangle = |S_{\rho} V|^2 \geqslant |\langle  S_{\rho} V,V \rangle|^2 .$$
Hence,
\begin{equation}\label{riccati-scalar-upper}
\frac{d }{d\rho} \kappa_V (\rho) \leqslant -\kappa_V^2 (\rho) + K_1^2\, .
\end{equation}

On the other hand, we need a lower bound for $d\kappa_V/d\rho$. This is possible because we \emph{a priori} know the solution of the Riccati equation exists up to $i_0$, and the equidistant hypersurfaces vary smoothly in a neighborhood of $\partial M$. This implies that there exists a positive number $\rho_{max}\leqslant i_0/2$ satisfying
$$\rho_{max}=\sup \big\{\rho\in [0,\frac{i_0}{2}]: \|S_{\tau}\|\leqslant 2K_1 \textrm{ for all }\tau\in [0,\rho] \big\}.$$
Hence for any $\rho\in [0,\rho_{max}]$, we have $|S_{\rho} V|\leqslant 2K_1$ as the condition above is a closed condition. Then from (\ref{kappa-riccati}),
\begin{equation}\label{riccati-scalar-lower}
\frac{d }{d\rho} \kappa_V (\rho) \geqslant -4K_1^2-K_1^2=-5K_1^2\, .
\end{equation}
Combining (\ref{riccati-scalar-upper}) and (\ref{riccati-scalar-lower}), we have
$$\big| \frac{d }{d\rho} \kappa_V (\rho) \big| \leqslant 5K_1^2\, , \quad \rho\in [0,\rho_{max}].$$
Thus for any $\rho\leqslant \min\{\rho_{max},(10K_1)^{-1}\}$, we have $|\kappa_V(\rho)|\leqslant 3K_1/2$. Since $z$ and $V$ are arbitrary, this shows $\|S_{\rho}\|\leqslant 3K_1/2$. 

We claim that the uniform constant $r_b$ can be chosen as $r_b=\min\{i_0/2,(10K_1)^{-1}\}$. This choice is obviously justified if $\rho_{max}=i_0/2$. Now if $\rho_{max}<i_0/2$, we prove that $\rho_{max}> (10K_1)^{-1}$. Suppose otherwise, and it implies that $\|S_{\rho}\|\leqslant 3K_1/2$ satisfies for any $\rho\leqslant \rho_{max}$. We know the solution of the Riccati equation exists in a neighborhood of $\rho_{max}$, and therefore there exists a larger $\rho>\rho_{max}$ satisfying the condition for $\rho_{max}$ since $\rho_{max}< i_0/2$ by assumption. This contradicts to the maximality of $\rho_{max}$. As a consequence, our estimate holds up to $\rho\leqslant (10K_1)^{-1}$ in this case. On the other hand, the fact that $(10K_1)^{-1}<\rho_{max}<i_0/2$ justifies our choice of $r_b$ in this case. This completes the proof for the first part of the lemma.

\smallskip
For the second part, we consider the matrix Riccati equation in the boundary normal coordinate. This time we use the Lie derivative version of the Riccati equation (e.g. Proposition 7(3) in \cite{PP}, p47). 
The components of the shape operator are denoted by $S_{\alpha}^{l}=\sum_{\beta=1}^{n-1} g^{\beta l}S_{\alpha \beta}$, where $S_{\alpha \beta}$ denotes the components of the second fundamental form of the equidistant hypersurfaces. Here the evaluation at $\rho$ is omitted. Then the Riccati equation has the following form:
$$\frac{d}{d \rho} S_{\alpha\beta}=\sum_{\gamma,l=1}^{n-1} g_{\gamma l}S_{\alpha}^{\gamma}S_{\beta}^{l}+R_M \big(\frac{\partial}{\partial x^{\alpha}}, \frac{\partial}{\partial \rho}, \frac{\partial}{\partial x^{\beta}}, \frac{\partial}{\partial \rho} \big).$$
By definition we have the equation on the distortion of metric:
$$\frac{d}{d \rho} g_{\alpha\beta}=2S_{\alpha\beta}.$$
Due to the first part of the lemma, $d g_{\alpha\beta}/d\rho$ is uniformly bounded. As a consequence, $g_{\alpha\beta}$ is uniformly bounded since it is bounded in the coordinate chart on $\partial M$. The tangential derivatives of $g_{\alpha\beta}$ are estimated as follows.

The Riccati equation can be written in terms of $(S_{\alpha\beta})$ and $(g_{\alpha\beta})$ using the formula for the matrix inverse. We differentiate these two equations with respect to all tangential directions $x^1,\cdots,x^{n-1}$, and we get a system of first-order ODE with the variable $\textbf{v}$:
$$\textbf{v}(\rho)=\big(\cdots,\frac{\partial g_{\alpha\beta}}{\partial x_T}(\rho),\cdots,\frac{\partial S_{\gamma l}}{\partial x_T}(\rho),\cdots \big), \quad \alpha,\beta,\gamma,l=1\cdots,n-1,$$
where $x_T$ ranges over all tangential directions $x^1,\cdots,x^{n-1}$. This system of equations can be written in the following form:
$$\frac{d}{d \rho} \textbf{v}=B_1\textbf{v}+B_2\textbf{v}+\nabla R_M^{\ast}.$$
The matrix $B_1$ is obtained by differentiating the term of the $S^2$ form in the Riccati equation, and only consists of components of the second fundamental form $(S_{\alpha\beta})$ and the metric $(g_{\alpha\beta})$. The matrix $B_2$ is obtained by differentiating the curvature term, and only consists of components of the curvature tensor and $(g_{\alpha\beta})$. The vector $\nabla R_M^{\ast}$ absorbs all the remaining terms and is considered as a constant vector. More precisely, the vector $\nabla R_M^{\ast}$ is made up of components of the covariant derivative $\nabla R_M$, and components of $R_M$, $(S_{\alpha\beta})$, $(g_{\alpha\beta})$.

Due to the first part of the lemma, the components $(S_{\alpha\beta})$ and $(g_{\alpha\beta})$ are uniformly bounded in the boundary normal neighborhood of width $r_b$. Then it follows that the components $(g^{\alpha\beta})$ are also uniformly bounded. This implies that the matrices $B_1,B_2$ have norms bounded above by $C(n,K_1)$, and the vector $\nabla R_M^{\ast}$ has length bounded above by $C(n,K_1,\|\nabla R_M\|)$. The initial condition $|\textbf{v}(0)|$ is bounded above by $n, \|\nabla S\|$. Then the standard theory of ODE yields a bound for $|\textbf{v}|$ and hence for all components of $\textbf{v}$. In particular, $\partial g_{\alpha\beta}/\partial x_T$ are uniformly bounded, which implies that $\|g_{ij}\|_{C^1}\leqslant C(n,\|R_M\|_{C^1},\|S\|_{C^1})$ for all $1\leqslant i,j\leqslant n$.

We keep differentiating the matrix Riccati equation with respect to $x_T$ and $\rho$ up to the fourth order. By the same argument, all relevant coefficients of that system of ODE are uniformly bounded by $\|R_M\|_{C^4}$, $(S_{\alpha\beta})$, $(g_{\alpha\beta})$ and previous lower order estimates. Since the initial condition at $\rho=0$ is bounded by $n, \|g_{\alpha\beta}(0)\|_{C^4}$, $\|S\|_{C^4}$ and $\|R_M\|_{C^3}$, the $C^4$ estimate for the metric components directly follows from (\ref{coorb}).
\end{proof}

\begin{CATradius}\label{CATradius}
(1) \,For any $M\in \mathcal{M}_n(K_1)$, we have $r_{\textrm{CAT}}(M)>0$.

Assume further $M\in \mathcal{M}_n(D,K_1,K_2,i_0)$. The submanifold $M_h$ is defined in Definition \ref{Mhdh}. Suppose $\widetilde{M}$ is an extension of $M$ satisfying Lemma \ref{extensionmetric}(1-3) with the extension width $\delta_{ex}$. Then \\
(2) \,for sufficiently small $h,\delta_{ex}$ explicitly depending on $K_1,K_2,i_0$, we have 
$$r_{\textrm{CAT}}(M_h)\geqslant \min \big\{C(n,\|R_M\|_{C^1},\|S\|_{C^1}),r_{\textrm{CAT}}(M) \big\},$$
$$r_{\textrm{CAT}}(\widetilde{M})\geqslant\min \big\{C(n,\|R_M\|_{C^1},\|S\|_{C^1}),\frac{i_0}{4},\frac{r_{\textrm{CAT}}(M)}{2} \big\};$$
(3) \,for sufficiently small $h,\delta_{ex}$, we have
$$r_{\textrm{CAT}}(M_h)\geqslant \min\big\{\frac{2}{3}r_{\textrm{CAT}}(M),\frac{\pi}{2K_1}\big\},\quad r_{\textrm{CAT}}(\widetilde{M})\geqslant \min\big\{\frac{2}{3}r_{\textrm{CAT}}(M),\frac{\pi}{2K_1}\big\}.$$
\end{CATradius}

\begin{proof}
Due to the Characterization Theorem in \cite{ABB2}, any point $x\in M$ has an open ball $U_x$ such that $U_x$ has curvature bounded above by $K_1^2$ in the sense of Alexandrov. In particular, for any point $p,q\in U_x$ satisfying $d_{U_x}(p,q)<\pi/K_1$, there is a unique minimizing geodesic in $U_x$ (not necessarily a minimizer of $M$) connecting $p$ and $q$ (e.g. Theorem 9.8 in \cite{AKP}). 

\smallskip
\noindent (1) Suppose $r_{\textrm{CAT}}(M)=0$, and there exist sequences of points $p_i,q_i$, such that there are two minimizing geodesics of $M$ joining each pair of points $p_i,q_i$ with $d(p_i,q_i)\to 0$. By the compactness of $M$, we can find converging subsequences of points, still denoted by $p_i$ and $q_i$. Let $x$ be their limit point. For sufficiently large $i$, there are two minimizing geodesics of $M$ connecting $p_i,q_i$ and they both lie in $U_x$, which is a contradiction to the property of $U_x$.

\smallskip
\noindent (2) Given an arbitrary point $p\in M_h$, suppose $q\in M_h$ is a point such that there are two minimizing geodesics of $M_h$ connecting $p,q$. 
Without loss of generality, assume $d_h(p,q)<\min\{\pi/2K_1,r_{\textrm{CAT}}(M)\}$. 
We choose $h$ sufficiently small such that $\|S_{\partial M_h}\|\leqslant 2\|S\|$ and $\|S_{\partial M_h}\|_{C^1}\leqslant 2\|S\|_{C^1}$.
Recall that no conjugate points occur along geodesics (of $M_h$) of length less than $\pi/2K_1$ (Corollary 3 in \cite{ABB2}).
Furthermore, we consider $p,q$ to be the closest pair: $d_h(p,q)=r_{\textrm{CAT}}(M_h)$.
Then by the first variation formula (e.g. Proposition 3 in \cite{ABB2}), the two geodesics connecting $p,q$ form a closed geodesic of $M_h$.  It is known that geodesics on manifolds with smooth boundary are of $C^{1,1}$. Hence their geodesic curvature exists almost everywhere and is bounded by $C(n,\|R_M\|_{C^1},\|S\|_{C^1})$ due to (\ref{acceleration}). Now consider these two geodesics of $M_h$ connecting $p,q$ as a closed $C^{1,1}$-curve of $M$, and it lies in the ball of $M$ centered at $p$ of the radius $\min\{\pi/2K_1,r_{\textrm{CAT}}(M)\}$, which is CAT$(K_1)$ due to Theorem 4.3 in \cite{AB}.
Hence by Corollary 1.2(c) in \cite{AB}, the length of this closed curve is bounded below by $C(n,\|R_M\|_{C^1},\|S\|_{C^1})$, and therefore $d_h(p,q)$ is bounded below by $C(n,\|R_M\|_{C^1},\|S\|_{C^1})$.

Next we derive a lower bound for $r_{\textrm{CAT}}(\widetilde{M})$. Suppose $p,q\in \widetilde{M}$ is the closest pair of points such that there are two minimizing geodesics of $\widetilde{M}$ joining $p,q$. Assume $\widetilde{d}(p,q)<\min\{\pi/4K_1,i_0/4,$ $r_{\textrm{CAT}}(M)/2\}$. Then we immediately see that at least one of these two geodesics intersects $\widetilde{M}-M$.
This implies that both geodesics lie in the boundary normal (tubular) neighborhood of $\partial M$ by assumption. 
Furthermore, the two geodesics connecting $p,q$ form a closed geodesic of $\widetilde{M}$ by the first variation formula. 
We move inwards this closed geodesic along the family of geodesics normal to $\partial M$ by distance $\delta_{ex}<i_0/2$. This process results in a closed $C^{1,1}$-curve of $M$ contained in the boundary normal neighborhood.
For sufficiently small $\delta_{ex}$ depending on $K_1,K_2$, this closed $C^{1,1}$-curve of $M$ has length at most $3\widetilde{d}(p,q)$ and its geodesic curvature is bounded by $C(n,\|R_M\|_{C^1},\|S\|_{C^1})$ almost everywhere. 
Hence this closed curve of $M$ lies in a ball of $M$ of the radius $\min\{\pi/2K_1,r_{\textrm{CAT}}(M)\}$ (which is CAT$(K_1)$),
and therefore its length is bounded below by $C(n,\|R_M\|_{C^1},\|S\|_{C^1})$ by Corollary 1.2(c) in \cite{AB}. This shows that the length of the original closed geodesic of $\widetilde{M}$ is bounded below by $C(n,\|R_M\|_{C^1},\|S\|_{C^1})$, which gives the lower bound for $\widetilde{d}(p,q)$.

\smallskip
\noindent (3) Here we only prove for $M_h$; the proof for $\widetilde{M}$ is the same. Suppose not, and we can find $p_i,q_i\in M_{h_i}$ ($h_i\to 0$) such that there are two minimizing geodesics of $M_{h_i}$ connecting each pair $p_i,q_i$ with $d_{h_i}(p_i,q_i)<\min\{2r_{\textrm{CAT}}(M)/3,$ $\pi/2K_1\}$. Moreover, we can assume $q_i$ is the closest point from $p_i$ such that this happens, and therefore the two geodesics connecting $p_i,q_i$ form a closed geodesic of $M_{h_i}$. 
Thus we have a sequence of closed $C^1$-curves with lengths less than $4r_{\textrm{CAT}}(M)/3$. This sequence of closed curves also has lengths uniformly bounded away from $0$ due to $(2)$. Hence by the Arzela-Ascoli Theorem, we can find a subsequence converging to a limit closed curve in $M$ of nonzero length (not necessarily of $C^1$). Let $p,q\in M$ be the limit points of $p_i,\,q_i$. Since $d_{h_i}$ converges to $d$ (Lemma \ref{distances}), then the lower semi-continuity of length yields that the limit closed curve has length at most $2d(p,q)$.

Consider the segment of the limit closed curve from $p$ to $q$ and the other segment from $q$ to $p$. Both segments must have lengths at least the distance $d(p,q)$. Since the limit closed curve has length at most $2d(p,q)$, each segment is a minimizing geodesic of $M$. If these two segments do not coincide, then we get two minimizing geodesics of $M$ from $p$ to $q$ of lengths at most $2r_{\textrm{CAT}}(M)/3$, which contradicts the condition for $r_{\textrm{CAT}}(M)$. If the two segments coincide, we pick a point $y\in M$ on the limit curve close to $p$, and consider points $y_1,y_2$ on the closed geodesic of $M_{h_i}$ near (fixed) $y$ at opposite sides from $p_i$. For sufficiently large $i$, the points $y_1,y_2$ can be arbitrarily close in $M_{h_i}$ and meanwhile bounded away from $p_i$. However, the angle between the geodesic segment of $M_{h_i}$ from $p_i$ to $y_1$ and the segment from $p_i$ to $y_2$ is always $\pi$, since the curve in question is a closed $C^1$-curve. This is a contradiction to the local CAT condition for $M_{h_i}$ combined with $(2)$.
\end{proof}

\begin{dhs}\label{dhs}
Let $h$ be sufficiently small determined at the beginning of Section \ref{subsection3.4}. Let $d_h^s(\cdot,z)$ (Definition \ref{definition-dhs}) be the smoothening of the function $d_h(\cdot,z)$ (Definition \ref{Mhdh}) with the smoothening radius $r=a_T h^3$, where $a_T=\min\{1,T^{-1}\}$. Then the following properties are satisfied for $z,z_1,z_2\in M_h$, $x\in M$ and $x_1,x_2\in \widetilde{M}$.\\
(1) \,$|d_h(x_1,z_1)-d_h(x_1,z_2)|\leqslant d_h(z_1,z_2).$\\
(2) \,$|d_h^s(x,z_1)-d_h^s(x,z_2)|\leqslant (1+CnK_1^2 h^6)d_h(z_1,z_2).$ \\
(3) \,For sufficiently small $h$ only depending on $K_1$, we have
$$|d_h(x_1,z)-d_h(x_2,z)|<\frac{3}{2}h^{-1}\widetilde{d}(x_1,x_2).$$
(4) \,For sufficiently small $h$ depending on $n,K_1,i_0$, if $d_h(x,z)<i_0$, then 
$$|d_h^s(x,z)-d_h(x,z)|<2a_T h^2.$$
\end{dhs}
\begin{proof}
(1) directly follows from the definition of $d_h$.\\
(2) Let $r=a_T h^3$. Observe that the ball of radius $h^3$ centered at any $x\in M$ does not intersect $\partial \widetilde{M}$, and hence the distance function $\widetilde{d}(\cdot,x)$ for $x\in M$ is simply a geodesic distance function. Due to (\ref{curvatureextended}), the Jacobian $J_x(v)$ of the exponential map $\exp_x(v)$ of $\widetilde{M}$ at $v\in \mathcal{B}_{r}(0)\subset T_x\widetilde{M}$ satisfies
\begin{equation}\label{Jacobian}
|J_x(v)-1| \leqslant CnK_1^2|v|^2 \leqslant CnK_1^2 h^6.
\end{equation}
Then it follows from (\ref{normalization}) that
\begin{eqnarray}\label{normalizationestimate}
\int_{\widetilde{M}} k_1\big(\frac{\widetilde{d}(y,x)}{r}\big)dy&=& 
\int_{\mathcal{B}_{r}(0)\subset T_x{\widetilde{M}}} k_1\big(\frac{|v|}{r}\big)J_x(v)dv \nonumber \\
&\leqslant& (1+CnK_1^2h^6) \int_{\mathbb{R}^n} k_1\big(\frac{|v|}{r}\big)dv.
\end{eqnarray}
This inequality (\ref{normalizationestimate}), (\ref{normalization}) and (1) yield (2). 

\smallskip
\noindent (3) Recall that the second fundamental form of $\partial M_h$ is bounded by $2K_1$ due to Lemma \ref{riccati}, and $\widetilde{M}$ can be considered as an extension of $M_h$ by gluing a collar of width $6h$. If $x_1,x_2\in M_h$, then Lemma \ref{distances} applies by replacing $M$ with $M_h$ and we have
\begin{equation}\label{dhprojection}
|d_h(x_1,z)-d_h(x_2,z)|\leqslant d_h(x_1,x_2)\leqslant (1+36K_1h)\widetilde{d}(x_1,x_2)\, .
\end{equation}

If $x_1,x_2\in \widetilde{M}-M_h$, then Lemma \ref{distances} yields
$$d_h(x_1^{\perp_h},x_2^{\perp_h})\leqslant (1+36K_1h)\widetilde{d}(x_1,x_2).$$
Then by the definition of $d_h$ (\ref{dh}) and (\ref{dhprojection}), we have
\begin{eqnarray*}
|d_h(x_1,z)-d_h(x_2,z)|&\leqslant& |d_h(x_1^{\perp_h},z)-d_h(x_2^{\perp_h},z)|+ h^{-1}|\widetilde{d}(x_1,x_1^{\perp_h})-\widetilde{d}(x_2,x_2^{\perp_h})| \\
&\leqslant& d_h(x_1^{\perp_h},x_2^{\perp_h}) + h^{-1}|\widetilde{d}(x_1,\partial M_h)-\widetilde{d}(x_2,\partial M_h)| \\
&\leqslant& (1+36K_1h)\widetilde{d}(x_1,x_2)+h^{-1}\widetilde{d}(x_1,x_2).
\end{eqnarray*}
Thus the desired estimate follows for sufficiently small $h$ only depending on $K_1$.

If $x_1\in \widetilde{M}-M_h,\,x_2\in M_h$, then similarly we have
\begin{eqnarray*}
|d_h(x_1,z)-d_h(x_2,z)|&\leqslant& |d_h(x_1^{\perp_h},z)-d_h(x_2,z)|+ h^{-1}\widetilde{d}(x_1,x_1^{\perp_h}) \\
&\leqslant& d_h(x_1^{\perp_h},x_2) + h^{-1}\widetilde{d}(x_1,\partial M_h) \\
&\leqslant& (1+36K_1h)\widetilde{d}(x_1,x_2) +h^{-1}\widetilde{d}(x_1,x_2),
\end{eqnarray*}
and the same estimate follows.

\smallskip
\noindent (4) In view of (\ref{Jacobian}) and (\ref{normalizationestimate}), the Jacobian only generates error terms of order at least $h^6$. Hence we only need to prove that for any point $y$ in the ball (of $\widetilde{M}$) of the smoothening radius $a_T h^3$ around the center $x\in M$, it satisfies that $|d_h(y,z)-d_h(x,z)|<3a_T h^2/2$, which is guaranteed by (3). 
\end{proof}

\begin{closeness-curve}\label{closeness-curve}
Let $\gamma_1,\gamma_2: [0,l]\to \mathbb{R}^n$ be two $C^{1,1}$ curves. If $\|\gamma_1-\gamma_2\|_{C^0}\leqslant \epsilon<l^2/4$ and $\|\gamma_i^{\prime\prime}\|_{L^{\infty}}\leqslant \kappa$ for $i=1,2$, then $\|\gamma_1^{\prime}-\gamma_2^{\prime}\|_{C^0}\leqslant C(\kappa)\sqrt{\epsilon}$.
\end{closeness-curve}

\begin{proof}
Since $\gamma_i$ is of $C^{1,1}$, $\gamma_i^{\prime}$ is absolute continuous. Hence Taylor's theorem with the integral form of the remainder applies:
$$\gamma_i (s_2)=\gamma_i (s_1)+\gamma_i^{\prime}(s_1)(s_2-s_1)+\int_{s_1}^{s_2} \gamma_i^{\prime\prime}(\tau)(s_2-\tau)d\tau, \quad \forall \; 0\leqslant s_1<s_2\leqslant l.$$
From $\|\gamma_i^{\prime\prime}\|_{L^{\infty}}\leqslant \kappa$, we have
$$\big|\gamma_i (s_2)-\gamma_i (s_1)-\gamma_i^{\prime}(s_1)(s_2-s_1) \big|\leqslant \frac{\kappa}{2}(s_2-s_1)^2.$$
We add this inequality for $\gamma_1,\gamma_2$ together:
$$\big|(\gamma_1 (s_2)-\gamma_2 (s_2))-(\gamma_1 (s_1)-\gamma_2 (s_1))-(\gamma_1^{\prime}(s_1)-\gamma_2^{\prime}(s_1))(s_2-s_1) \big|\leqslant \kappa(s_2-s_1)^2.$$
Then by $\|\gamma_1-\gamma_2\|_{C^0}\leqslant \epsilon$,
$$|\gamma_1^{\prime}(s_1)-\gamma_2^{\prime}(s_1)|\leqslant \frac{2\epsilon}{s_2-s_1}+\kappa (s_2-s_1).$$
Take $s_2-s_1=\sqrt{\epsilon}$ if exists, and we have
$$|\gamma_1^{\prime}(s_1)-\gamma_2^{\prime}(s_1)|\leqslant (\kappa+2)\sqrt{\epsilon}.$$
Since $\sqrt{\epsilon}<l/2$, we can find $s_2=s_1+\sqrt{\epsilon}$ for any $s_1\in [0,l/2]$. For $s_1\in (l/2,l]$, one can repeat the whole process backwards. Hence the estimate above holds for all $s_1\in[0,l]$, which proves the lemma.
\end{proof}

\begin{dd}\label{dd}
Let $h$ be sufficiently small determined at the beginning of Section \ref{subsection3.4}. Let $d_h^s(\cdot,z)$ (Definition \ref{definition-dhs}) be the smoothening of the function $d_h(\cdot,z)$ (Definition \ref{Mhdh}) with the smoothening radius $r=a_T h^3$, where $a_T=\min\{1,T^{-1}\}$. Then for sufficiently small $h$ depending on $n,K_1,K_2$, given any $x\in M$ and $z\in M_h$ satisfying $h/4 \leqslant d_h (x,z)\leqslant \min\{i_0/2,r_0/2,\pi/6K_1\}$, we have
$$|\nabla_x d_h^s(x,z)|> 1-2h.$$
\end{dd}
\begin{proof}
Let $r=a_T h^3$. By the definition (\ref{dhsdef}), we have
\begin{eqnarray*}
\nabla_x d_h^s(x,z)&=&c_{n}r^{-n}\int_{\widetilde{M}}\nabla_x k_1\big(\frac{\widetilde{d}(y,x)}{r}\big)d_h(y,z)dy \\
&=& c_{n}r^{-n} \int_{\widetilde{B}_r(x)\subset \widetilde{M}} k_1^{\prime}\frac{1}{r} \big(\frac{-\exp_{x}^{-1}(y)}{\widetilde{d}(y,x)} \big) d_h(y,z) dy,
\end{eqnarray*}
where $\exp_x$ denotes the exponential map of $\widetilde{M}$ at $x\in M$. Now we change to the geodesic normal coordinate of $\widetilde{M}$ around $x$, and identify vectors in the tangent space $T_x\widetilde{M}$ with points in $\mathbb{R}^n$:
\begin{eqnarray*}
\nabla_x d_h^s(x,z) 
&=&c_{n}r^{-n} \int_{\mathcal{B}_{r}(0)\subset T_x \widetilde{M}} k_1^{\prime}\frac{1}{r} \frac{(-v)}{|v|} d_h(\exp_x(v),z) J_x(v) dv \\
&=& c_{n}r^{-n} \int_{\mathcal{B}_{r}(0)\subset T_x \widetilde{M}}  -\nabla_v \Big(k_1\big(\frac{|v|}{r}\big)\Big) d_h(\exp_x(v),z) J_x(v) dv \\
&=& c_{n}r^{-n} \int_{\mathcal{B}_{r}(0)\subset T_x \widetilde{M}}  k_1\big(\frac{|v|}{r}\big) \nabla_v\big(d_h(\exp_x(v),z) J_x(v)\big) dv,
\end{eqnarray*}
where $J_x(v)$ denotes the Jacobian of $\exp_x$ at $v$. Here we have used integration by parts in the last equality.

It is known that $|\nabla_v J_x(v)|\leqslant C(n,K_1,K_2)|v|\leqslant C(n,K_1,K_2)h^3$ due to the $C^1$-estimate for the metric components (Lemma 8 in \cite{HV}) and Lemma \ref{extensionmetric}(3). 
Then by (\ref{normalization}), we have
\begin{eqnarray*}
&& \bigg| \, c_{n}r^{-n} \int_{\mathcal{B}_{r}(0)}  k_1\big(\frac{|v|}{r}\big) d_h(\exp_x(v),z) \big(\nabla_v J_x(v)\big) dv \, \bigg| \\ &\leqslant& c_{n}r^{-n} \int_{\mathcal{B}_{r}(0)}  k_1\big(\frac{|v|}{r}\big) \frac{\pi}{4K_1} C(n,K_1,K_2)h^3 dv \leqslant C(n,K_1,K_2)h^3.
\end{eqnarray*}
Hence we only need to estimate the lower bound for the length of the dominating term
\begin{equation}\label{ddA0}
A_0=c_{n}r^{-n} \int_{\mathcal{B}_{r}(0)\subset T_x \widetilde{M}}  k_1\big(\frac{|v|}{r}\big) \big(\nabla_v d_h(\exp_x(v),z)\big) J_x(v) dv.
\end{equation}

We start by considering the following two simple cases.

\smallskip
\textbf{Case 1:}  $d_h(z,\partial M_h)>\min\{i_0/2,r_0/2,\pi/6K_1\}$.

In this case, we know $x\in M_h$ and no geodesic from $z$ to $x$ intersects with $\partial M_h$ in this case. Then the distance function $d_h(\cdot,z)$ in the relevant domain is simply a geodesic distance function with the second derivative bounded by $5/h$ for sufficiently small $h$ depending on $K_1$ (e.g. Theorem 27 in \cite{PP}, p175). Since the exponential map and its inverse are uniformly bounded up to $C^2$ in the relevant domain for sufficiently small $h$ depending on $K_1,K_2$, we have
$$\big|\nabla_v d_h(\exp_x(v),z)-\nabla_v d_h(\exp_x(v),z)|_{v=0}\big|\leqslant Ch^{-1}|v| \leqslant Ch^2.$$
Note that vectors in $T_v(T_x \widetilde{M})$ are identified with vectors in $T_x \widetilde{M}$. Observe that at $v=0$, we know 
$$\nabla_v d_h(\exp_x(v),z) \big|_{v=0}=(d\exp_x |_{v=0})^{-1} \nabla_x d_h(x,z)=\nabla_x d_h(x,z).$$ 
Hence by the Jacobian estimate (\ref{Jacobian}) and the normalization (\ref{normalization}), we obtain
\begin{eqnarray*}
|\nabla_x d_h^s(x,z)-\nabla_x d_h(x,z)|&\leqslant & |A_0-\nabla_x d_h(x,z)|+C(n,K_1,K_2)h^3 \\
&\leqslant& Ch^2+C(n,K_1,K_2)h^3.
\end{eqnarray*}
which gives the desired lower bound for $|\nabla_x d_h^s(x,z)|$ for sufficiently small $h$, due to $|\nabla_x d_h(x,z)|$ $=1$. 

\smallskip
\textbf{Case 2:}  $x\in M-M_h$ and $\widetilde{d}(x,\partial M_h)>r$.

In this case, the gradient $\nabla_x d_h(x,z)=h^{-1}\nabla_x \widetilde{d}(x,\partial M_h)$ by the definition of $d_h$ (\ref{dh}). The second derivative of $\widetilde{d}(\cdot,\partial M_h)$ is bounded by the bound $2K_1$ on the second fundamental forms of the equidistant hypersurfaces from $\partial M$ in the boundary normal neighborhood of $\partial M$ (Lemma \ref{riccati}). Hence we have
\begin{equation}\label{bdhcloseness-0}
\big|\nabla_v d_h(\exp_x(v),z)-\nabla_v d_h(\exp_x(v),z)|_{v=0}\big|\leqslant C(K_1)h^{-1}|v| \leqslant C(K_1) h^{2}.
\end{equation}
Then the same argument as Case 1 shows that
\begin{equation}\label{bdhcloseness}
|\nabla_x d_h^s(x,z)-\nabla_x d_h(x,z)|\leqslant C(K_1) h^{2}+C(n,K_1,K_2)h^3,
\end{equation}
which yields a lower bound considering $|\nabla_x d_h(x,z)|=h^{-1}$.

\medskip
The general case when $x$ is close to $\partial M_h$ requires more careful treatment. We spend the rest of the proof to address it.

\smallskip
\textbf{Case 3:} $x\in M-M_h$ with $\widetilde{d}(x,\partial M_h)\leqslant r$ or $x\in M_h$.

Since $d_h(x,z)\leqslant \min\{r_0/2,\pi/6K_1\}$ is bounded by the radius of radial uniqueness (\ref{CATchoice}), the gradient $|\nabla_x d_h(x,z)|$ equals to $1$ or $h^{-1}$ depending on whether $x$ is in $M_h$. It is known that geodesics of $M_h$ are of $C^{1,1}$ and the second derivative of a geodesic exists except at countably many switch points (switching between interior segments and boundary segments) where both one-sided second derivatives exist (e.g. Section 2 in \cite{ABB}). Furthermore, the second derivative exists and vanishes at intermittent points which are the accumulation points of switch points. It was also proved that if the endpoints of a family of geodesics converge, then the geodesics converge uniformly in $C^1$ (the first Lemma in Section 4 of \cite{ABB}). However, the estimates in the said work were done in terms of an extrinsic parameter (depending on how a manifold is embedded in the ambient space), and we show the following modification in terms of intrinsic parameters.

The manifold $M_{h}$ has curvature bounded above by $4K_1^2$ locally in the sense of Alexandrov due to the Characterization Theorem in \cite{ABB2}. Furthermore by Theorem 4.3 in \cite{AB} and (\ref{CATchoice}), for any $z\in M_h$, the ball of $M_h$ around $z$ of the radius $\min\{2r_0/3,\pi/4K_1\}$ is a metric space of curvature bounded above by $4K_1^2$ . Denote by $\gamma_{x},\gamma_{y}$ the minimizing geodesics of $M_h$ from $x,y\in M_h$ to $z$. Denote the length of $\gamma_x$ by $L_x$ (i.e. $L_x=d_h(x,z)$). The geodesics $\gamma_x,\gamma_y$ are parametrized in the arclength parameter on $[0,L_x],[0,L_y]$ respectively. Without loss of generality, assume $L_x\leqslant L_y$. Hence
$$d_h(\gamma_y(L_x),\gamma_x(L_x))=d_h(\gamma_y(L_x),\gamma_y(L_y))=L_y-L_x\leqslant d_h(x,y),$$
where we used $\gamma_x(L_x)=\gamma_y(L_y)=z$. Then Corollary 9.13 in \cite{AKP} shows that if $d_h(x,z)\leqslant \pi/6K_1$ and $d_h(x,y)$ is sufficiently small depending on $K_1$, we have 
$$\|\gamma_{x}-\gamma_{y}\|_{C^0([0,L_x])} < 2d_h(x,y),$$ 
where the $C^0$-norm is the uniform norm with respect to $d_h$. This leads to $\|\gamma_{x}-\gamma_{y}\|_{C^0([0,L_x])} < C\widetilde{d}(x,y)$ if $\widetilde{d}(x,y)$ is sufficiently small by (\ref{dhprojection}). On the other hand, due to Lemma \ref{riccati} and (\ref{acceleration}), the second derivatives of $\gamma_{x},\gamma_{y}$ are bounded by $C(n,K_1,K_2)$ whenever they exist in the boundary normal coordinate of $\partial M_h$, and both one-sided second derivatives respect the same bound at switch points.

We lift the part of the curves $\gamma_{x},\gamma_{y}$ near $x,y$ onto the tangent space $T_x \widetilde{M}$. Without loss of generality, assume all of $\gamma_x,\gamma_y$ lie in the image of $\exp_x$. Since the exponential map and its inverse are uniformly bounded up to $C^2$, the properties stated above satisfied by $\gamma_x,\gamma_y$ are also satisfied by their lifts: namely, if $\widetilde{d}(x,y)$ is sufficiently small depending on $K_1$,
$$\|\exp_x^{-1}\circ \gamma_x-\exp_x^{-1}\circ \gamma_y\|_{C^0([0,L_x])}<C\widetilde{d}(x,y);$$
and the second derivatives of $\exp_x^{-1}\circ \gamma_x,\exp_x^{-1}\circ \gamma_y$ are uniformly bounded by $C(n,K_1,K_2)$ in $L^{\infty}$-norm. Here the $C^0$-norm is the uniform norm with respect to the Euclidean distance in $T_x \widetilde{M}$. Hence Lemma \ref{closeness-curve} applies:
\begin{equation}\label{closeness-lift}
\|(\exp_x^{-1}\circ \gamma_x)^{\prime}-(\exp_x^{-1}\circ \gamma_y)^{\prime}\|_{C^0([0,L_x])}<C(n,K_1,K_2)\sqrt{\widetilde{d}(x,y)}\, .
\end{equation}
At the starting point $y=\gamma_y(0)$ of $\gamma_y$, we know $\gamma_y^{\prime}(0)=-\nabla_y d_h(y,z)$ and hence
$$(\exp_x^{-1}\circ \gamma_y)^{\prime}(0)=(d\exp_x |_{v})^{-1} \gamma_y^{\prime}(0)=-\nabla_v d_h(\exp_x(v),z), $$
where $v=\exp_x^{-1}(y)$. At the starting point $x=\gamma_x(0)$ of $\gamma_x$, we simply have $(\exp_x^{-1}\circ \gamma_x)^{\prime}(0)=-\nabla_x d_h(x,z)$ by definition. Thus for sufficiently small $h$ depending on $K_1$, if $y\in M_h$ and $\widetilde{d}(x,y)\leqslant h^3$, the estimate (\ref{closeness-lift}) at starting points gives
\begin{equation}\label{closeness}
|\nabla_v d_h(\exp_x(v),z)-\nabla_x d_h(x,z)| < C\sqrt{\widetilde{d}(x,y)} \leqslant C(n,K_1,K_2)h^{\frac{3}{2}}.
\end{equation}

The difference between this case and Case 1 is that the formula for $\nabla_x d_h^s(x,z)$ (at the beginning of the proof) may split into two parts: the integral over points in $M_h$ and over points in $M-M_h$. The key observation is that in a small neighborhood intersecting $\partial M_h$, the gradient $\nabla_x d_h(x,z)$ for $x\in M-M_h$ is essentially normal to $\partial M_h$, which has almost the same direction as the normal component (with respect to $\partial M_h$) of $\nabla_x d_h(x,z)$ for $x\in M_h$. A precise version of this observation will be shown later. The $h^{-1}$ scaling in the definition of $d_h$ (\ref{dh}) plays a crucial role in obtaining the desired lower bound.

Denote the part of the integral $A_0$ (\ref{ddA0}) over points in $M_h$ by $A_1$, and the part of $A_0$ over points in $M-M_h$ by $A_2$. We divide Case 3 into the following three situations depending on where $x$ lies.

\smallskip
\textbf{Case 3(i):} $x\in M_h$ and $\widetilde{d}(x,\partial M_h)>r$.

In this case, the integral $A_0$ only involves points in $M_h$ and $A_0=A_1$. Then the same argument as Case 1 and (\ref{closeness}) imply that
\begin{equation*}\label{dhslower}
|\nabla_x d_h^s(x,z)-\nabla_x d_h(x,z)|<C(n,K_1,K_2)h^{\frac{3}{2}}.
\end{equation*}

\textbf{Case 3(ii):} $x\in \partial M_h$.

Denote by $\textbf{n}_x\in T_x (\widetilde{M})$ the outward-pointing unit vector normal to $\partial M_h$. The estimate (\ref{closeness}) yields the closeness between normal components:
$$\big|\langle \nabla_v d_h(\exp_x(v),z),\textbf{n}_x \rangle -\langle \nabla_x d_h(x,z),\textbf{n}_x \rangle \big|< Ch^{\frac{3}{2}}, \textrm{ if }\exp_x(v)\in M_h.$$
Since clearly $\langle \nabla_x d_h(x,z),\textbf{n}_x \rangle\geqslant 0$ for $x\in \partial M_h$, we have 
\begin{equation}\label{smallnormal}
\langle \nabla_v d_h(\exp_x(v),z),\textbf{n}_x \rangle >-Ch^{\frac{3}{2}}, \textrm{ if }\exp_x(v)\in M_h,
\end{equation}
which implies that $\langle A_1,\textbf{n}_x \rangle > -Ch^{\frac{3}{2}}$.

On the other hand, we replace the evaluation at $v=0$ in the estimate (\ref{bdhcloseness-0}) with $v=\exp_x^{-1}(x^{\prime})$ for an arbitrary point $x^{\prime}\in M-M_h$ close to $x$. Then consider their normal components similarly. Since $\nabla_x d_h(x^{\prime},z)$ can be arbitrarily close to $h^{-1}\textbf{n}_x$ and the exponential map only changes the inner product by a higher order $C(K_1)r^2$ term, we have
\begin{equation}\label{largenormal}
\langle \nabla_v d_h(\exp_x(v),z),\textbf{n}_x \rangle \geqslant h^{-1}-Ch^{2}, \textrm{ if }\exp_x(v)\in M-M_h.
\end{equation}
Furthermore by (\ref{bdhcloseness-0}), the tangential component of $\nabla_v d_h(\exp_x(v),z)$ can only have length at most $Ch^{2}$ if $\exp_x(v)\in M-M_h$. This implies that $|A_2-\langle A_2,\textbf{n}_x \rangle \textbf{n}_x|<Ch^{2}$.

\smallskip
\textbf{(1)} If $c_{n}r^{-n} \int_{\{v\in \mathcal{B}_r(0): \,\exp_x(v)\in M-M_h\}}  k_1\big(\frac{|v|}{r}\big) dv \geqslant h$, then (\ref{largenormal}) yields that $\langle A_2,\textbf{n}_x \rangle \geqslant 1-Ch^{3}$.
Thus by (\ref{smallnormal}),
$$|A_0|\geqslant |\langle A_0,\textbf{n}_x \rangle| =|\langle A_1+A_2,\textbf{n}_x \rangle| > 1-Ch^{\frac{3}{2}}-Ch^{3}.$$

\textbf{(2)} If $c_{n}r^{-n} \int_{\{v\in \mathcal{B}_r(0): \,\exp_x(v)\in M-M_h\}}  k_1\big(\frac{|v|}{r}\big) dv < h$, then by (\ref{closeness}) and (\ref{normalization}), we have
\begin{eqnarray*}
|A_1| &>& \bigg| \, c_{n}r^{-n} \int_{\{v\in \mathcal{B}_r(0): \,\exp_x(v)\in M_h\}}  k_1\big(\frac{|v|}{r}\big) \big(\nabla_x d_h(x,z)\big) J_x(v) dv \, \bigg| - Ch^{\frac{3}{2}} \\
&>& 1-h-Ch^{\frac{3}{2}}. 
\end{eqnarray*}

Observe that (\ref{largenormal}) implies that $\langle A_2,\textbf{n}_x \rangle > 0$ for sufficiently small $h$. If $\langle A_1,\textbf{n}_x \rangle \geqslant 0$, then 
\begin{eqnarray*}
|A_0|=|A_1+A_2|&\geqslant& \big|A_1+\langle A_2,\textbf{n}_x \rangle \textbf{n}_x \big|-\big|A_2-\langle A_2,\textbf{n}_x \rangle \textbf{n}_x \big| \\
&>& |A_1|-Ch^{\frac{3}{2}} > 1-h-Ch^{\frac{3}{2}}-Ch^{2}.
\end{eqnarray*}
If $\langle A_1,\textbf{n}_x \rangle < 0$, then $|\langle A_1,\textbf{n}_x \rangle|<Ch^{\frac{3}{2}}$ by (\ref{smallnormal}). This shows that $\big|A_1-\langle A_1,\textbf{n}_x \rangle \textbf{n}_x \big|>1-h-Ch^{\frac{3}{2}}$. Hence we have 
\begin{eqnarray*}
|A_0|&\geqslant& \big|A_1+A_2-\langle A_1+A_2,\textbf{n}_x \rangle \textbf{n}_x \big| \\
&\geqslant& \big|A_1-\langle A_1,\textbf{n}_x \rangle \textbf{n}_x \big|- \big|A_2-\langle A_2,\textbf{n}_x \rangle \textbf{n}_x \big| \\
&>& 1-h-Ch^{\frac{3}{2}}-Ch^{2}.
\end{eqnarray*}

\smallskip
\textbf{Case 3(iii):} $x\notin \partial M_h$ and $\widetilde{d}(x,\partial M_h)\leqslant r$.

In this case, we choose an arbitrary point $x_0\in \partial M_h$ such that $\widetilde{d}(x_0,x)\leqslant r$. By the triangle inequality, (\ref{closeness}) yields that
$$\big|\nabla_v d_h(\exp_x(v),z)-\nabla_{v} d_h(\exp_x(v),z)|_{v=v_0}\big| <C(n,K_1,K_2)h^{\frac{3}{2}}, \textrm{ if }\exp_x(v)\in M_h,$$
where $v_0=\exp_{x}^{-1}(x_0)$. Then we consider the normal component with respect to $(d\exp_{x}|_{v_0})^{-1} \textbf{n}_{x_0}$ $\in T_x(\widetilde{M})$ and replace the vector $\textbf{n}_x$ in Case 3(ii) with $(d\exp_{x}\big|_{v_0})^{-1} \textbf{n}_{x_0}$. Since $\langle \nabla_x d_h(x,z)|_{x=x_0},$ $\textbf{n}_{x_0} \rangle_{x_0}\geqslant 0$ with respect the inner product of $T_{x_0} \widetilde{M}$, after lifting the vectors onto $T_x \widetilde{M}$ via the exponential map, we have
\begin{eqnarray*}
\langle (d\exp_{x} |_{v_0})^{-1} \big(\nabla_x d_h(x,z) \big|_{x=x_0}\big), (d\exp_{x} |_{v_0})^{-1} (\textbf{n}_{x_0}) \rangle_x \geqslant -C(K_1)r^2.
\end{eqnarray*}
Then the rest of the argument in Case 3(ii) applies up to a higher order term as $d\exp_x|_{v_0}$ only changes the inner product by a higher order $C(K_1)r^2$ term.

\smallskip
Finally, combining all the cases together, we obtain
\begin{eqnarray*}
|\nabla_x d_h^s(x,z)|\geqslant |A_0|-C(n,K_1,K_2)h^3 >1-h-C(n,K_1,K_2)h^{\frac{3}{2}},
\end{eqnarray*}
and therefore the lemma follows.
\end{proof}

\begin{mindistance}\label{mindistance}
For $i\geqslant 1$ and sufficiently small $h$ depending on $n,T,K_1,i_0$, we have
$$dist_{\widetilde{M}\times \mathbb{R}}(\partial \Omega_{i,j}^0,\Omega_{i,j}) > \min\{\frac{h^3}{100},\frac{h^2}{20T}\}.$$
For $i=0$, we have
$$dist_{\widetilde{M}\times \mathbb{R}}(\partial \Omega_{0,j}^0,\Omega_{0,j}) > \frac{h^3}{6T^2}.$$
\end{mindistance}
\begin{proof}
There are two types of boundaries involved. The first type is from the level sets of $d_h^s(\cdot,z_{i,j})$. For $i\geqslant 2$, the distance of the first type is from the the boundary of the cylinder $\{x: d_h^s(x,z_{i,j}) \leqslant  \frac{1}{2}\min\{1,T^{-1}\}h\}\times [-T_i,T_i]$ and the boundary of $\cup_{l=0}^{i-1}\cup_{j}\overline{\Omega}_{l,j}$. Since a larger cylinder $\{x: d_h^s(x,z_{i,j}) \leqslant \min\{1,T^{-1}\}h\}\times [-T_i-h,T_i+h]$ is also contained in $\cup_{l=0}^{i-1}\cup_{j}\overline{\Omega}_{l,j}$ due to (\ref{dinclusion}), then the distance of this type is bounded below by the distance between these two cylinders, which is bounded below by $\min\{1,T^{-1}\}h^2/20$ by Lemma \ref{dhs}(4,3) if $h<1/10$. 

For $i=1$, the distance of the first type is from the the boundary of the cylinder $\{x: d_h^s(x,z_{1,j}) \leqslant  h/2\}\times [-T_1,T_1]$ and the boundary of $\cup_{j}\Omega_{0,j}$. By (\ref{sublemma0d}) and Sublemma \ref{sublemmainitial2}, the cylinder $\{x: d_h^s(x,z_{1,j}) \leqslant  3h/4\}\times [-T_1,T_1]$ is contained in the open set $\cup_{j}\Omega_{0,j}$, and hence the distance between the boundary of the cylinder and that of $\cup_{j}\Omega_{0,j}$ is bounded away from 0. To obtain an explicit estimate, one can prove a slightly tighter estimate than Sublemma \ref{sublemmainitial2} if $T>10h$:
$$\big(\cup_{b\in [0,2h]}\Gamma_b(8h)\big)\times [-T+\frac{11}{2}h,T-\frac{11}{2}h] \subset \cup_j\Omega_{0,j}.$$
With (\ref{sublemma0d}), this shows that a larger cylinder $\{x: d_h^s(x,z_{1,j}) \leqslant  3h/4\}\times [-T_1-h/2,T_1+h/2]$ is contained in $\cup_{j}\Omega_{0,j}$. Then Lemma \ref{dhs}(4,3) yields a lower bound $h^2/40$ if $h<1/20$.

For $i\geqslant 1$, the other type of boundaries is generated by the level sets of $\psi_{i,j}$. Suppose boundary points $(x_1,t_1)$ and $(x_2,t_2)$ belong to $\{\psi_{i,j} =9T^2h\}$ and $\{\psi_{i,j} = 8T^2h\}$ respectively, and hence by the definition of $\psi_{i,j}$ we have
\begin{eqnarray*}
&&\Big(\big(1-\xi(d(x_1,\partial M))-\xi(\rho_0-d_h^s(x_1))\big)T_i-d_h^s(x_1)\Big)^2\\
&&-\Big(\big(1-\xi(d(x_2,\partial M))-\xi(\rho_0-d_h^s(x_2))\big)T_i-d_h^s(x_2)\Big)^2-t_1^2+t_2^2 = T^2 h.
\end{eqnarray*}
Then,
\begin{eqnarray*}
&&2T^2 \big|\xi(\rho_0-d_h^s(x_1))-\xi(\rho_0-d_h^s(x_2)) \big|+2T^2 \big|\xi(d(x_1,\partial M)) -\xi(d(x_2,\partial M)) \big| \\
&&+2T\big|d_h^s(x_1)-d_h^s(x_2) \big|+2T|t_1-t_2| >T^2 h.
\end{eqnarray*}
By the definition of $\xi$,
\begin{eqnarray*}
&&\frac{6T^2}{h}|d_h^s(x_1,z_{i,j})-d_h^s(x_2,z_{i,j})|+\frac{6T^2}{h}| d(x_1,\partial M)-d(x_2,\partial M)| \\
&&+2T|d_h^s(x_1,z_{i,j})-d_h^s(x_2,z_{i,j})|+2T|t_1-t_2| > T^2 h.
\end{eqnarray*}
Then it follows that at least one of the four absolute values must be larger than $h^2/24$ if $h<3T$, which implies that at least one of $|d_h(x_1,z_{i,j})-d_h(x_2,z_{i,j})|$, $|d(x_1,\partial M)-d(x_2,\partial M)|$ or $|t_1-t_2|$ is larger than $h^2/50$ by Lemma \ref{dhs}(4). Here we divided the smoothening radius by a constant to keep the error brought by the convolution relatively small. Since $d(x,\partial M)=\widetilde{d}(x,\partial M)$ for $x\in M$, Lemma \ref{dhs}(3) yields that at least one of $\widetilde{d}(x_1,x_2)$ or $|t_1-t_2|$ is larger than $h^3/100$ and hence the lemma follows.

Finally for the initial step $i=0$, the first type of boundary distance is from $\{\rho(x)=-3h/2\}$ and the boundary of $\Upsilon$, which is clearly bounded below by $h/2$. The second type of boundary distance is between level sets of $\psi_{0,j}$. One can follow the same argument as for $i\geqslant 1$ for this type of boundary distance, and obtain a lower bound $h^3/6T^2$.
\end{proof}

\begin{geodiff}\label{geodiff}
Suppose $\gamma(s)$ is a geodesic of $M$ satisfying $\gamma(0)\in \partial M$ and the initial vector $\gamma^{\prime}(0)\in T_{\gamma(0)}\partial M$. Then there exists a constant $\epsilon_0$ explicitly depending on $n,\|R_M\|_{C^1},\|S\|_{C^1},i_0$ such that for any $s\leqslant \epsilon_0$, we have $d(\gamma(s),\partial M) \leqslant C(n,\|R_M\|_{C^1},\|S\|_{C^1})s^2$.
\end{geodiff}

\begin{proof}
Without loss of generality, assume the geodesic $\gamma(s)$ lies entirely in the interior of $M$ except for the initial point. 
Consider another geodesic of $\partial M$ with the same initial point $\gamma(0)$ and the same initial vector $\gamma^{\prime}(0)$. We claim that the distance between this geodesic of $\partial M$ and $\gamma(s)$ is bounded above by $Cs^2$ for sufficiently small $s$. Clearly this claim yields the lemma.

Denote the geodesics of $M,\partial M$ in question with the arclength parametrization by $\gamma_1,\gamma_2$. Take $\epsilon_0<i_0$ and we consider the geodesics $\gamma_i(s)$ ($i=1,2$) in a $C^1$ boundary normal coordinate $(x^1,\cdots,x^n)$. Due to Lemma 8 in \cite{HV} and Lemma \ref{riccati}, within a uniform radius explicitly depending on $n,\|R_M\|_{C^1},\|S\|_{C^1},i_0$, the $C^1$-norm of metric components is uniformly bounded by a constant explicitly depending on $n,\|R_M\|_{C^1},\|S\|_{C^1}$. Since $\gamma_1,\gamma_2$ have the same initial point and the same initial vector, we know $\gamma_1^j(0)=\gamma_2^j(0)$ and $\partial_s \gamma_1^j(0)=\partial_s \gamma_2^j(0)$ for all $j=1,\cdots,n$, where $\gamma_i^j$ denotes the $j$-th component of $\gamma_i$ with respect to the coordinate $x^j$. The fact that $|\partial_s \gamma_1(s)|_{M}=|\partial_s \gamma_2(s)|_{\partial M}=1$ yields $|\partial_s \gamma_i^j(s)|\leqslant C$ for any $j$ due to the $C^0$ metric bound in bilinear form. Moreover, the geodesic equation in local coordinates has the following form:
$$\partial_s^2 \gamma^j+\sum_{k,l}\Gamma_{kl}^j (\partial_s \gamma^k)( \partial_s \gamma^l)=0,$$
and $\gamma_1,\gamma_2$ satisfy this equation with $\Gamma_{kl}^j$ of $M,\partial M$ respectively. Hence by applying the $C^1$ bound for metric components, we have an estimate for the second derivative:
\begin{equation}\label{acceleration}
|\partial_s^2 \gamma_i^j(s)| \leqslant C(n,\|R_M\|_{C^1},\|S\|_{C^1}), \; \textrm{ for all }j=1,\cdots,n.
\end{equation}
Since $\gamma_1,\gamma_2$ lie entirely in $\textrm{int}(M),\partial M$ by assumption, they are at least of $C^2$ and hence
\begin{eqnarray*}
|\gamma_1^j(s)-\gamma_2^j(s)|\leqslant\frac{s^2}{2}\sup_{s^{\prime}\in (0,s)} \big|\partial_s^2 \gamma_1^j(s^{\prime})-\partial_s^2 \gamma_2^j(s^{\prime}) \big| \leqslant C(n,\|R_M\|_{C^1},\|S\|_{C^1})s^2.
\end{eqnarray*}
This implies $d(\gamma_1(s),\gamma_2(s))\leqslant C(n,\|R_M\|_{C^1},\|S\|_{C^1})s^2$ due to the $C^0$ metric bound.
\end{proof}

\begin{areaLipschitz}\label{areaLipschitz}
Denote $A_t(\epsilon)=\{x\in \Sigma_t: l(x)>\epsilon\}$ and by $U(A_t(\epsilon))$ the set of all points on all minimizing geodesics from $A_t(\epsilon)$ to $\Gamma$. Then for sufficiently small $\epsilon$ explicitly depending on $K_1$ and any $t^{\prime}\in [t-\epsilon/2,t)$, we have 
$$\vol_{n-1}(A_t(\epsilon))<5^{n-1} \vol_{n-1}\big(U(A_t(\epsilon))\cap \Sigma_{t^{\prime}}\big).$$
\end{areaLipschitz}
\begin{proof}
We define a function $F: U(A_t(\epsilon))\cap \Sigma_{t^{\prime}}\to A_t(\epsilon)$ by mapping a point $x\in U(A_t(\epsilon))\cap \Sigma_{t^{\prime}}$ to the initial point of the particular minimizing geodesic containing $x$ from $A_t(\epsilon)$ to $\Gamma$. This function is well-defined since minimizing geodesics cannot intersect at $\Sigma_{t^{\prime}}$; otherwise they would fail to minimize length past an intersection point. To show the measure estimate in question, it suffices to show that $F$ is locally Lipschitz with a Lipschitz constant $5$ for sufficiently small $\epsilon$ depending on $K_1$. Since the measure in question is an $(n-1)$-dimensional Hausdorff measure, the Lipschitz continuity of $F$ implies the measure estimate with the constant $5^{n-1}$ (Section 5.5.2 in \cite{BBI}).

Here we show that the function $F$ is locally Lipschitz. For any point $y_0 \in U(A_t(\epsilon))\cap\{x: t-\epsilon/2\leqslant d(x,\Gamma)\leqslant t\}$, there exists $x_0\in U(A_t(\epsilon))\cap \Sigma_{t-\epsilon}$ such that $x_0$ lies on a minimizing geodsic from $y_0$ to $\Gamma$, which indicates $d(y_0,\Gamma)=d(y_0,x_0)+d(x_0,\Gamma)$. Observe that the geodesic segment from $y_0$ to $x_0$ does not intersect the boundary. Then there exists a small neighborhood of $y_0$, such that for any $y$ in this neighborhood, the minimizing geodesic from $x_0$ to $y$ does not intersect the boundary. Thus the distance function $d(\cdot,x_0)$ in the small neighborhood of $y_0$ is just a geodesic distance function with the second derivative bounded by $3/\epsilon$ for sufficiently small $\epsilon$ depending on $K_1$ (e.g. Theorem 27 in \cite{PP}, p175). Hence we have
\begin{eqnarray*}
d(y,\Gamma)&\leqslant& d(x_0,\Gamma)+d(y,x_0) \\
&=& d(y_0,\Gamma)-d(y_0,x_0)+d(y,x_0)\\
&\leqslant& d(y_0,\Gamma)+\nabla_y d(y_0,x_0)\cdot \exp^{-1}_{y_0} (y) + \frac{3}{2\epsilon} d(y,y_0)^2+o\big(d(y,y_0)^2\big).
\end{eqnarray*}
This shows that the distance function $d(\cdot,\Gamma)$ is a semi-concave function in $U(A_t(\epsilon))\cap\{x: t-\epsilon/2\leqslant d(x,\Gamma)\leqslant t\}$ for sufficiently small $\epsilon$ with the semi-concavity constant $3/\epsilon$. Now consider the gradient flow by the distance function $d(\cdot,\Gamma)$, and the function $F$ is simply the gradient flow restricted to this region $U(A_t(\epsilon))\cap\{x: t^{\prime}\leqslant d(x,\Gamma)\leqslant t\}$ for $t^{\prime}\in [t-\epsilon/2,t)$. By Lemma 2.1.4(i) in \cite{P}, the restricted gradient flow (or F) is locally Lipschitz with a Lipschitz constant $e^{3/2}<5$.
\end{proof}

\appendix

\section{Dependency of constants}\label{constants}

In this section, we show explicitly how the constant in Theorem \ref{stability} depends on geometric parameters. We first show the dependency of constants in Theorem \ref{main1}, and then trace the dependency through the proofs in Section \ref{section-projection}, \ref{section-appro}.

\medskip
For $i\geqslant 1$, the lower bounds (\ref{dpsi}) and (\ref{ppsi}) hold:
$$\min_{(x,t)\in \Omega_{i,j}^0}|\nabla_x\psi_{i,j}|>2T\sqrt{h}; \qquad \min_{(x,t)\in \Omega_{i,j}^0}p\big((x,t),\nabla \psi_{i,j} \big) > 8T^2 h.$$
From the definition (\ref{dhsdef}), Lemma \ref{extensionmetric}(3), (\ref{Jacobian}), (\ref{dC2}) and (\ref{dC21}) that for sufficiently small $h$ depending on $n,K_1,K_2,i_0$, we have
$$\|d_h^s(x,z_{i,j})\|_{C^0(\Omega_{i,j}^0)} < \rho_0<i_0; \quad \|\nabla_x d_h^s(x,z_{i,j})\|_{C^0(\Omega_{i,j}^0)}< 2h^{-1};$$
$$\|\nabla_x^2 d_h^s(x,z_{i,j})\|_{C^0(\Omega_{i,j}^0)}< C(n,i_0)h^{-6}; $$
$$\|\nabla_x^2 d_h^s(x,z_{i,j})\|_{Lip(\Omega_{i,j}^0)}< C(n,\|R_M\|_{C^1},i_0)h^{-9}.$$
On the other hand, the $C^{2,1}$-norm of $d(\cdot,\partial M)$ is bounded by $2\|S\|_{C^1}$ for sufficiently small $h$. Therefore by definition of $\psi_{i,j}$, for sufficiently small $h$ depending on $n,K_1,K_2,i_0$,
$$\|\psi_{i,j}\|_{C^{0}(\Omega_{i,j}^0)} < T^2; \quad \|\psi_{i,j}\|_{C^{1}(\Omega_{i,j}^0)} <  C(T)h^{-2};$$
$$\|\psi_{i,j}\|_{C^{2}(\Omega_{i,j}^0)} <  C(n,T,i_0)h^{-7}; \quad \|\psi_{i,j}\|_{C^{2,1}(\Omega_{i,j}^0)} <  C(n,T,\|R_M\|_{C^1},i_0)h^{-10}.$$

For $i=0$, we have
$$\min_{(x,t)\in \Omega_{0,j}^0}|\nabla_x\psi_{0,j}|>2h; \qquad \min_{(x,t)\in \Omega_{0,j}^0}p \big((x,t),\nabla \psi_{0,j} \big) > 4h^2,$$
and the bounds above for $\psi_{i,j}$ also hold for $\psi_{0,j}$.

Now we calculate the parameters in the table (4.3) in \cite{BKL2} for our case. The following notations (until (\ref{Ch})) were used in the paper \cite{BKL2} and are not used in our present paper. We write them here only for the convenience of the readers.
$$M_1\sim M_2\sim h^{-15},\; \lambda \sim h^{-15},\; R_1\sim h^{17},\; \epsilon_0\sim h^{34},$$
$$R_2\sim R \sim h^{51},\; r\sim h^{135},\; \delta\sim h^{138},$$
$$N=C \big(n,T,\|R_M\|_{C^1},i_0,\|g^{ij}\|_{C^1},\vol(M),\vol_{n-1}(\Gamma) \big)h^{-135(n+1)}.$$
$$c_{161}\sim c_{158}\sim c_{155,N}+c_{156}\sim c_{155,N}+c_{156}^{-1/(1-\alpha)}, \; \alpha^N=\frac{1}{2}.$$
The quotient $c_{155,j}/c_{155,j-1}$ is polynomial large in $h$, and $c_{156}\sim c_{106}/c_{131}$ is also polynomial large, where the exponents are explicit multiples of $n$. Therefore $c_{161}$ has at most exponential growth with an explicit exponent $C_4 n$ for some absolute constant $C_4$. Then we turn to the constant in our result:
\begin{eqnarray}\label{Ch}
C(h)&\sim& C \big(n,T,\|R_M\|_{C^1},\|S\|_{C^1},i_0,\vol(M),\vol_{n-1}(\Gamma) \big) h^{-C(n)h^{-C(n)}} \nonumber \\
&<& C_3 \big(n,T,\|R_M\|_{C^1},\|S\|_{C^1},i_0,\vol(M),\vol_{n-1}(\Gamma) \big) \exp(h^{-C_4 n}),
\end{eqnarray}
where we have used the fact that the $C^1$-norm of metric components is bounded by a constant depending on $n,\|R_M\|_{C^1},\|S\|_{C^1}$. The dependency on the diameter $D,r_0$ is introduced after replacing $h$ by $h/C^{\prime}$ during the last part of the proof of Theorem \ref{main1}.

\medskip
From here, we come back to the notations of our present paper. Next we show the dependency of $C_1$ and $C_2$ in Theorem \ref{stability}. The final parameter is $\eta$ in Proposition \ref{approximation} and we start from $\eta$ to work out the parameters $J,\delta$. The criteria for determining parameters are already described during the proofs of relevant lemmas and propositions in Section \ref{section-projection} and \ref{section-appro}. Let $\eta\in(0,1)$ be the parameter in Proposition \ref{approximation}. Then,
$$\Lambda=1,\quad N=C(n,\vol(\partial M))\eta^{-n+1};$$
$$\varepsilon(volume,M_{\beta}^{\ast})=\varepsilon_{\ast}=C\eta^n \; \textrm{(Proposition \ref{approximation})};$$
(Lemma \ref{volume} and Lemma \ref{volumebeta})
$$\varepsilon=\varepsilon(projection)=\frac{\varepsilon(volume,M_{\alpha})}{2\vol(M)}=\frac{\varepsilon(volume,M_{\beta}^{\ast})}{2^{L+1}2\vol(M)}=C(\vol(M),L)\eta^n.$$
The following three parameters are determined by (\ref{ualasterror}) in Proposition \ref{measureerror}:
\begin{equation}\label{Bgamma}
8C_5^2\Lambda^2\gamma^{\frac{1}{n+1}}=\frac{\varepsilon^2}{4} \Longrightarrow \gamma=\big(\frac{\varepsilon^2}{32C_5^2\Lambda^2}\big)^{n+1}=\big(\frac{\varepsilon^2}{32C_5^2}\big)^{n+1};
\end{equation}
$$C(\Lambda)\lambda_{J}^{-\frac{1}{2}}\gamma^{-2}\leqslant \frac{\varepsilon^2}{4} \Longrightarrow \lambda_{J}\geqslant 16C^2\gamma^{-4}\varepsilon^{-4};$$
$$8N\varepsilon_2(0)+8N^2\varepsilon_2^2(0)=\frac{\varepsilon^2}{4}<1 \Longrightarrow \varepsilon_2(0)=\frac{\varepsilon^2}{64N}.$$
By the formula for $\varepsilon_2(0)$ in (\ref{epsilon20}),
\begin{eqnarray}\label{newepsilon2}
\varepsilon_2(0)&=&C_3^{\frac{1}{3}}h^{-\frac{2}{9}}\exp(h^{-C_4 n}) \frac{\Lambda\gamma^{-3}+h^{-\frac{1}{2}}\varepsilon_1}{\big(\log (1+h^{\frac{3}{2}}\gamma^{-3}\frac{\Lambda}{\varepsilon_1})\big) ^{\frac{1}{6}}}+C_5\Lambda\gamma^{-3} h^{\frac{1}{3n+3}} \nonumber \\
&<& C_3^{\frac{1}{3}}\exp(h^{-C_4 n}) \frac{\gamma^{-3}h^{-1}}{\big(\log (1+h^{\frac{3}{2}}\gamma^{-3}\frac{1}{\varepsilon_1})\big) ^{\frac{1}{6}}}+C_5\gamma^{-3} h^{\frac{1}{3n+3}}.
\end{eqnarray}
We choose $h$ such that the second term in (\ref{newepsilon2}) equals to $\varepsilon_2(0)/2=\varepsilon^2/128N$:
\begin{equation}\label{Bh}
C_5\gamma^{-3}h^{\frac{1}{3n+3}}=\frac{\varepsilon_2(0)}{2} \Longrightarrow h=\big(\frac{\varepsilon_2(0)\gamma^3}{2C_5}\big)^{3n+3}=\big(\frac{\varepsilon^{6n+8}}{8^{3n+3}128NC_5^{6n+7}}\big)^{3n+3}.
\end{equation}
Then the first term of (\ref{newepsilon2}) being $\varepsilon_2(0)/2=\varepsilon^2/128N$ yields that
\begin{equation}\label{Bepsilon1}
\varepsilon_1=h^{\frac{3}{2}}\gamma^{-3}\exp\bigg(-\frac{\gamma^{-18}h^{-6} 128^6 N^6 C_3^{2}\exp(6h^{-C_4 n})}{\varepsilon^{12}}\bigg),
\end{equation}
which indicates the choice of $J$ by Lemma \ref{smallinitial}:
\begin{eqnarray}\label{BlambdaJ}
\lambda_{J}&\geqslant& C(D,\Lambda)\gamma^{-24}\varepsilon_1^{-8} \nonumber \\
&=&C(D)h^{-12}\exp\bigg(\frac{8 \gamma^{-18}h^{-6} 128^6 N^6 C_3^{2}\exp(6h^{-C_4 n})}{\varepsilon^{12}}\bigg).
\end{eqnarray}
For the choice of $\delta$, choose $N\varepsilon_2(\delta)+N^2\varepsilon_2^2(\delta)-N\varepsilon_2(0)-N^2\varepsilon_2^2(0)<\frac{\varepsilon^2}{32}$, or simply choose $N\varepsilon_2(\delta)-N\varepsilon_2(0)<\frac{\varepsilon^2}{64}$. By differentiating (\ref{newepsilon2}) with respect to $\varepsilon_1$,
$$\frac{C_3^{\frac{1}{3}}\exp(h^{-C_4 n})\gamma^{-6}\sqrt{h}}{\big(\log (1+h^{\frac{3}{2}}\gamma^{-3}\frac{1}{\varepsilon_1})\big)^2(\varepsilon_1^2+h^{\frac{3}{2}}\gamma^{-3}\varepsilon_1)}2C_0^{\prime}J\lambda_{J}^{\frac{3}{2}}\delta<\frac{\varepsilon^2}{64N}.$$
Hence it suffices to choose $\delta$ satisfying
\begin{equation}\label{delta}
\frac{C_3^{\frac{1}{3}}\exp(h^{-C_4 n})\gamma^{-3}}{h\varepsilon_1}C_0^{\prime}J\lambda_{J}^{\frac{3}{2}}\delta<\frac{\varepsilon^2}{128N}.
\end{equation}

From now on, we absorb polynomial terms into exponential terms and denote by $\sim$ if two quantities differ by a factor of some constant in the exponent. Inserting the choice of $\gamma$ (\ref{Bgamma}) to $\varepsilon_1$ (\ref{Bepsilon1}) and $\lambda_{J}$ (\ref{BlambdaJ}), we get
$$\varepsilon_1\sim \exp\big(-C_5^{36n+36}C_3^{2}\exp(h^{-C_4 n})\varepsilon^{-36n-48}N^6\big); $$
$$\lambda_{J}\sim C(D)\exp\big(C_5^{36n+36}C_3^{2}\exp(h^{-C_4 n})\varepsilon^{-36n-48}N^6\big).$$
By Weyl's asymptotic formula for eigenvalues: $\lambda_j\sim C(n,\vol(M))j^{2/n}$, we know
$$J\sim C(n,\vol(M))\lambda_{J}^{\frac{n}{2}},$$
and hence by (\ref{delta}), we have
\begin{eqnarray*}
\delta&\sim& \frac{1}{C_0^{\prime} C_3^{\frac{1}{3}}}\frac{\varepsilon^{6n+8} \exp(-h^{-C_4 n})\varepsilon_1}{C_5^{6n+6}NJ\lambda_{J}^{\frac{3}{2}}} \\
&\sim& C(D,\vol(\partial M)) C_3^{-\frac{1}{3}} \frac{\exp\big(- C_5^{36n+36}C_3^{2}\exp(h^{-C_4 n})\varepsilon^{-36n-48}N^6\big)}{J} \\
&\sim& C(n,D,\vol(M),\vol(\partial M))\exp\big(-n C_5^{36n+36}C_3^{2}\exp(h^{-C_4 n})\varepsilon^{-36n-48}N^6\big).
\end{eqnarray*}
The terms we need to estimate are $\exp(h^{-C_4 n})\varepsilon^{-36n-48}N^6$. By the choice of $h$ (\ref{Bh}), we get
$$\exp(h^{-C_4 n})\sim \exp\bigg(\Big(\frac{\varepsilon^{6n+8}}{NC_5^{6n+7}}\Big)^{-C_4n(3n+3)}\bigg),$$
which absorbs $\varepsilon^{-36n-48}N^6$. Then from
$$\varepsilon= C(\vol(M),L)\eta^n,\quad N\sim \eta^{-n+1},$$
it follows that
\begin{eqnarray*}
\delta &\sim& C(n,D,\vol(M),\vol(\partial M))\exp\Big(-C(C_3,C_4,C_5) \exp\big(C\varepsilon^{-C_4 n(3n+3)(6n+9)}\big)\Big) \\
&\sim& C(n,D,\vol(M),\vol(\partial M))\exp\Big(-C(C_3,C_4,C_5)\exp\big(C(L)\eta^{-C_2^{\prime}(n)}\big)\Big) \\
&\sim& \exp\Big(-\exp \big( C_1^{\prime}\eta^{-C_2^{\prime}} \big) \Big),
\end{eqnarray*}
where $C_1^{\prime}=C_1^{\prime} \big(n,D,\vol(M),\vol(\partial M),C_3,C_4,C_5,L \big)$ and $C_2^{\prime}=C_2^{\prime}(n)>1$. The dependency of $C_3,C_4,C_5$ is stated in Proposition \ref{wholedomain}, and the dependency of $L$ is stated in Lemma \ref{coordinate}. Therefore we obtain
$$\eta \sim  (C_1^{\prime})^{\frac{1}{C_2^{\prime}}}\Big( \log \big(|\log\delta| \big) \Big)^{-\frac{1}{C_2^{\prime}}},$$
and the dependency of constants in Theorem \ref{stability} follows from Proposition \ref{approximation} and Theorem \ref{2007}. More precisely, the constant $C_1=C_1(C_1^{\prime},C_2^{\prime},C_6,C_7)$ explicitly depends only on $n$,$D$,$\|R_M\|_{C^1}$,$\|S\|_{C^1}$,$i_0$,$r_0$,$\vol(M)$,$\vol(\partial M)$,$L$,$C_6$,$C_7$,
and the constant $C_2=C_2(C_2^{\prime})$ explicitly depends only on $n$. Note that the dependency of $L,C_6,C_7$ is not explicit.
The choice of the parameter $\delta$ depends on all present parameters including all curvature bounds assumed for $K_2$, and the choice of small $\eta$ in Theorem \ref{2007}.

We remark that one can obtain an explicit estimate without using the parameter $L$. To do this, one can use all $N$ number of $\Gamma_i$ to slice the manifold, and evaluate an approximate volume for $M_{\beta}^{\ast}$ similar to Lemma \ref{volumebeta}. The error of the approximate volume would be $2^N \varepsilon$, and the parameter $\varepsilon$ would be $\varepsilon=C2^{-N}\eta^n$.
In addition, the constant $C_6$ can be replaced by an absolute constant. However, the number $2^N$ grows exponentially in $\eta$.
This process results in an explicit estimate with three logarithms, and the constants explicitly depend only on $n$,$D$,$\|R_M\|_{C^1}$,$\|S\|_{C^1}$,$i_0$,$r_0$,$\vol(M)$,$\vol(\partial M)$.

\end{document}